\pdfoutput=1
\documentclass[hidelinks, reqno, 11 pt]{amsart}

\usepackage[numbers, square]{natbib}
\setcitestyle{citesep={,}}

\usepackage[english,french]{babel}
\usepackage{graphicx}
\usepackage{subcaption}
\usepackage{tabularx}
\usepackage{booktabs}
\usepackage{array}
\usepackage{amsmath}
\usepackage{amsfonts}
\usepackage{amssymb}
\usepackage{amsthm}
\usepackage{bbm}
\usepackage{float}
\usepackage[scr]{rsfso}
\usepackage{enumitem}
\usepackage{permute}
\usepackage{slashed}
\usepackage[usenames,dvipsnames]{xcolor}
\usepackage[pagebackref = true, colorlinks, linkcolor = Red, citecolor = Green, bookmarksdepth=2, linktocpage=true]{hyperref}
\usepackage[capitalize]{cleveref}
\usepackage{verbatim}

\usepackage[T1]{fontenc}
\usepackage{makecell}

\usepackage{tikz}
\usetikzlibrary{calc}
\usetikzlibrary {arrows.meta}

\allowdisplaybreaks

\DeclareMathOperator{\interior}{int}

\DeclareMathOperator{\Stab}{Stab}

\DeclareMathOperator{\supp}{supp}

\DeclareMathOperator{\SO}{SO}

\DeclareMathOperator{\PSL}{PSL}

\DeclareMathOperator{\Ad}{Ad}

\newcommand{\C}{\mathbb{C}}
\newcommand{\N}{\mathbb{N}}

\newcommand{\R}{\mathbb{R}}
\newcommand{\Z}{\mathbb{Z}}
\newcommand{\Ha}{\mathbb{H}}

\newcommand{\LieG}{\mathfrak{g}}
\newcommand{\LieA}{\mathfrak{a}}
\newcommand{\LieN}{\mathfrak{n}}
\newcommand{\LieK}{\mathfrak{k}}

\newcommand{\LieP}{\mathfrak{p}}

\newcommand{\Fboundary}{\mathcal{F}}

\newcommand{\involution}{\mathsf{i}}
\newcommand{\growthindicator}{\psi_\Gamma}

\newcommand{\limitset}{\Lambda}

\newcommand{\limitcone}{\mathcal{L}}
\newcommand{\BMS}{\mathsf{m}_\psi}

\newcommand{\BR}{m^\mathrm{BR}_\psi}
\newcommand{\BRstar}{m^{\mathrm{BR}_\star}_\psi}
\newcommand{\rank}{\mathsf{r}}
\newcommand{\rankG}{\mathrm{rank}(G)}
\newcommand{\primGamma}{\Gamma_{\mathrm{prim}}}

\DeclareFontFamily{U}{mathb}{\hyphenchar\font45}
\DeclareFontShape{U}{mathb}{m}{n}{
	<5> <6> <7> <8> <9> <10> gen * mathb
	<10.95> mathb10 <12> <14.4> <17.28> <20.74> <24.88> mathb12
}{}
\DeclareSymbolFont{mathb}{U}{mathb}{m}{n}
\DeclareMathSymbol{\bigast}{2}{mathb}{"06}

\def\XXint#1#2#3{{\setbox0=\hbox{$#1{#2#3}{\int}$}
		\vcenter{\hbox{$#2#3$}}\kern-.5\wd0}}

\theoremstyle{plain}
\newtheorem{theorem}[equation]{Theorem}
\newtheorem*{theorem*}{Theorem}
\newtheorem{proposition}[equation]{Proposition}
\newtheorem{lemma}[equation]{Lemma}
\newtheorem{corollary}[equation]{Corollary}

\theoremstyle{definition}
\newtheorem{definition}[equation]{Definition}

\newtheoremstyle{remark}
{}   % ABOVESPACE, \topsep for no space
{}   % BELOWSPACE, \topsep for no space
{\normalfont}  % BODYFONT
{}       % INDENT (empty value is the same as 0pt)
{\itshape} % HEADFONT
{.}         % HEADPUNCT
{5pt plus 1pt minus 1pt} % HEADSPACE
{}          % CUSTOM-HEAD-SPEC

\theoremstyle{remark}
\newtheorem{remark}[equation]{Remark}

\setlist[enumerate,1]{ref=(\arabic*)}
\setlist[enumerate,2]{ref=(\theenumi)(\alph*)}
\setlist[enumerate,3]{ref=(\theenumi)(\theenumii)(\roman*)}
\setlist[enumerate,4]{ref=(\theenumi)(\theenumii)(\theenumiii)(\Alph*)}

\newlist{alternative}{enumerate}{4}     % this creates a dedicated counter named 'subtaski'
\setlist[alternative,1]{label=(\arabic*), ref=(\arabic*)}
\setlist[alternative,2]{label=(\alph*), ref=(\thealternativei)(\alph*)}
\setlist[alternative,3]{label=(\roman*), ref=(\thealternativei)(\thealternativeii)(\roman*)}
\setlist[alternative,4]{label=(\Alph*), ref=(\thealternativei)(\thealternativeii)(\thealternativeiii)(\Alph*)}

\Crefname{enumi}{Property}{Properties}
\Crefname{alternativei}{Alternative}{Alternatives}
\Crefname{subsection}{Subsection}{Subsections}

\numberwithin{equation}{section}

\begin{document}
	\selectlanguage{english}
	
	%\setlength{\parindent}{0 in}
	%\setlength{\mathindent}{0.5 in}
	%\numberwithin{equation}{section}
	%\renewcommand*{\abstractname}{Introduction}
	
	\title[Joint Equidistribution of Cylinders and Holonomies]{Joint Equidistribution of Maximal Flat Cylinders and Holonomies for Anosov Homogeneous Spaces}
	
	\author{Michael Chow}
	\address{Department of Mathematics, Yale University, New Haven, Connecticut 06511}
	\email{mikey.chow@yale.edu}
	
	\author{Elijah Fromm}
	\address{Department of Mathematics, Yale University, New Haven, Connecticut 06511}
	\email{elijah.fromm@yale.edu}
	
	\date{}
	\begin{abstract}
		Let $G$ be a connected semisimple real algebraic group and $P<G$ be a minimal parabolic subgroup with Langlands decomposition $P=MAN$. Let $\Gamma < G$ be a Zariski dense \emph{Anosov} subgroup with respect to $P$. Since $\Gamma$ is Anosov, the set of conjugacy classes of primitive elements of $\Gamma$ is in one-to-one correspondence with the set of (positively oriented) maximal flat cylinders in $\Gamma\backslash G/M$. We describe the joint equidistribution of maximal flat cylinders and their holonomies as their circumferences tend to infinity. This result can be viewed as the Anosov analogue of the joint equidistribution result of closed geodesics and holonomies in rank one by Margulis--Mohammadi--Oh \cite{MMO14}. 
	\end{abstract}

	\maketitle
	
	\selectlanguage{english}
	
	\setcounter{tocdepth}{1}
	\tableofcontents

	\section{Introduction}\label{sec:introduction}
	\subsection{Background and setup}
	
	Let $G$ be a connected semisimple real algebraic group. Let $P < G$ be a minimal parabolic subgroup with Langlands decomposition $P=MAN$ where $N$ is the unipotent radical of $P$, $A = \exp\LieA$ is a maximal real split torus and $M$ is a maximal compact subgroup of $P$ commuting with $A$. Let $\Gamma <$ G be a torsion-free discrete subgroup. In this paper, we study the equidistribution of nontrivial closed $A$-orbits in $\Gamma \backslash G/M$, or equivalently, nontrivial closed $AM$-orbits in $\Gamma \backslash G$, and their holonomies when $\Gamma$ is a torsion-free Zariski dense Anosov subgroup with respect to $P$.  
	
	Let $\Fboundary:=G/P$ denote the Furstenberg boundary and $\Fboundary^{(2)}$ denote the unique open $G$-orbit in $\Fboundary \times \Fboundary$. 
	A Zariski dense discrete subgroup $\Gamma < G$ is \emph{Anosov} with respect to $P$ if $\Gamma$ is a finitely generated Gromov hyperbolic group and admits a $\Gamma$-equivariant continuous embedding from the Gromov boundary $\partial \Gamma$ of $\Gamma$ to $\Fboundary$ such that if $x,y \in \Fboundary$ are the images of two distinct points in $\partial \Gamma$, then $(x,y) \in \Fboundary^{(2)}$. The notion of Anosov subgroups (with respect to any parabolic subgroup of $G$) was first introduced by Labourie \cite{Lab06} for surface groups and later generalized by Guichard--Wienhard \cite{GW12} to Gromov hyperbolic groups (cf. \cite{KLP17,GGKW17,Wie18}). Throughout the paper, all Anosov subgroups are Anosov with respect to $P$.  Anosov subgroups are regarded as natural higher rank generalizations of Zariski dense convex-cocompact subgroups since the two notions coincide when $\rankG=1$. 
	
	Let $\Gamma< G$ be a torsion-free Zariski dense Anosov subgroup. Let $\LieA^+ \subset \LieA$ be the positive Weyl chamber associated to $N$ and let $A^+ = \exp \LieA^+$. Let $\growthindicator: \LieA^+ \to \{-\infty\} \cup [0,\infty)$ denote the growth indicator function of $\Gamma$ introduced by Quint \cite{Qui02a} (\cref{def:GrowthIndicatorFunction}). The function $\growthindicator$ is a higher rank generalization of the critical exponent in rank one. We fix a \emph{tangent} form $\psi \in \LieA^*$, that is,
	\begin{equation}
		\label{eqn:IntroPsi}
		\psi:\LieA \to \R \text{ linear, } \psi \ge \growthindicator, \psi(\mathsf{v}) = \growthindicator(\mathsf{v}) = 1 \text{ for some } \mathsf{v} \in \interior \LieA^+,
	\end{equation} 
	where $\interior\LieA^+$ denotes the interior of $\LieA^+$.
	The space of such tangent forms is homeomorphic to $\R^{\rankG-1}$. Up to multiplicative constant, there exists a unique $(\Gamma,\psi)$-conformal measure on $\Fboundary$ and it is necessarily supported on the limit set $\limitset$ of $\Gamma$ as shown by Lee--Oh \cite[Theorem 1.3]{LO20b}, \cite[Theorem 1.2]{LO22}. Hence, up to multiplicative constant, there exists a unique Bowen--Margulis--Sullivan (BMS) measure 
	\[\BMS\]
	on $\Gamma \backslash G/M$ associated to $\psi$ (\cref{def:BMSMeasures}). When $\rankG \ge 2$, $\BMS$ is an infinite measure (see \cite[Theorem 3.5]{Sam15}, see also \cite[Corollary 4.9]{LO20b}). Using the Hopf parametrization $\Fboundary^{(2)} \times \LieA \cong G/M$ (\cref{def:HopfParametrization}), the support of $\BMS$ is 
	\[\Omega:= \supp\BMS= \Gamma \backslash (\limitset^{(2)} \times \LieA) \subset \Gamma \backslash (\Fboundary^{(2)} \times \LieA) \cong \Gamma \backslash G/M,\]
	where $\limitset^{(2)}:= (\limitset \times \limitset) \cap \Fboundary^{(2)}$ (\cref{def:BMSMeasures}).
	
	For any closed $AM$-orbit $C=\Gamma g AM \subset \Gamma \backslash G/M$, its group of periods $g^{-1}\Gamma g\cap AM$ is either trivial or isomorphic to $\Z$, and in the latter case, $C$ is a \emph{maximal flat cylinder};
	\[C\cong ( g^{-1}\Gamma g \cap AM)\backslash AM/M \cong \R^{\rankG-1} \times \mathbb{S}^1.\]
	
	We emphasize that the closed $AM$-orbits being cylinders is a feature of Anosov subgroups. For instance, if $\Delta <G$ is a torsion-free lattice, then there exists a closed $AM$-orbit in $\Delta\backslash G$ whose group of periods is isomorphic to $\Z^{\rankG}$ and hence the closed $AM$-orbit is a compact torus \cite[Theorem 2.8]{PR72}. 
	
	When $g^{-1}\Gamma g\cap (\interior A^+)M \ne \emptyset$, where $\interior A^+$ denotes the interior of $A^+$, we call $\Gamma gAM \subset \Gamma\backslash G/M$ a \emph{positively oriented} maximal flat cylinder. Let $[\Gamma]$ denote the set of conjugacy classes in $\Gamma$ and let
	\[[\primGamma] :=\{[\gamma]\in[\Gamma]: \gamma \in \Gamma \text{ primitive}\}.\] 
	The set of all positively oriented maximal flat cylinders is in one-to-one correspondence with $[\primGamma]$. We also note that positively oriented maximal flat cylinders are precisely the maximal flat cylinders contained in $\Omega$ (\cref{lem:ClosedAOrbitsAreMaximalFlatCylinders}). In the rest of the introduction, all maximal flat cylinders are positively oriented unless stated otherwise. Denote by $\mathcal{C}_\Gamma$ the set of all maximal flat cylinders:
	\[\mathcal{C}_\Gamma := \{C \subset \Omega: C \text{ is a  maximal flat cylinder}\}.\]
	
	\begin{figure}[H]
		\definecolor{front}{RGB}{31, 38, 62}
		\definecolor{middle}{RGB}{63,91,123}
		\definecolor{back}{RGB}{98,145,166}
		\centering
		\begin{tikzpicture}[scale=1.3, every node/.style={scale=1},blend group=lighten]
			\begin{scope}[rotate=24]
				\coordinate (A1) at (-1.6, 0);
				\coordinate (A2) at (-1.6, 1);
				\coordinate (A3) at (-0.6, 0.5);
				\coordinate (A4) at (0, 0.5);
				\coordinate (A5) at (0.6, 0.5);
				\coordinate (A6) at (1.6, 1);
				\coordinate (A7) at (1.6, 0);
				\coordinate (A8) at (1.6, -1);
				\coordinate (A9) at (0.6, -0.5);
				\coordinate (A10) at (0, -0.5);
				\coordinate (A11) at (-0.6, -0.5);
				\coordinate (A12) at (-1.6, -1);
				
				\draw[brown, thin, fill = brown, fill opacity=0.2] (-1.6, 0)  ..  controls (-1.6, 1) and (-0.6, 0.5) .. (0, 0.5) .. controls (0.6, 0.5) and (1.6, 1) .. (1.6, 0) .. controls (1.6, -1) and (0.6, -0.5) .. (0, -0.5) .. controls  (-0.6, -0.5) and (-1.6, -1) .. (-1.6, 0) ;
				
				\draw (-1.1, 0.1) .. controls (-1, -0.1) and (-0.6, -0.1) .. (-0.5, 0.1);
				\draw[fill = white] (-1, 0) .. controls (-0.9, -0.07) and (-0.7, -0.07) .. (-0.6, 0);
				\draw[fill = white] (-1, 0) .. controls (-0.9, 0.1) and (-0.7, 0.1) .. (-0.6, 0);
				
				\draw (1.1, 0.1) .. controls (1, -0.1) and (0.6, -0.1) .. (0.5, 0.1);
				\draw[fill = white] (1, 0) .. controls (0.9, -0.07) and (0.7, -0.07) .. (0.6, 0);
				\draw[fill = white] (1, 0) .. controls (0.9, 0.1) and (0.7, 0.1) .. (0.6, 0);
			\end{scope}
			
			\draw[olive, thin] (-0.75, -0.4) .. controls (-0.5, -0.4) and (1.2, 0).. (1.2, .5) .. controls (1.2, 0.8) and (0.55, 0.4) .. (0.55, 0.25);
			\draw[olive, dashed, thin] (0.55, 0.25) .. controls (0.4, 0) and (-0.2, -0.3) .. (-0.55, -0.24);
			\draw[olive, thin] (-0.55, -0.24) .. controls (-0.4, 0.6) and (1.8, 1.5) .. (1.35, 0.15) .. controls (1.2, -0.2) and (0.2, -0.3).. (0, -0.55);
			\draw[olive, dashed, thin] (-0.05, -0.55) .. controls (-0.4, -0.5) and (-0.8, -0.5) .. (-0.75, -0.4);
			
			\draw[olive, thin] (-0.75, 1.6) .. controls (-0.8, 1.5) and (-0.4, 1.5) ..  (0, 1.45) .. controls (0.2, 1.4) and (1.2, 1.8) .. (1.45, 2.2) .. controls (1.8, 3.5) and (-0.4, 2.6)..  (-0.55, 1.76) .. controls (-0.2, 1.7) and (0.5, 2) .. (0.55, 2.25) .. controls (0.55, 2.4) and (1.2, 2.8) ..(1.2, 2.5);
			\draw[olive, thin] (1.2, 2.5) .. controls (1.2, 2) and (-0.5, 1.6) .. (-0.75, 1.6);
			
			\draw[olive, opacity=0, fill = olive, fill opacity=0.3]  (1.2, .5) .. controls (1.2, 0.8) and (0.55, 0.4) .. (0.55, 0.25) -- (0.55, 2.25) .. controls (0.55, 2.4) and (1.2, 2.8) ..(1.2, 2.5) -- (1.2, .5);
			\draw[olive, opacity=0, fill = olive, fill opacity=0.3] (-0.75, -0.4) .. controls (-0.5, -0.4) and (1.2, 0).. (1.2, .5) -- (1.2, 2.5) .. controls (1.2, 2) and (-0.5, 1.6) .. (-0.75, 1.6) -- (-0.75, -0.4);
			\draw[olive, opacity=0, fill = olive, fill opacity=0.3] (0.55, 0.25) .. controls (0.4, 0) and (-0.2, -0.3) .. (-0.55, -0.24) -- (-0.55, 1.76) .. controls (-0.2, 1.7) and (0.5, 2) .. (0.55, 2.25) -- (0.55, 0.25);
			\draw[olive, opacity=0, fill = olive, fill opacity=0.3] (-0.55, -0.24) .. controls (-0.4, 0.6) and (1.8, 1.5) .. (1.35, 0.15) --  (1.45, 2.2) .. controls (1.8, 3.5) and (-0.4, 2.6)..  (-0.55, 1.76) -- (-0.55, -0.24);
			\draw[olive, opacity=0, fill = olive, fill opacity=0.3] (1.35, 0.15) .. controls (1.2, -0.2) and (0.2, -0.3).. (0, -0.55) -- (0, 1.45) .. controls (0.2, 1.4) and (1.2, 1.8) .. (1.45, 2.2) -- (1.35, 0.15);
			\draw[olive, opacity=0, fill = olive, fill opacity=0.3] (-0.05, -0.55) .. controls (-0.4, -0.5) and (-0.8, -0.5) .. (-0.75, -0.4)--(-0.75, 1.6) .. controls (-0.8, 1.5) and (-0.4, 1.5) ..  (0, 1.45) -- (0, -0.55);

			\draw[teal, thin] (-1.3, -0.4) .. controls (-1.3, -0.8) and (-0.2, -0.8).. (-0.2, -0.2) .. controls (-0.2, 0.2) and (-1.3, 0.2) .. (-1.3, -0.4);
			\draw[teal, thin] (-1.3, 1.6) .. controls (-1.3, 2.2) and (-0.2, 2.2) .. (-0.2, 1.8) .. controls (-0.2, 1.2) and (-1.3, 1.2) .. (-1.3, 1.6);
			\draw[teal, opacity=0, fill = teal, fill opacity=0.2] (-1.3, -0.4) .. controls (-1.3, -0.8) and (-0.2, -0.8) .. (-0.2, -0.2) -- (-0.2, 1.8) .. controls (-0.2, 1.2) and (-1.3, 1.2) .. (-1.3, 1.6) -- (-1.3, -0.4);
			\draw[teal, opacity=0, fill = teal, fill opacity=0.2] (-0.2, -0.2) .. controls (-0.2, 0.2) and (-1.3, 0.2) .. (-1.3, -0.4) -- (-1.3, 1.6) .. controls (-1.3, 2.2) and (-0.2, 2.2) .. (-0.2, 1.8) -- (-0.2, -0.2);

			\draw[<-] (2.45, 0.3) -- (2.45, 1.3);
			\draw (2.2, 0.5) node[above] {$\pi_\psi$};
			\draw (3.35, 1.3) node[above] {$\Omega \cong \Gamma \backslash (\limitset^{(2)} \times \LieA)$};
			\draw (3.35, 0.3) node[below] {$\mathcal{X}_\psi \cong \Gamma \backslash (\limitset^{(2)} \times \R)$};
			\draw (3, -0.5) node[below] {$\pi_\psi(\Gamma(x,y,v))=\Gamma(x,y,\psi(v))$};

		\end{tikzpicture}
		
		\caption{Maximal flat cylinders in the vector bundle $\pi_\psi: \Omega \to \mathcal{X}_\psi$. In this picture, $\mathcal{X}_\psi$ is depicted as the unit tangent bundle of the closed surface shown.} \label{fig:Bundle}
	\end{figure}
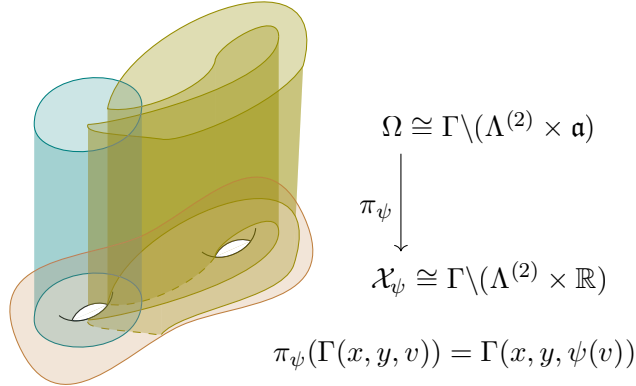
	
	Since $\Gamma$ is Anosov and torsion-free, every nontrivial element of $\gamma \in \Gamma$ is \emph{loxodromic}, that is, conjugate to some $\exp(\lambda(\gamma))m \in (\interior A^+)M$. The element 
	\[\lambda(\gamma) \in \interior \LieA^+\] 
	is unique and called the \emph{Jordan projection} of $\gamma$ and $m \in M$ is in a unique conjugacy class 
	\[h_{\gamma} :=[m] \in [M]\]
	called the \emph{holonomy} of $\gamma$. If $C \in \mathcal{C}_\Gamma$ corresponds to $[\gamma] \in [\primGamma]$, then we define the \emph{$\psi$-circumference} and \emph{holonomy} of $C$ as
	\[\ell_\psi(C): = \psi(\lambda(\gamma)) \in (0,\infty) \quad \text{and} \quad h_C:= h_\gamma \in [M],\]
	respectively (see \cref{thm:AnosovSubgroupsPSTheoryProperties} for the positivity of $\psi$-circumferences). 
	
	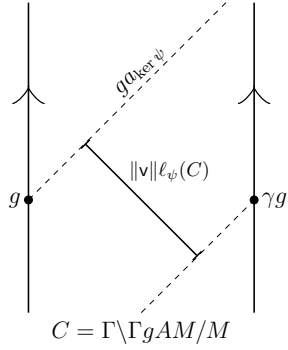
\begin{figure}[H]
		\centering
		\scalebox{0.75}{\begin{tikzpicture}
				\draw[thick] (-2,-4) -- (-2,1.5);
				\draw[thick] (2,-4) -- (2,1.5);
				\draw[dashed] (-2,-2) -- (1.5, 1.5);
				\draw[dashed] (0, -4) -- (2,-2);
				\draw[-{Classical TikZ Rightarrow[length=3mm]}, thick] (-2,0) -- (-2, 0);
				\draw[-{Classical TikZ Rightarrow[length=3mm]}, thick] (2,0) -- (2, 0);
				\filldraw[black] (-2,-2) circle (2pt) node[anchor=east]{\large $g$};
				\filldraw[black] (2,-2) circle (2pt) node[anchor=west]{\large $\gamma g$};
				\draw (0,-0.35) node[anchor=south]{\rotatebox{45}{\large $ga_{\ker\psi}$}};
				\draw (0,-4) node[anchor=north]{\large $C=\Gamma\backslash \Gamma gAM/M$};
				\draw[|-|, thick] (-1,-1) -- (1, -3);
				\draw (0.5,-1.5) node{\small $\|\mathsf{v}\|\ell_\psi(C)$};
		\end{tikzpicture}}
		\caption{A closed $A$-orbit in $\Gamma\backslash G/M$ laid flat. The left and right boundaries are identified so that $g$ and $\gamma g$ are identified. The $\psi$-circumference of $C$ can be seen in the width of the slice produced by the $\ker\psi$ hyperplane.} \label{fig:Flat}
	\end{figure}
	We give a geometric description of $\psi$-circumferences. As stated before, every maximal flat cylinder is contained in $\Omega$. Another feature of Anosov subgroups is that $\Omega$ is homeomorphic to a trivial $\ker\psi$-vector bundle over a compact metric space $\mathcal{X}_\psi$ equipped with a finite measure $m_{\mathcal{X}_\psi}$ on $\mathcal{X}_\psi$ such that 
	\begin{equation}
		\label{eqn:BMS}
		d\BMS\bigr|_\Omega = dm_{\mathcal{X}_\psi} \, du,
	\end{equation}
	for some Lebesgue measure $du$ on $\ker\psi$ (see \cref{subsec:TheSupportOfTheBMSMeasureAsAVectorBundle} for details). The $\psi$-circumference of any maximal flat cylinder is the length of its intersection curve with $\mathcal{X}_\psi$ (\cref{fig:Bundle}). More precisely, $\mathcal{X}_\psi$ is equipped with a \emph{translation flow} and the intersection curve of a maximal flat cylinder with $\mathcal{X}_\psi$ is a periodic orbit of the translation flow (\cref{lem:PeriodicOrbits}). See \cref{fig:Flat} for another interpretation.
	
	\subsection{Statement of the main results}
	
	The main result of our paper is the joint equidistribution of maximal flat cylinders and their holonomies as their $\psi$-circumferences tend to infinity. For $T >0$, let $\mathcal{C}_\psi(T)$ denote the set of all maximal flat cylinders with $\psi$-circumference at most $T$:
	\[\mathcal{C}_\psi(T) := \{C \in \mathcal{C}_\Gamma : \ell_\psi(C) \le T\}.\]
	We note that $\mathcal{C}_\psi(T)$ is always a finite set (see \cref{lem:FinitelyManyCylinders}). To formulate the joint equidistribution statement, denote by $V_C$ the volume measure on a maximal flat cylinder $C$ induced by the Haar measure on $AM$. Let $\mathrm{C}_{\mathrm{c}}(\Gamma \backslash G/M)$ denote the set of continuous compactly supported functions on $\Gamma \backslash G/M$ and let $\mathrm{Cl}(M)$ denote the set of continuous real-valued class functions on $M$. For $T>0$, we define Radon measures $\mu_T$ and $\eta_T$ on the product of $\Gamma \backslash G/M$ and $[M]$ as follows. For $T>0$, $f \in \mathrm{C}_{\mathrm{c}}(\Gamma \backslash G/M)$ and $\varphi \in \mathrm{Cl}(M)$, let 
	\begin{align*}
		\mu_T(f\otimes \varphi) & :=\sum_{C \in\mathcal{C}_{\psi}(T)} V_C(f)\varphi(h_C);
		\\
		\eta_T(f\otimes \varphi) & :=\sum_{C\in\mathcal{C}_\psi(T)}\frac{V_C(f)}{\ell_\psi(C)} \varphi(h_C).
	\end{align*}
	Let $M_\Gamma < M$ denote the closed subgroup generated by all holonomies of elements in $\Gamma$. Then $M_\Gamma$ is a normal subgroup of $M$ containing the identity component of $M$ \cite[Corollary 1.10]{GR07} so in particular, $M_\Gamma$ has finite index in $M$. We note that when $M$ is not connected, $M_\Gamma$ is not necessarily equal to $M$ (see \cref{subsec:HitchinRepresentations} for example). We now state our main theorem describing the joint equidistribution of maximal flat cylinders and their holonomies.
	
	\begin{theorem}[Joint equidistribution]
		\label{thm:JointEquidistribution}
		Let $\Gamma<G$ be a Zariski dense Anosov subgroup. For any tangent form $\psi \in \LieA^*$ and for any $f\in C_{\mathrm{c}}(\Gamma \backslash G/M)$ and $\varphi \in \mathrm{Cl}(M)$, we have
		\begin{align*}
			\lim\limits_{T \to \infty}\frac{\mu_T(f\otimes\varphi)}{e^{T}} & = \frac{1}{|m_{\mathcal{X}_\psi}|} \BMS(f)\int_{M_\Gamma} \varphi \, dm_\Gamma;
			\\
			\lim\limits_{T \to \infty}\frac{\eta_T(f\otimes\varphi)}{e^{T}/T} &= \frac{1}{|m_{\mathcal{X}_\psi}|} \BMS(f)\int_{M_\Gamma} \varphi \, dm_\Gamma,
		\end{align*}
		where $m_\Gamma$ denotes the Haar probability measure on $M_\Gamma$.
	\end{theorem}
	
	We note that the right-hand side of \cref{thm:JointEquidistribution} is independent of the normalization of $\BMS$ by \eqref{eqn:BMS}. We also obtain the following equidistribution of holonomies from \cref{thm:JointEquidistribution}. Since $\BMS$ is an infinite measure when $\textrm{rank} (G)\ge 2$, this is not an immediate consequence of \cref{thm:JointEquidistribution}.
	
	\begin{corollary}[Equidistribution of holonomies]
		\label{cor:EquidistributionOfHolonomies}
		Using the same notation as in \cref{thm:JointEquidistribution}, we have
		\[\sum_{C \in \mathcal{C}_{\psi}(T)} \varphi(h_C) \sim \frac{e^{T}}{T}\int_{M_\Gamma} \varphi  \, dm_\Gamma \quad \text{as } T \to \infty.\]
	\end{corollary}
	
	\begin{remark}
		As a consequence of \cref{thm:JointEquidistribution}, the set of holonomies of $\Gamma$ is dense in $M_\Gamma$, that is,
		\[M_\Gamma=\overline{\{m \in M: \exists C \in \mathcal{C}_\Gamma \text{ such that } h_C = [m]\}}.\]
	\end{remark}
	
	\begin{remark}
		Let $\limitcone$ denote the limit cone of $\Gamma$, that is, the smallest closed cone containing the Jordan projections of all $\gamma \in \Gamma$. \cref{thm:JointEquidistribution,cor:EquidistributionOfHolonomies} can be adapted for linear forms $\psi \in \LieA^*$ which are positive on $\limitcone \setminus\{0\}$ by using the fact that $\delta_\psi \psi$ is tangent to $\growthindicator$ for some $\delta_\psi > 0$ called the \emph{$\psi$-critical exponent} of $\Gamma$ (see \cref{thm:AnosovSubgroupsPSTheoryProperties}). In this context, our theorem implies the $\psi$-critical exponent is equal to the \emph{$\psi$-topological entropy}, that is,
		\begin{multline}
			\label{eqn:IntroRemark1}
			\delta_\psi = \lim_{t \to +\infty}\frac{1}{t}\log \#\{\gamma \in \Gamma: \psi(\mu(\gamma)) \le t\}  
			\\ 
			=\lim_{t \to +\infty}\frac{1}{t}\log \#\{[\gamma] \in [\primGamma]: \psi(\lambda(\gamma)) \le t\}, 
		\end{multline}
		where $\mu: G \to \LieA^+$ denotes the Cartan projection (\cref{def:GrowthIndicatorFunction}). Moreover, \cref{cor:EquidistributionOfHolonomies} immediately implies
		\begin{equation}
			\label{eqn:IntroRemark2}
			\#\mathcal{C}_\psi(T) = \#\{[\gamma] \in [\primGamma]: \psi(\lambda(\gamma)) \le t\} \sim \frac{e^{\delta_\psi T}}{\delta_\psi T},
		\end{equation}
		where for $f_1,f_2:(0,\infty) \to \R$, we write 
		\[f_1 \sim f_2 \iff \lim\limits_{T\to\infty} \frac{f_1(T)}{f_2(T)} = 1.\]
		The above \eqref{eqn:IntroRemark1} and \eqref{eqn:IntroRemark2} were proved by Sambarino \cite[Theorem 7.8]{Sam14b}, \cite[Corollary 4.4]{Sam14a} when $\Gamma$ is the fundamental group of a closed connected negatively curved Riemannian manifold and in view of \cite[Appendix A]{Car21}, Sambarino's work extends to Anosov subgroups. We also mention \cite[Corollary 11.1]{BCKM22} which is a counting result analogous to \eqref{eqn:IntroRemark2} for \emph{cusped} Anosov representations.
	\end{remark}
	
	\subsection{Joint equidistribution with respect to norm-like functions}
	
	Our proofs also allow us to prove in \cref{sec:NormAlternate} similar results when $\mathcal{C}_\Gamma$ is ordered according to a norm-like function. Let $\mathsf{N}:\LieA^+\to\R$ be a \emph{norm-like} function, that is, $\mathsf{N}$ is twice continuously differentiable except possibly at the origin, convex, homogeneous of degree 1 and positive on $\limitcone \setminus \{0\}$. For example, $L^p$ norms are norm-like for $1\le p < \infty$. The function $\mathsf{N}$ determines an ordering on $\mathcal{C}_\Gamma$: 
	\[\mathcal{C}_\mathsf{N}(T) := \{C \in \mathcal{C}_\Gamma : \mathsf{N}(\lambda(\gamma_C)) \le T\}.\]
	We define the $\mathsf{N}$-critical exponent as 
	\begin{equation}
		\label{eqn:NCriticalExponentIntro}
		\delta_\mathsf{N}:=\max_{\mathsf{N}(w)=1} \psi_\Gamma(w)>0.
	\end{equation}
	For simplicity, we only state here the equidistribution of holonomies with respect to $\mathsf{N}$.
	
	\begin{corollary}[Equidistribution of holonomies with respect to $\mathsf{N}$]
		\label{cor:HolonomyDistributionNorm} 
		There exists a constant $0<c_\mathsf{N}\le 1$ such that for any $\varphi\in\mathrm{Cl}(M)$, we have
		\[\sum_{C\in\mathcal{C}_\mathsf{N}(T)}\varphi(h_C)\sim c_\mathsf{N}\frac{e^{\delta_\mathsf{N} T}}{\delta_\mathsf{N} T} \int_{M_\Gamma} \varphi\, dm_\Gamma  \quad \textrm{as } T \to \infty.\]
		Moreover, $c_\mathsf{N} = 1$ if and only if the Hessian of $\mathsf{N}$ at $\mathsf{v}$ is identically zero where $\mathsf{v}$ is the unique vector achieving the maximum in \eqref{eqn:NCriticalExponentIntro}.
	\end{corollary}
	
	\subsection{Comparison with rank one case}
	
	When $\rankG = 1$, Anosov subgroups coincide with Zariski dense convex cocompact subgroups. Denote the critical exponent of $\Gamma$ by $\delta$ and the set of closed geodesics in $\Gamma \backslash G/M$ of length at most $T$ by
	\[\mathcal{G}_\Gamma(T):= \{\text{primitive closed geodesics of length at most } T\}.\]
	
	Then using the dictionary in \cref{tab:Dictionary} and noting that $M$ is connected when $G$ is rank one and center-free, \cref{cor:EquidistributionOfHolonomies} in those cases says that for all $\varphi \in \mathrm{Cl}(M)$, we have
	\[\sum_{C \in \mathcal{G}_\Gamma(T)} \varphi(h_C) \sim \frac{e^{\delta T}}{\delta T}\int_{M} \varphi  \, dm \quad \text{as } T \to \infty.\]
	This is a special case of \cite[Theorem 1.4]{MMO14}. 
	
	\begin{table}[H]
		\begin{tabular}{c|c}
			rank one & higher rank \\
			\hline
			convex cocompact groups & Anosov groups \\
			primitive closed geodesic $C$ & pos. oriented maximal flat cylinder $C$ \\
			$\delta \times (\text{length of } C)$ & $\psi$-circumference of $C$ \\
			holonomy of $C$ & holonomy of $C$ \\
			finite BMS-measure & infinite BMS-measure associated to $\psi$ \\
			\includegraphics[scale = 0.4]{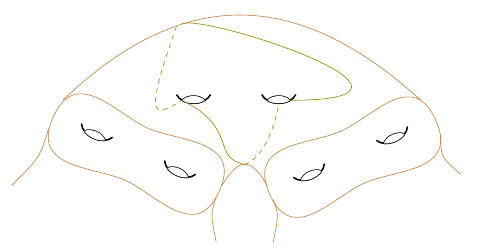} & \includegraphics[scale=0.3]{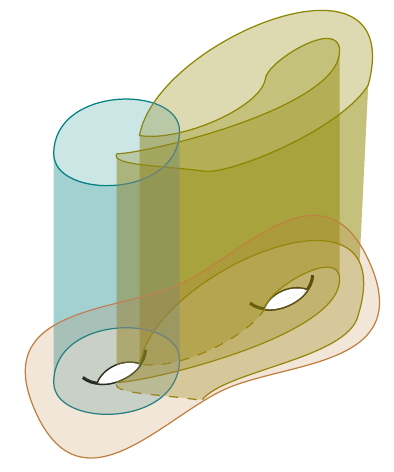} 
		\end{tabular}
		
		\caption{\label{tab:Dictionary} A dictionary between the rank one setting and the higher rank setting.}
	\end{table}
	
	For rank one groups, the asymptotic for $\mathcal{G}_\Gamma(T)$ was proved by Margulis \cite{Mar04}\footnote{\cite{Mar04} contains Margulis' previously unpublished 1970 thesis. See \cite{Par05} for a review.} when $\Gamma$ is a uniform lattice, Gangolli--Warner \cite{GW78} when $\Gamma$ is a nonuniform lattice, and Roblin \cite{Rob03} when $\Gamma$ is geometrically finite, and the equidistribution of closed geodesics was independently studied by Margulis \cite{Mar04} and Bowen \cite{Bow72a,Bow72b} when $\Gamma$ is a uniform lattice and Roblin \cite{Rob03} when $\Gamma$ is geometrically finite. Equidistribution of holonomies was proved by Parry--Pollicott \cite{PP86} when $\Gamma$ is a uniform lattice in $\SO(n,1)^\circ$. In general rank one groups, equidistribution of holonomies was proved by Sarnak--Wakayama \cite{SW99} when $\Gamma$ is a lattice and joint equidistribution was proved by Margulis--Mohammadi--Oh \cite{MMO14} when $\Gamma$ is a Zariski dense geometrically finite subgroup. A joint equidistribution result was obtained by Oh--Pan \cite{OP19} for abelian covers of convex cocompact rank one groups; in this case, the BMS measure is infinite unlike in \cite{MMO14}. A key technique in \cite{MMO14} is to use mixing of the frame flow \cite{Win15}, and this technique goes back to Margulis' work on the counting and equidistribution of closed geodesics in negatively curved compact manifolds in his 1970 thesis. This mixing is exponential for $\Gamma$ convex cocompact by Chow--Sarkar \cite{CS22}, and for $\Gamma$ geometrically finite with parabolic elements by Li--Pan--Sarkar \cite{LPS23}. In these cases, the distribution of holonomies can be determined up to an exponential error term (cf. \cite[Theorem 1.2]{MMO14}).
	
	\subsection{Some related results in higher rank}
	
	For cocompact lattices of higher rank Lie groups, we mention the thesis of Spatzier \cite{Spa83} where he describes the exponential growth rate of the total volume of maximal flat periodic tori as their regular systole tends to infinity and the recent work by Dang--Li \cite{DL22} on the counting and equidistribution of these maximal flat periodic tori.
	
	\subsection{Example: Self-joinings of convex-cocompact subgroups}
	\label{subsec:ConvexCocompactSelfJoinings}
	For $G = \PSL_2\C\times\PSL_2\C$, all Anosov subgroups arise as \emph{convex-cocompact self-joinings}, which have interesting applications to rigidity as shown by Kim--Oh \cite{KO22,KO23a,KO23b}. Let $\Delta < \PSL_2\C$ be a Zariski dense convex cocompact subgroup and $\rho: \Delta \to \PSL_2\C$ be a convex cocompact discrete faithful representation. Then the \emph{self-joining} $\Delta_\rho$ of $\Delta$ by $\rho$ is defined as the diagonal embedding of $\Delta$ in $G$ via $\rho$, that is,  
	\[\Delta_\rho := \{(\gamma,\rho(\gamma)) \in \PSL_2\C \times \PSL_2\C: \gamma \in \Delta\}< \PSL_2\C\times\PSL_2\C.\]
	
	\begin{figure}[H]
		\definecolor{front}{RGB}{31, 38, 62}
		\definecolor{middle}{RGB}{63,91,123}
		\definecolor{back}{RGB}{98,145,166}
		\centering
		\begin{tikzpicture}[scale=0.6, every node/.style={scale=1},blend group=lighten]
			\begin{scope}[rotate=24]
				\coordinate (A1) at (-1.6, 0);
				\coordinate (A2) at (-1.6, 1);
				\coordinate (A3) at (-0.6, 0.5);
				\coordinate (A4) at (0, 0.5);
				\coordinate (A5) at (0.6, 0.5);
				\coordinate (A6) at (1.6, 1);
				\coordinate (A7) at (1.6, 0);
				\coordinate (A8) at (1.6, -1);
				\coordinate (A9) at (0.6, -0.5);
				\coordinate (A10) at (0, -0.5);
				\coordinate (A11) at (-0.6, -0.5);
				\coordinate (A12) at (-1.6, -1);
				
				\draw[brown, thin] (-1.6, 0)  ..  controls (-1.6, 1) and (-0.6, 0.5) .. (0, 0.5) .. controls (0.6, 0.5) and (1.6, 1) .. (1.6, 0) .. controls (1.6, -1) and (0.6, -0.5) .. (0, -0.5) .. controls  (-0.6, -0.5) and (-1.6, -1) .. (-1.6, 0) ;
				
				\draw (-1.1, 0.1) .. controls (-1, -0.1) and (-0.6, -0.1) .. (-0.5, 0.1);
				\draw[fill = white] (-1, 0) .. controls (-0.9, -0.07) and (-0.7, -0.07) .. (-0.6, 0);
				\draw[fill = white] (-1, 0) .. controls (-0.9, 0.1) and (-0.7, 0.1) .. (-0.6, 0);
				
				\draw (1.1, 0.1) .. controls (1, -0.1) and (0.6, -0.1) .. (0.5, 0.1);
				\draw[fill = white] (1, 0) .. controls (0.9, -0.07) and (0.7, -0.07) .. (0.6, 0);
				\draw[fill = white] (1, 0) .. controls (0.9, 0.1) and (0.7, 0.1) .. (0.6, 0);
			\end{scope}
			
			\begin{scope}[rotate=-24,shift={(-3.5,-1.5)}]
				\coordinate (A1) at (-1.6, 0);
				\coordinate (A2) at (-1.6, 1);
				\coordinate (A3) at (-0.6, 0.5);
				\coordinate (A4) at (0, 0.5);
				\coordinate (A5) at (0.6, 0.5);
				\coordinate (A6) at (1.6, 1);
				\coordinate (A7) at (1.6, 0);
				\coordinate (A8) at (1.6, -1);
				\coordinate (A9) at (0.6, -0.5);
				\coordinate (A10) at (0, -0.5);
				\coordinate (A11) at (-0.6, -0.5);
				\coordinate (A12) at (-1.6, -1);
				
				\draw[brown, thin] (-1.6, 0)  ..  controls (-1.6, 1) and (-0.6, 0.5) .. (0, 0.5) .. controls (0.6, 0.5) and (1.6, 1) .. (1.6, 0) .. controls (1.6, -1) and (0.6, -0.5) .. (0, -0.5) .. controls  (-0.6, -0.5) and (-1.6, -1) .. (-1.6, 0) ;
				
				\draw (-1.1, 0.1) .. controls (-1, -0.1) and (-0.6, -0.1) .. (-0.5, 0.1);
				\draw[fill = white] (-1, 0) .. controls (-0.9, -0.07) and (-0.7, -0.07) .. (-0.6, 0);
				\draw[fill = white] (-1, 0) .. controls (-0.9, 0.1) and (-0.7, 0.1) .. (-0.6, 0);
				
				\draw (1.1, 0.1) .. controls (1, -0.1) and (0.6, -0.1) .. (0.5, 0.1);
				\draw[fill = white] (1, 0) .. controls (0.9, -0.07) and (0.7, -0.07) .. (0.6, 0);
				\draw[fill = white] (1, 0) .. controls (0.9, 0.1) and (0.7, 0.1) .. (0.6, 0);
			\end{scope}
			
			\draw[brown, thin] (-5.1, 1)  ..  controls (-3, 3) and (-1, 3) .. (1.2, 1);
			\draw[brown, thin] (-2.3, -0.5)  ..  controls (-2, 0) and (-1.8, 0) .. (-1.5, -0.5);
			
			\draw[brown, thin] (-5.35, 0.5)  ..  controls (-5.5, 0) and (-5.5, 0) .. (-6, -0.5);
			\draw[brown, thin] (-2.4, -0.7)  ..  controls (-2.5, -1) and (-2.5, -1) .. (-2.4, -1.5);
			\draw[brown, thin] (-1.4, -0.75)  ..  controls (-1.3, -1) and (-1.3, -1) .. (-1.4, -1.5);
			\draw[brown, thin] (1.55, 0.4)  ..  controls (1.55, 0) and (1.55, 0) .. (1.9, -0.3);
			
			\begin{scope}[shift={(-2,1)}]
				\draw (-1.1, 0.1) .. controls (-1, -0.1) and (-0.6, -0.1) .. (-0.5, 0.1);
				\draw[fill = white] (-1, 0) .. controls (-0.9, -0.07) and (-0.7, -0.07) .. (-0.6, 0);
				\draw[fill = white] (-1, 0) .. controls (-0.9, 0.1) and (-0.7, 0.1) .. (-0.6, 0);
			\end{scope}
			
			\begin{scope}[shift={(-0.5,1)}]
				\draw (-1.1, 0.1) .. controls (-1, -0.1) and (-0.6, -0.1) .. (-0.5, 0.1);
				\draw[fill = white] (-1, 0) .. controls (-0.9, -0.07) and (-0.7, -0.07) .. (-0.6, 0);
				\draw[fill = white] (-1, 0) .. controls (-0.9, 0.1) and (-0.7, 0.1) .. (-0.6, 0);
			\end{scope}
			
			\draw[olive, thin] (-3, 1) .. controls (-2, 0.5) and (-2.5, 0).. (-1.9, -0.13);
			\draw[olive, dashed, thin] (-1.9, -0.13) .. controls (-1.5, 0) and (-1.3, 0.8) .. (-1.3,1);
			\draw[olive, thin] (-1.1,1) .. controls (2, 1) and (-2.5, 2.5).. (-3, 2.35);
			\draw[olive, dashed, thin] (-3.1, 2.3) .. controls (-3.7, 0.6) and (-3.5, 0.7) .. (-3,1);
			
			\draw (-1.5, 3.5) node[below] {$\Delta \backslash \Ha^3$};
			
			\begin{scope}[shift={(10,0)}]
				\begin{scope}[rotate=24]
					\coordinate (A1) at (-1.6, 0);
					\coordinate (A2) at (-1.6, 1);
					\coordinate (A3) at (-0.6, 0.5);
					\coordinate (A4) at (0, 0.5);
					\coordinate (A5) at (0.6, 0.5);
					\coordinate (A6) at (1.6, 1);
					\coordinate (A7) at (1.6, 0);
					\coordinate (A8) at (1.6, -1);
					\coordinate (A9) at (0.6, -0.5);
					\coordinate (A10) at (0, -0.5);
					\coordinate (A11) at (-0.6, -0.5);
					\coordinate (A12) at (-1.6, -1);
					
					\draw[brown, thin] (-1.6, 0)  ..  controls (-1.6, 1) and (-0.6, 0) .. (0, 0.5) .. controls (0.6, 0.5) and (1.6, 1) .. (1.6, 0) .. controls (1.6, -1) and (0.6, -0) .. (0, -0.5) .. controls  (-0.6, -0.5) and (-1.6, -1) .. (-1.6, 0) ;
					
					\draw (-1.1, 0.1) .. controls (-1, -0.1) and (-0.6, -0.1) .. (-0.5, 0.1);
					\draw[fill = white] (-1, 0) .. controls (-1, -0.07) and (-1, -0.07) .. (-0.6, 0);
					\draw[fill = white] (-1, 0) .. controls (-1, 0.1) and (-1, 0.1) .. (-0.6, 0);
					
					\draw (1.1, 0.1) .. controls (1, -0.1) and (0.6, -0.1) .. (0.5, 0.1);
					\draw[fill = white] (1, 0) .. controls (1, -0.07) and (1, -0.07) .. (0.6, 0);
					\draw[fill = white] (1, 0) .. controls (1, 0.1) and (1, 0.1) .. (0.6, 0);
				\end{scope}
				
				\begin{scope}[rotate=-24,shift={(-3.5,-1.5)}]
					\coordinate (A1) at (-1.6, 0);
					\coordinate (A2) at (-1.6, 1);
					\coordinate (A3) at (-0.6, 0.5);
					\coordinate (A4) at (0, 0.5);
					\coordinate (A5) at (0.6, 0.5);
					\coordinate (A6) at (1.6, 1);
					\coordinate (A7) at (1.6, 0);
					\coordinate (A8) at (1.6, -1);
					\coordinate (A9) at (0.6, -0.5);
					\coordinate (A10) at (0, -0.5);
					\coordinate (A11) at (-0.6, -0.5);
					\coordinate (A12) at (-1.6, -1);
					
					\draw[brown, thin] (-1.6, 0)  ..  controls (-1.6, 1) and (-0.6, 0) .. (0, 0.5) .. controls (0.6, 0.5) and (1.6, 1) .. (1.6, 0) .. controls (1.6, -1) and (0.6, -0) .. (0, -0.5) .. controls  (-0.6, -0.5) and (-1.6, -1) .. (-1.6, 0) ;
					
					\draw (-1.1, 0.1) .. controls (-1, -0.1) and (-0.6, -0.1) .. (-0.5, 0.1);
					\draw[fill = white] (-1, 0) .. controls (-1, -0.07) and (-1, -0.07) .. (-0.6, 0);
					\draw[fill = white] (-1, 0) .. controls (-1, 0.1) and (-1, 0.1) .. (-0.6, 0);
					
					\draw (1.1, 0.1) .. controls (1, -0.1) and (0.6, -0.1) .. (0.5, 0.1);
					\draw[fill = white] (1, 0) .. controls (1, -0.07) and (1, -0.07) .. (0.6, 0);
					\draw[fill = white] (1, 0) .. controls (1, 0.1) and (1, 0.1) .. (0.6, 0);
				\end{scope}
				
				\draw[brown, thin] (-5.1, 1)  ..  controls (-3, 3) and (-1, 3) .. (1.2, 1);
				\draw[brown, thin] (-2.3, -0.5)  ..  controls (-2, 0) and (-1.8, 0) .. (-1.5, -0.5);
				
				\draw[brown, thin] (-5.35, 0.5)  ..  controls (-5.5, 0) and (-5.5, 0) .. (-6, -0.5);
				\draw[brown, thin] (-2.4, -0.7)  ..  controls (-2.5, -1) and (-2.5, -1) .. (-2.4, -1.5);
				\draw[brown, thin] (-1.4, -0.75)  ..  controls (-1.3, -1) and (-1.3, -1) .. (-1.4, -1.5);
				\draw[brown, thin] (1.55, 0.4)  ..  controls (1.55, 0) and (1.55, 0) .. (1.9, -0.3);
				
				\begin{scope}[shift={(-2,1)}]
					\draw (-1.1, 0.1) .. controls (-1, -0.1) and (-1, -0.1) .. (-0.5, 0.1);
					\draw[fill = white] (-1, 0) .. controls (-1, -0.07) and (-1, -0.07) .. (-0.6, 0);
					\draw[fill = white] (-1, 0) .. controls (-1, 0.1) and (-1, 0.1) .. (-0.6, 0);
				\end{scope}
				
				\begin{scope}[shift={(-0.5,1)}]
					\draw (-1.1, 0.1) .. controls (-1, -0.1) and (-1, -0.1) .. (-0.5, 0.1);
					\draw[fill = white] (-1, 0) .. controls (-1, -0.07) and (-1, -0.07) .. (-0.6, 0);
					\draw[fill = white] (-1, 0) .. controls (-1, 0.1) and (-1, 0.1) .. (-0.6, 0);
				\end{scope}
				
				\draw[olive, thin] (-3, 1) .. controls (-3, 0.5) and (-2.5, 0.5).. (-1.9, -0.13);
				\draw[olive, dashed, thin] (-1.9, -0.13) .. controls (-1.5, 0.5) and (-1, 0.8) .. (-1.3,1);
				\draw[olive, thin] (-1.1,1) .. controls (2.8, 1) and (-2.5, 2.5).. (-3, 2.35);
				\draw[olive, dashed, thin] (-3.1, 2.3) .. controls (-5, 0.6) and (-3.5, 0.7) .. (-3,1);
				
				\draw (-1.5, 3.5) node[below] {$\rho(\Delta) \backslash \Ha^3$};
			\end{scope}
			
			\draw (3, 5) node[above] {$\Delta_\rho \backslash (\Ha^3\times\Ha^3)$};
			
			\draw[->] (2.5, 5) -- (-1.5, 3.5);
			\draw[->] (3.5, 5) -- (8.5, 3.5);
		\end{tikzpicture}
	\end{figure}
	
	Note that $\Delta_\rho$ is Zariski dense in $\PSL_2\C \times \PSL_2\C$ if and only if $\rho$ does not extend to an automorphism of $\PSL_2\C$ and is hence conjugation (cf. \cite[Lemma 4.1]{KO22}). 
	
	For primitive $[\gamma] \in [\Delta]$, let $C_\gamma$ denote the primitive closed geodesic in $\Delta \backslash \Ha^3$ corresponding to $[\gamma] \in [\Delta]$ and $\ell(C_\gamma)$ denote the length of $C_\gamma$. In this setting, $M_{\Delta_\rho} = M \cong \mathbb{S}^1 \times \mathbb{S}^1$ and we have the following result analogous to \cite[Theorem 1.3]{MMO14}. 
	
	\begin{corollary}[Holonomy rigidity for convex cocompact groups]
		\label{cor:SelfJoiningsExample}
		Let $\Delta < \PSL_2\C$ be a Zariski dense convex cocompact subgroup with critical exponent $\delta_\Delta$. Let $\Delta_\rho$ be the self-joining of $\Delta$ by a convex-cocompact discrete faithful representation $\rho: \Delta \to \PSL_2\C$. Identify holonomies with pairs of angles. If $\rho$ does not extend to $\PSL_2\C$, then for any $0 <\theta_1 < \theta_2 < 2\pi$ and any $0< \theta_3 < \theta_4 < 2\pi$, we have
		\begin{multline*}
			\#\{C_\gamma: \ell(C_\gamma) \le T \text{ and } h_{(\gamma,\rho(\gamma))} \in (\theta_1, \theta_2) \times (\theta_3,\theta_4) \} 
			\\
			\sim \frac{(\theta_2-\theta_1)(\theta_4-\theta_3)e^{\delta_\Delta T}}{(2\pi)^2 \delta_\Delta T} \qquad \text{as } T \to \infty.
		\end{multline*}
	\end{corollary} 
	
	We note that when $\rho$ does extend to $\PSL_2\C$, the set of holonomies is contained in the diagonal of $\mathbb S^1 \times \mathbb S^1$ and hence, is not even dense in $\mathbb S^1 \times \mathbb S^1$. The representation $\rho$ does not extend  if and only if the set of holonomies $\{h_{(\gamma,\rho(\gamma))} \in \mathbb{S}^1 \times \mathbb{S}^1:\ell(C_\gamma) \le T\}$ is equidistributed in $\mathbb{S}^1 \times \mathbb{S}^1$ as $T \to \infty$. Hence, \cref{cor:SelfJoiningsExample} can be viewed as a holonomy rigidity statement for convex cocompact subgroups.   
	
	\subsection{Example: Hitchin representations of surface groups}
	\label{subsec:HitchinRepresentations}
	Let $\rho:\Gamma_0 \to \PSL_d\R$ be a \emph{Hitchin representation}, that is, $\Gamma_0$ is the fundamental group of a closed orientable surface of genus at least 2 and $\rho$ can be continuously deformed to $\rho_d \circ \rho_0$, where $\rho_0: \Gamma_0 \to \PSL_2\R$ is some discrete faithful representation and $\rho_d:\PSL_2\R \to \PSL_d\R$ denotes the irreducible representation which is unique up to conjugation. Suppose $\Gamma = \rho(\Gamma_0) < \PSL_d\R$ is Zariski dense. 
	
	For $G=\PSL_d\R$, $\LieA \cong \{(t_1,\dots, t_d) \in \R^d: \sum_i t_i = 0\}$. We choose $\LieA^+ = \{(t_1,\dots, t_d) \in \R^d: t_1 \ge \dots \ge t_d, \sum_i t_i = 0\}$. For each $i \in \{1,2,\dots,d-1\}$, let $\alpha_i \in \LieA^*$ denote the simple root given by 
	\[\alpha_i(t_1,\dots, t_d) = t_i - t_{i+1}.\]
	By Potrie--Sambarino \cite[Theorem B]{PS17}, $\alpha_i$ is a tangent form. By Labourie \cite[Theorem 1.5]{Lab06}, 
	\[M_\Gamma = \{e\}.\]
	Denote the eigenvalues of $\gamma \in \Gamma$ by $\lambda_1(\gamma)>\cdots>\lambda_d(\gamma)>0$ and denote the maximal flat cylinder corresponding to $[\gamma]$ by $C(\gamma)$. 
	
	\begin{corollary}
		\label{cor:HitchinExample}
		Let $\Gamma <\PSL_d\R$ be a Zariski dense image of a Hitchin representation. Let $i \in \{1,\dots,d-1\}$. Then for any compactly supported continuous function $f:\Gamma \backslash \PSL_d\R/M \to \R$, as $T \to \infty$,  we have
		\[e^{-T}\sum_{\substack{[\gamma] \in [\Gamma] \\ \lambda_i(\gamma)-\lambda_{i+1}(\gamma) \le T}} \frac{1}{\lambda_i(\gamma)-\lambda_{i+1}(\gamma)}V_{C(\gamma)}(f) \rightarrow \frac{1}{|m_{\mathcal{X}_{\alpha_i}}|}\mathsf{m}_{\alpha_i}(f).\]
	\end{corollary}	
	
	\subsection{\texorpdfstring{Outline of the proof of \cref{thm:JointEquidistribution}}{Outline of the proof of Theorem 1.3}}
	Our proof follows a similar line of proof as in \cite{MMO14}. However, in our higher rank setting, care is needed to overcome the technical obstructions coming from the higher dimensional nature of $A$, the BMS-measure being infinite and the $A$-action not being strong mixing. More precisely, let $g_0 \in G$ and $\varepsilon > 0$. The \emph{$\varepsilon$-flow box centered at $g_0$} is defined as
	\[\mathfrak{B}(g_0,\varepsilon):=g_0(N^+_\varepsilon N \cap N_\varepsilon N^+ AM)M_\varepsilon A_\varepsilon,\]
	where $N^+$ is the horospherical subgroup opposite to $N$ and a subgroup with $\varepsilon$ in the subscript denotes the $\varepsilon$-neighborhood of identity in the subgroup. 
	
	Let $\Theta$ be a conjugation-invariant Borel subset of $M_\Gamma$. It suffices to understand the asymptotic behavior of \[\mu_T(\tilde{\mathcal{B}}(g_0,\varepsilon)\otimes\Theta),\]
	where $\tilde{\mathcal{B}}(g_0,\varepsilon)$ denotes the image of $\mathfrak{B}(g_0,\varepsilon)$ under the projection $G \to \Gamma \backslash G/M$. Let $\limitcone$ denote the \emph{limit cone} of $\Gamma$ (\cref{def:LimitCone}) and let 
	\[\mathcal{L}_T^+ := \{\exp(w): w \in \limitcone, \psi(w) \le T\}.\]
	\cref{lem:MuTAndCountingPrimitiveHyperbolics} relates  $\mu_T(\tilde{\mathcal{B}}(g_0,\varepsilon)\otimes\Theta)$ to the number of elements of $\Gamma$ in the set
	\[\mathcal{W}_T(g_0,\varepsilon,\Theta):=\{gamg^{-1}:g\in\mathcal{B}(g_0,\varepsilon),am\in \mathcal{L}_T^+ \Theta\}.\]
	
	We are then led to consider the set 
	\[\mathcal{V}_T(g_0,\varepsilon,\Theta) := \mathcal{B}(g_0,\varepsilon)\mathcal{L}_T^+\Theta\mathcal{B}(g_0,\varepsilon)^{-1},\]
	which can be thought of as a thickening of $\mathcal{W}_T(g_0,\varepsilon,\Theta)$. We can relate the asymptotic behavior of $\# (\Gamma\cap\mathcal{W}_T(g_0,\varepsilon,\Theta))$ to the asymptotic behavior of $\#(\Gamma\cap\mathcal{V}_T(g_0,\varepsilon,\Theta))$ by using \cref{lem:EffectiveClosingLemma} which is an \emph{effective closing lemma for regular directions}. \cref{lem:EffectiveClosingLemma} says that if $\gamma \in \Gamma$ corresponds to an $AM$-orbit in $\Gamma \backslash G$ that almost closes up along some $\exp(w)m \in AM$ where $w \in \LieA$ is a sufficiently large vector in a regular direction, then there exists a nearby closed $AM$-orbit with period approximately $\exp(w)m$. We can apply \cref{lem:EffectiveClosingLemma} because $\Gamma$ being Anosov implies that its limit cone $\limitcone$ is contained in the interior of $\LieA^+$.
	
	When $\rankG=1$ and the BMS measure is finite, Margulis--Mohammadi--Oh \cite{MMO14} showed that the asymptotic behavior of $\# (\Gamma\cap\mathcal{V}_T(g_0,\varepsilon,\Theta))$ can be obtained by using strong mixing of the $A$-action on $\Gamma \backslash G$. When $G$ has higher and hence, $\BMS$ is infinite, the $A$-action is not strongly mixing on $\Gamma \backslash G$. In place of strong mixing, we have \emph{local mixing} of the one-parameter diagonal flow $\exp(t\mathsf{v})$ due to Chow--Sarkar \cite{CS23} and Edwards--Lee--Oh \cite{ELO22b}, where $\mathsf{v}$ is as in \eqref{eqn:IntroPsi}. In fact, their local mixing theorem applies to the more general one parameter family $\exp(t\mathsf{v}+\sqrt{t}u)$ with $u \in \ker\psi$. Our proof needs this more refined version along with an accompanying uniformity statement. Let $C_{\mathrm{c}}(\Gamma \backslash G)$ denote the set of continuous compactly supported functions on $\Gamma \backslash G$ and $\rho$ denote the half sum of the positive roots with respect to $\LieA^+$. 
	
	\begin{theorem}[Local mixing, {\cite[Theorem 1.3]{CS23}}, {\cite[Theorem 3.4]{ELO22b}}]
		\label{thm:LocalMixingIntro}	
		There exists $\kappa_{\mathsf{v}} >0$ such that for any $u \in \ker\psi$ and for any $\phi_1, \phi_2 \in C_{\mathrm{c}}(\Gamma \backslash G)$, we have
		\begin{multline*}
			\lim_{t \to +\infty} t^{\frac{\rankG - 1}{2}}e^{(2\rho - \psi)(t\mathsf{v} + \sqrt{t}u)} \int_{\Gamma \backslash G} 	\phi_1(x\exp(t\mathsf{v} + \sqrt{t}u)) \phi_2(x) \, dx \\
			=\frac{\kappa_{\mathsf{v}}e^{-I(u)}}{|m_{\mathcal{X}_\psi}|}  \sum_{Z \in \mathfrak{Z}_\Gamma} 	\BR\bigr|_{ZN^+}(\phi_1)\cdot 	\BRstar\bigr|_{ZN}(\phi_2),
		\end{multline*}
		where $dx$ denotes the right $G$-invariant measure on $\Gamma \backslash G$, $I: \ker\psi \to \R$ is defined by $I(\cdot) =
		\langle \cdot, \cdot\rangle_* - \frac{\langle \cdot, \mathsf{v} \rangle_*^2}{\langle \mathsf{v}, \mathsf{v}\rangle_*}$ for some inner product $\langle \cdot, \cdot \rangle_*$ on $\LieA$, $\BR$ and $\BRstar$ denote the Burger--Roblin measures associated to $\psi$ and $\mathfrak{Z}_\Gamma$ denotes the finite set of $A$-ergodic components of $\BMS$. 
		
		Moreover, there exists constants $\eta_{\mathsf{v}}$ and $T_{\mathsf{v}}$ such that for all $\phi_1, \phi_2 \in C_{\mathrm{c}}(\Gamma \backslash G)$, there exists a constant $D_\mathsf{v}(\phi_1,\phi_2)$ depending continuously on $\phi_1$ and $\phi_2$ such that for all $(t,u) \in (T_{\mathsf{v}},\infty) \times \ker\psi$ such that $t\mathsf{v}+\sqrt{t}u \in \limitcone$, we have
		\begin{multline}
			\label{eqn:LocalMixingUniformBound}
			\left|t^{\frac{\rankG - 1}{2}}e^{(2\rho - \psi)(t\mathsf{v} + \sqrt{t}u)} \int_{\Gamma \backslash G} \phi_1(x\exp(t\mathsf{v} + \sqrt{t}u)) \phi_2(x) \, dx\right| \\ \le D_\mathsf{v}(\phi_1,\phi_2) e^{-\eta_{\mathsf{v}} I(u)}.
		\end{multline}
	\end{theorem}
	
	Due to the higher dimensional nature of $\mathcal{L}_T^+$, we need to show that the error term $E(\phi_1,\phi_2,t,u)$ in \cref{thm:LocalMixingIntro} for some compact family of functions $\phi_1$ and $\phi_2$ does not contribute to the asymptotic after integrating on $\mathcal{L}_T^+$. We do this by using \eqref{eqn:LocalMixingUniformBound} to bound $|E(\phi_1,\phi_2,t,u)|$ and show that   
	\[\int_{a_{t\mathsf{v}+\sqrt{t}u} \in \mathcal{L}_T^+}e^{t}E(\phi_1,\phi_2,t,u) \, dt \, du = o(e^T),\]
	where we write $o(f) = g$ if $f,g:\R \to \R$ such that $\lim\limits_{T\to\infty}\frac{f(T)}{g(T)}=1$. After some technical arguments and applying \cref{thm:LocalMixingIntro}, we obtain \cref{prop:STAsymptotic} which is an asymptotic for the number of elements from $\Gamma$ in product subsets $S_T:=\Xi_1 \mathcal{L}_T^+\Theta\Xi_2$ of $N^+AMN$. From \cref{prop:STAsymptotic}, we deduce the asymptotic
	\[\#\Gamma\cap\mathcal{V}_T(g_0,\varepsilon,\Theta) = \frac{[M:M_\Gamma]}{|m_{\mathcal{X}_\psi}|}e^T\left( \frac{\mathsf{m}_\psi(\tilde{\mathcal{B}}(g_0,\varepsilon) \otimes \Theta)}{b_{\rank}(\varepsilon)}  (1+O(\varepsilon)) +o_T(1)\right),\]
	where $[M:M_\Gamma]$ denotes the index of $M_\Gamma$ in $M$, $b_{\rank}(\varepsilon)$ denotes the volume of the Euclidean $\varepsilon$-ball of dimension $\rank$ and $\BMS$ denotes here the BMS measure associated to $\psi$ on $\Gamma \backslash G$ (\cref{prop:CountingInVT}). This asymptotic serves as input to obtain the asymptotic for $\mu_T$. Using the asymptotic for $\mu_T$, we deduce the asymptotic for $\eta_T$ and \cref{cor:EquidistributionOfHolonomies} follows from this by using a careful choice of function $f$ in \cref{thm:JointEquidistribution}.
	
	\subsection*{Organization}
	\cref{sec:EffectiveClosingLemmaForRegularDirections} begins by fixing notation related to the Lie group $G$ that we will use throughout the paper. In \cref{subsec:FlowBoxes}, we define $\varepsilon$-flow boxes. In \cref{subsec:EffectiveClosingLemmaForRegularDirections}, we present and prove an effective closing lemma for regular directions (\cref{lem:EffectiveClosingLemma}). 
	
	In \cref{sec:GeometricMeasures}, we recall geometric objects and geometric measures associated to general Zariski dense discrete subgroups $\Gamma < G$. In particular, we recall the limit set and limit cone in 
	\cref{subsec:LimitSetAndLimitCone}, we discuss the holonomy group of $\Gamma$ in \cref{subsec:TheHolonomyGroup} and we recall the constructions of the BMS and BR measures in \cref{subsec:GeometricMeasures}. We also recall the product structures of the BR measures in \cref{subsec:ProductStructureOfTheBRMeasures}.
	
	In \cref{sec:AnosovSubgroupsAndMaximalFlatCylinders}, we specialize to Anosov subgroups. In \cref{subsec:ErgodicDecompositions}, we recall ergodic decompositions of the BMS and BR measures and in \cref{subsec:TheSupportOfTheBMSMeasureAsAVectorBundle} we recall the vector bundle structure of the support of the BMS measure. In \cref{subsec:LocalMixingTheorem}, we recall the local mixing theorems we will need. Maximal flat cylinders are introduced and discussed in \cref{subsec:MaximalFlatCylinders}. 
	
	In \cref{subsec:CountinginNAMNCoordinates}, we use local mixing to prove \cref{prop:STAsymptotic} which is an asymptotic for the number of elements in $\Gamma$ in product subsets $S_T:=\Xi_1 \mathcal{L}_T^+\Theta\Xi_2$ of $N^+AMN$. In \cref{subsec:CountinginVT}, we use \cref{prop:STAsymptotic} to deduce an asymptotic for $\#\Gamma\cap\mathcal{V}_T(g_0,\varepsilon,\Theta)$ (\cref{prop:CountingInVT}). 
	
	In \cref{sec:ProofOfJointEquidistribution}, we show how the asymptotic for $\#\Gamma\cap\mathcal{V}_T(g_0,\varepsilon,\Theta)$ can be used to deduce an asymptotic for $\mu_T(\tilde{\mathcal{B}}(g_0,\varepsilon)\otimes\Theta)$ (\cref{prop:MuTOfFlowBox}) and prove the main joint equidistribution result \cref{thm:MuTJointEquidistribution} and its corollaries. In \cref{sec:NormAlternate}, we show how to modify our 
	arguments to prove joint equidistribution with respect to norm-like functions (\cref{thm:JointEquidsitributionNVersion}).
	
	\subsection*{Acknowledgements}
	The authors are very grateful to their advisor, Hee Oh, for introducing them to this problem and for many very helpful discussions regarding the problem and the preparation of this paper.
	
	\section{Effective closing lemma for regular directions}
	\label{sec:EffectiveClosingLemmaForRegularDirections}
	
	Let $G$ be a connected semisimple real algebraic group with identity element $e \in G$. Fix a Cartan involution of the Lie algebra $\LieG$ of $G$ and let $\LieG = \LieK \oplus \LieP$ be the associated eigenspace decomposition corresponding to the eigenvalues $+1$ and $-1$ respectively. Let $K < G$ be the maximal compact subgroup whose Lie algebra is $\mathfrak{k}$. Let $\LieA \subset \LieP$ be a maximal abelian subalgebra and choose a closed positive Weyl chamber $\LieA^+ \subset \LieA$. Let $A = \exp \LieA$, $A^+ = \exp \LieA^+$, and denote $a_w = \exp(w)$ for all $w \in \LieA$. Define $M := C_K(A)$ and the contracting and expanding \emph{horospherical subgroups} by
	\begin{align*}
		N:=N^- &:= \left\{n \in G: \lim_{t \to \infty} a_{-tw} n a_{tw} = e \text{ for all } w \in \interior\LieA^+\right\};
		\\
		N^+ &:= \left\{h \in G: \lim_{t \to \infty} a_{tw} h a_{-tw} = e \text{ for all } w \in \interior\LieA^+\right\},
	\end{align*}
	respectively, and their Lie algebras $\LieN:=\LieN^- := \log N$ and $\LieN^+:=\log N^+$. Let 
	\[P:=P^-:=MAN \text{ and } P^+:=MAN^+.\]
	Then $P$ and $P^+$ are opposite minimal parabolic subgroups. Let 
	\[\mathcal{F}: = G/P \cong K/M\]
	denote the \emph{Furstenburg boundary} of $G$, where the isomorphism $G/P \cong K/M$ is given by the Iwasawa decomposition $G \cong K \times A \times N$.
	
	Let $W:=N_K(A)/M$ denote the Weyl group. Let $w_0 \in K$ be a representative of the element in $W$ such that $\Ad_{w_0}(\LieA^+) = -\LieA^+$.	
	
	\begin{definition}[Opposition involution]
		\label{def:OppositionInvolution}
		The map $\involution: \LieA^+ \to \LieA^+$ defined by $\involution(v) := -\Ad_{w_0}(v)$ is called the \emph{opposition involution}. 
	\end{definition}
	
	For all $g \in G$, let
	\begin{align}
		\label{eqn:FurstenbergBoundaryNotation}
		& g^+ := gP \in \Fboundary ;
		& g^- &:= gw_0P \in \Fboundary.
	\end{align}
	
	Fix a left $G$-invariant and right $K$-invariant Riemannian metric $d$ on $G$ and denote the corresponding inner product and norm on any of its tangent spaces by $\langle \cdot, \cdot \rangle$ and $\|\cdot\|$ respectively. The Riemannian metric on $G$ also induces an inner product and norm on $\LieA$, which are invariant under $W$ and we denote by $\langle \cdot, \cdot \rangle$ and $\|\cdot\|$ respectively. Using the inner product on $\LieA$, we identify $\LieA$ with $\R^{\rankG}$ and equip it with the Lebesgue measure which induces a Haar measure on $A$. Fix $\Ad_M$-invariant metrics $\|\cdot\|$ on $\LieN^\pm$. For $\varepsilon >0$, we denote several $\varepsilon$-neighborhoods by
	\begin{align*}
		& G_\varepsilon:= \{g \in G: d_G(e,g) < \varepsilon\}; \quad & G_\varepsilon(g_0) := g_0G_\varepsilon;
		\\
		& N_\varepsilon := \{n_x:=\exp x: x\in\LieN, \|x\| < \varepsilon\};
		\quad & A_\varepsilon := A \cap G_\varepsilon;
		\\
		& N_\varepsilon^+ := \{h_x:=\exp x: x\in\LieN^+, \|x\| < \varepsilon\}; \quad & M_\varepsilon := M \cap G_\varepsilon.
	\end{align*}
	
	Throughout the paper, if $f_1,f_2$ are functions of $\varepsilon>0$ and $T>0$, then we use big-$O$ notation to write
	\begin{equation}
		\label{eqn:bigO}
		f_1 = O(f_2) \iff \limsup_{\substack{\varepsilon \to 0 \\ T \to \infty}} \left|\frac{f_1(\varepsilon,T)}{f_2(\varepsilon,T)}\right| < \infty.
	\end{equation}
	
	\subsection{Flow boxes}
	\label{subsec:FlowBoxes}
	
	In this subsection, we define \emph{flow boxes} following \cite{Mar04} and \cite{MMO14}. Flow boxes will be used in our effective closing lemma for regular directions (\cref{lem:EffectiveClosingLemma}) to describe almost closed $AM$-orbits in $\Gamma \backslash G$.
	
	To motivate the definition of flow boxes, we first recall that the product maps $N \times N^+ \times A \times M \to G$ and $N^+ \times N \times A \times M \to G$ are diffeomorphisms onto Zariski open neighborhoods of $e \in G$. Consequently, we have the following lemma.
	\begin{lemma}
		\label{lem:TransversalityEstimate}
		Let $\varepsilon > 0$ be sufficiently small. For any $0<\varepsilon_1,\varepsilon_2<\varepsilon$, $h\in N^+_{\varepsilon_1}$ and $n\in N_{\varepsilon_2}$, we have $hn = n_1h_1am$ for some unique $h_1 \in N^+$, $n_1 \in N$, $a \in A$ and $m \in M$. Moreover, $h_1\in N^+_{O(\varepsilon_1)}$, $n_1\in N_{O(\varepsilon_2)}$, $a\in A_{O(\varepsilon)}$, and $m\in M_{O(\varepsilon)}$. The corresponding statement if we swap the roles of $N$ and $N^+$ holds as well.
	\end{lemma}
	
	\begin{definition}[$\varepsilon$-flow box at $g_0$]
		\label{def:FlowBox}
		Given $g_0 \in G$ and $\varepsilon > 0$, the \emph{$\varepsilon$-flow box at $g_0$} is defined by \[\mathfrak{B}(g_0,\varepsilon):=g_0(N^+_\varepsilon N \cap N_\varepsilon N^+ AM)M_\varepsilon A_\varepsilon = g_0\mathcal{B}(e,\varepsilon).\]
	\end{definition}
	
	The set of all $\varepsilon$-flow boxes forms a basis for the topology on $G$. The next lemma records elementary properties of the flow boxes.
	
	\begin{lemma}
		\label{lem:FlowBoxProperties}
		Let $g_0 \in G$ and $\varepsilon >0$.
		\begin{enumerate}
			\item 
			\label{itm:FlowBoxProperty1}
			For any $g \in \mathfrak{B}(g_0,\varepsilon)$, $\{w \in \LieA:g\exp(w) \in \mathfrak{B}(g_0,\varepsilon)\}$ is a Euclidean ball of radius $\varepsilon$.
			\item 
			\label{itm:FlowBoxProperty2}
			$\mathfrak{B}(g_0,\varepsilon)e^\pm = g_0N_\varepsilon^\pm e^\pm$ (where $e^\pm$ are defined by \eqref{eqn:FurstenbergBoundaryNotation}). 
			\item 
			\label{itm:FlowBoxProperty3}
			For sufficiently small $\varepsilon > 0$, the following holds. For any $h\in N^+_\varepsilon$ and $n\in N_\varepsilon$,  $h N \cap n N^+ AM$ consists of a single element $g \in \mathfrak{B}(\varepsilon)$. Moreover, by \cref{lem:TransversalityEstimate}, 
			\[\mathfrak{B}(g_0,\varepsilon)=g_0(N^+_\varepsilon N_{O(\varepsilon)} \cap N_\varepsilon N^+_{O(\varepsilon)} A_{O(\varepsilon)}M_{O(\varepsilon)})M_\varepsilon A_\varepsilon.\]
		\end{enumerate}
	\end{lemma}
	
	\subsection{Effective closing lemma for regular directions}
	\label{subsec:EffectiveClosingLemmaForRegularDirections}
	
	Let $\LieA^*$ denote the space of real linear forms on $\LieA$. Let $\Phi \subset \LieA^*$ denote the restricted root system of $\LieA$ and let $\Phi^+ \subset \Phi$ denote the set of positive roots corresponding to $\LieA^+$. A direction $w \in \LieA$ is called \emph{regular} if 
	\[\forall \alpha \in \Phi^+, \, \alpha(w) \ne 0.\]
	
	\cref{lem:EffectiveClosingLemma} is an effective closing lemma for regular directions, which is an adaptation of \cite[Lemma 3.1]{MMO14} which we will be able to use in our Anosov setting later. The proof is similar to that of \cite[Lemma 3.1]{MMO14}, but we include it for completeness. For $g_1,g_2 \in G$, we write 
	\[g_1\sim_{O(\varepsilon)} g_2 \iff d(g_1,g_2) = O(\varepsilon)\]
	and we use the same notation for subgroups of $G$ and their lie algebras, relying on context. 
	
	\begin{lemma}[Effective closing lemma for regular directions] 
		\label{lem:EffectiveClosingLemma}
		There exists $T_0>0$, depending only on $G$, for which the following holds. Let $\varepsilon>0$ be sufficiently small and $g_0\in G$. Suppose there exists $g_1,g_2\in\mathfrak{B}(g_0,\varepsilon)$ and $\gamma\in G$ such that 
		\begin{equation}
			\label{eqn:EffectiveClosingLemmaHypothesis1}
			g_1 \tilde{a}_\gamma\tilde{m}_\gamma = \gamma g_2
		\end{equation}
		for some $\tilde{m}_\gamma\in M$ and $\tilde{a}_\gamma\in A$ with
		\begin{equation}
			\label{eqn:EffectiveClosingLemmaHypothesis2}
			T:=\min_{\alpha \in \Phi^+} \alpha(\log \tilde{a}_\gamma)\ge T_0.
		\end{equation}
		Then there exists $g\in\mathfrak{B}(g_0,\varepsilon + O(\varepsilon e^{-T}))$, $a_\gamma \in A$ and $m_\gamma \in M$ such that 
		\[\gamma=ga_\gamma m_\gamma g^{-1}.\]
		Moreover, $a_\gamma\sim_{O(\varepsilon)} \tilde{a}_\gamma$ and $m_\gamma\sim_{O(\varepsilon)}\tilde{m}_\gamma$.
	\end{lemma}
	
	\begin{proof}
		Let $\varepsilon>0$ be sufficiently small so that each instance that \cref{lem:TransversalityEstimate} and \cref{lem:FlowBoxProperties}\ref{itm:FlowBoxProperty3} are used in the proof is valid. By \cref{lem:FlowBoxProperties}\ref{itm:FlowBoxProperty3}, we have 
		\[g_1=g_0 h_1 n_1 a_1 m_1 \quad \text{ and } \quad g_2=g_0 n_2 h_2 a_2 m_2\]
		for some $h_1\in N^+_\varepsilon$, $n_1 \in N_{O(\varepsilon)}$, $a_1 \in A_\varepsilon$, $m_1 \in M_\varepsilon$, $n_2 \in N_\varepsilon$, $h_2 \in N_{O(\varepsilon)}$, $a_3 \in A_{O(\varepsilon)}$ and $m_3 \in M_{O(\varepsilon)}$. \cref{lem:FlowBoxProperties}\ref{itm:FlowBoxProperty3} also tells us that $h_1 N \cap n_2 N^+ AM =\{g_3\}$, where $g_3 \in \mathfrak{B}(\varepsilon)$ and 
		\[g_3=h_1n_3 = n_2h_3a_3m_3\]
		for some $n_3 \in N_{O(\varepsilon)}$, $h_3 \in N^+_{O(\varepsilon)}$, $a_3 \in A_{O(\varepsilon)}$ and $m_3 \in M_{O(\varepsilon)}$. Set $g_4=g_0 g_3\in\mathfrak{B}(g_0,\varepsilon)$. Then 
		\[g_1=g_4 n_4 a_1 m_1 \quad \text{ and } \quad g_2=g_4 h_4 a_4 m_4,\]
		where $n_4= n_3^{-1}n_1 \in N_{O(\varepsilon)}$, $h_4=(a_3m_3)^{-1}h_3^{-1}h_2(a_3m_3) \in N^+_{O(\varepsilon)}$, $a_4 = a_3^{-1}a_2 \in A_{O(\varepsilon)}$ and $m_4 = m_3^{-1}m_2 \in M_{O(\varepsilon)}$.
		
		The hypothesis \eqref{eqn:EffectiveClosingLemmaHypothesis1} becomes $g_4 n_4 a_1 m_1 \tilde{a}_\gamma\tilde{m}_\gamma=\gamma g_4 h_4 a_4 m_4$. Then
		\[g_4^{-1}\gamma g_4 = n_4 a'_\gamma m'_\gamma(h_4)^{-1} =  a'_\gamma m'_\gamma n_5(h_4)^{-1},\]
		where $a'_\gamma =\tilde{a}_\gamma  a_1a_4^{-1} \sim_{O(\varepsilon)} \tilde{a}_\gamma$, $m'_\gamma = m_1\tilde{m}_\gamma m_4^{-1} \sim_{O(\varepsilon)} \tilde{m}_\gamma$ and $n_5 = (a'_\gamma m'_\gamma)^{-1}n_4(a'_\gamma m'_\gamma) \in N_{O(\varepsilon e^{-T})}$ since $\min_{\alpha \in \Phi^+} \alpha(\log \tilde{a}_\gamma)=T$. 
		
		Using \cref{lem:TransversalityEstimate}, we write $n_5(h_4)^{-1} = a_5m_5h_5n_6$ for some $h_5 \in N^+_{O(\varepsilon)}$, $n_6 \in N_{O(\varepsilon e^{-T})}$, $a_5 \in A_{O(\varepsilon)}$ and $m_5 \in M_{O(\varepsilon)}$. Then 
		\[g_4^{-1}\gamma g_4 = a'_\gamma m'_\gamma a_5m_5h_5n_6 = h_6a''_\gamma m''_\gamma n_6,\]
		where $a''_\gamma = a'_\gamma a_5 \sim_{O(\varepsilon)} a'_\gamma  \sim_{O(\varepsilon)} \tilde{a}_\gamma$, $m''_\gamma = m'_\gamma m_5\sim_{O(\varepsilon)} \tilde{m}_\gamma$ and $h_6 = (a''_\gamma m''_\gamma)n_5(a''_\gamma m''_\gamma)^{-1} \in N^+_{O(\varepsilon e^{-T})}$ since $T=\min_{\alpha \in \Phi^+} \alpha(\log \tilde{a}_\gamma)$.
		
		To complete the proof, it suffices to show 
		\[h_6a''_\gamma m''_\gamma n_6 \in (h_x n_y)a_\gamma''A_{O(\varepsilon)}m_\gamma''M_{O(\varepsilon)}(h_x n_y)^{-1}\]
		for some $h_x \in N^+_{O(\varepsilon e^{-T})}$ and $n_y\in N_{O(\varepsilon e^{-T})}$, as we then take $g=g_4 h_x n_y$.
		
		For each $h_x\in N^+_{\varepsilon}$, there is a unique element $h_{\beta(x)}\in N^+_{O(\varepsilon)}$ such that 
		\begin{equation}
			\label{eqn:EffectiveClosingLemmaProof1}
			(n_6)h_x\in h_{\beta(x)} N_{O(\varepsilon e^{-T})} A_\varepsilon M_\varepsilon.
		\end{equation} 
		Moreover, since the product map $N^+ \times N \times A \times M \to G$ is a diffeomorphism onto its image, $\beta:\LieN^+\to\LieN^+$ is a smooth function. Consider the map 
		\[f:N^+_{\varepsilon} \to N_{\varepsilon + O(\varepsilon e^{-T})}\] 
		given by 
		\[h_x\mapsto h_x (a_\gamma''m_\gamma'')(h_{\beta(x)})^{-1}(a_\gamma''m_\gamma'')^{-1} = h_{x}h_{ \Ad_{a_\gamma''m_\gamma''}(\beta(x))}.\]
		Assuming $T_0$ is sufficiently large, $a''_\gamma \sim_{O(\varepsilon)} \tilde{a}_\gamma$ and hypothesis \eqref{eqn:EffectiveClosingLemmaHypothesis2} implies that $\|Df-I\|_{op}<1$ pointwise on $N^+_\varepsilon$. Then $f$ is injective, and therefore a diffeomorphism onto its image. Since $f$ is a diffeomorphism onto its image, there exists $h_x\in N^+_{O(\varepsilon e^{-T})}$ such that 
		\[h_x(a_\gamma''m_\gamma'')(h_{\beta(x)})^{-1}(a_\gamma''m_\gamma'')^{-1} = h_6.\] 
		Recalling \eqref{eqn:EffectiveClosingLemmaProof1}, we write $(h_{\beta(x)})^{-1}n_6 = a_6m_6n_7h_x^{-1}$ for some $a_6 \in A_\varepsilon$, $m \in M_\varepsilon$ and $n_7 \in N_{O(\varepsilon e^{-T})}$. Then 
		\[h_6a_\gamma''m_\gamma''n_6 = h_x(a_\gamma''m_\gamma'')(h_{\beta(x)})^{-1}n_6 = h_x(a_\gamma'''m_\gamma''')n_7h_x^{-1},\] 
		where $a_\gamma''' = a_\gamma'' a_6$ and $m_\gamma'''=m_\gamma''m_6$. By similar reasoning to the above, there exists $n_y\in N_{O(\varepsilon e^{-T})}$ such that $n_7 = (a_\gamma'''m_\gamma''')^{-1}n_y(a_\gamma'''m_\gamma''')n_y^{-1}$ and this completes the proof.
	\end{proof}
	
	\section{Geometric measures}
	\label{sec:GeometricMeasures}
	
	Henceforth, let $\Gamma < G$ be a Zariski dense discrete subgroup.
	
	\subsection{Limit set and limit cone}
	\label{subsec:LimitSetAndLimitCone}
	
	Recall that $\Fboundary :=G/P \cong K/M$. Let $m_\Fboundary$ denote the unique $K$-invariant probability measure on $\Fboundary$.
	
	\begin{definition}[Limit set]
		The \emph{limit set} $\limitset \subset \Fboundary$ of $\Gamma$ is defined by
		\begin{equation*}
			\limitset := \{\xi \in \Fboundary : \exists \{\gamma_n\}_{n \in \N} \subset \Gamma, (\gamma_n)_*m_\Fboundary \xrightarrow{n \to \infty} \delta_\xi\},
		\end{equation*}
		where $\delta_\xi$ denotes the Dirac measure at $\xi$. The limit set is the unique minimal nonempty closed $\Gamma$-invariant subset of $\Fboundary$ \cite{Ben97} and $\limitset$ is Zariski dense in $\Fboundary$ \cite[Section 3.6]{Ben97}.
	\end{definition}
	
	An element $\gamma \in G$ is called \emph{loxodromic} if 
	\[\gamma = g\exp{\lambda(\gamma)}mg^{-1}\]
	for some $g \in G$, $\lambda(\gamma)\in \interior\LieA^+$ and $m \in M$. In that case, 
	\[\lambda(\gamma) \in \interior\LieA^+\]
	is unique and called the \emph{Jordan projection} of $\gamma$ and $m$ belongs to a unique conjugacy class 
	\[h_\gamma = [m] \in [M]\]
	called the \emph{holonomy} of $\gamma$. In addition, $g^+$ and $g^-$ defined by \eqref{eqn:FurstenbergBoundaryNotation} are the unique attracting and repelling fixed points of $\gamma$, respectively. Moreover, if $\gamma \in \Gamma$, then $g^+,g^- \in \limitset$. 
	
	\begin{definition}[Limit cone]
		\label{def:LimitCone}
		The \emph{limit cone} $\limitcone \subset \LieA^+$ of $\Gamma$ is the smallest closed cone containing all Jordan projections of loxodromic elements in $\Gamma$ \cite[Sections 4 and 4.3]{Ben97}. We denote the interior of $\limitcone$ by $\interior\limitcone$.
	\end{definition}
	
	\begin{remark}
		Recall the opposition involution $\involution$ (\cref{def:OppositionInvolution}). We have $\lambda(g^{-1}) = \involution(\lambda(g))$ and hence, $\limitcone = \involution(\limitcone)$.
	\end{remark}
	
	\begin{theorem}[{\citealp{Ben97}}]
		The limit cone $\limitcone$ is convex and $\interior\limitcone \ne \emptyset$.    
	\end{theorem}
	
	\subsection{The holonomy group}
	\label{subsec:TheHolonomyGroup}
	
	Let $M^\circ$ denote the identity component of $M$. Then the identity component of $P$ is $P^\circ = M^\circ AN$.
	
	\begin{definition}[Holonomy group of $\Gamma$]
		\label{def:HolonomyGroup}
		The \emph{holonomy group} of $\Gamma$ is the closed subgroup $M_\Gamma < M$ generated by all of the holonomies in $\Gamma$, that is,
		\[M_\Gamma := \overline{\langle \{m: \exists \gamma \in \Gamma \text{ such that } h_\gamma = [m] \} \rangle}.\]
	\end{definition}
	By \cite[Corollary 1.10]{GR07}, $M_\Gamma$ is a normal subgroup of $M$ containing $M^\circ$ so in particular, $M_\Gamma$ has finite index in $M$. 
	
	By \cite{GR07}, we have another characterization of $M_\Gamma$. Let 
	\[\Fboundary^\circ := G/P^\circ.\]
	Fix a $\Gamma$-minimal subset $\limitset_0 \subset \Fboundary^\circ$. Then 
	\begin{equation}
		\label{eqn:HolonomyGroupDefinition2}
		M_\Gamma = \{m \in M : \Lambda_0m = \Lambda_0\} < M.
	\end{equation} 
	See also \cite[Proposition 4.9]{BQ14} for another characterization of $M_\Gamma$.
	
	\subsection{Busemann function}
	\label{subsec:BusemannFunction}
	
	The \emph{Iwasawa cocycle} $\sigma: G \times \Fboundary \to \LieA$ is the map which assigns to each $(g,kM) \in G \times \Fboundary$ the unique element $\sigma(g,kM) \in \LieA$ such that $gk \in Ka_{\sigma(g,\xi)}N$. It satisfies the cocycle relation $\sigma(g_1g_2,\xi) = \sigma(g_1,g_2\xi) + \sigma(g_2,\xi)$ for all $g_1,g_2 \in G$ and $\xi \in \Fboundary$.
	
	\begin{definition}[Busemann function]
		\label{def:BusemannFunction}
		The \emph{$\LieA$-valued Busemann function} $\beta: \Fboundary \times G \times G \to \LieA$ is defined by
		\begin{equation*}
			\beta_\xi(g_1,g_2) := \sigma(g_1^{-1}, \xi) - \sigma(g_2^{-1}, \xi)
		\end{equation*}
		for all $g_1,g_2 \in G$ and $\xi \in \Fboundary$. The properties of the Iwasawa cocycle imply that the Busemann function satisfies
		\begin{enumerate}
			\item $\beta_\xi(e, g) = -\sigma(g^{-1}, \xi)$;
			\item $\beta_{g\xi}(gg_1,gg_2) = \beta_\xi(g_1,g_2)$;
			\item $\beta_\xi(g_1,g_2) = \beta_\xi(g_1,g) + \beta_\xi(g,g_2)$
		\end{enumerate}
		for all $g,g_1,g_2 \in G$ and $\xi \in \Fboundary$.
	\end{definition}
	
	\subsection{Conformal measures} 
	\label{subsec:ConformalMeasures}
	
	In this subsection, we recall facts about $\Gamma$-conformal measures. Patterson \cite{Pat76} and Sullivan \cite{Sul79} were the first to introduce $\Gamma$-conformal measures in rank one groups. Their work has been extended to higher rank groups by Albuquerque \cite{Alb99} and Quint \cite{Qui02b}. Generalizing the notion of critical exponent in rank one, Quint introduced the growth indicator function of $\Gamma$ \cite{Qui02a}. 
	
	For $g \in G$, let  $\mu(g)$ denote the \emph{Cartan projection} of $g$, that is, $\mu(g) \in \LieA^+$ is the unique element in $\LieA^+$ such that $g \in K\exp(\mu(g))K$. We note that $\mu(g^{-1}) = \involution(\mu(g))$ for all $g \in G$. 
	
	\begin{definition}[Growth indicator function]
		\label{def:GrowthIndicatorFunction}
		The \emph{growth indicator function} $\growthindicator : \LieA^+ \to \R \cup \{-\infty\}$ of $\Gamma$ is the degree 1 homogeneous function defined by
		\begin{equation*}
			\growthindicator(w) := \|w\| \inf_{\text{open cones }\mathfrak{C}\ni w} \tau_\mathfrak{C} \qquad \text{for all $w \in \LieA^+$},
		\end{equation*}
		where $\tau_\mathfrak{C}$ is the abscissa of convergence of $t \mapsto \sum_{\gamma \in \Gamma, \mu(\gamma) \in \mathfrak{C}} e^{-t\|\mu(\gamma)\|}$.
	\end{definition}
	
	We note that $\growthindicator \circ \involution = \growthindicator$.
	
	\begin{theorem}[{\citealp[Theorem 4.2.2]{Qui02a}}]
		The growth indicator function is concave, upper semicontinuous, and satisfies $\growthindicator|_{\LieA^+\setminus\limitcone} = -\infty$, $\growthindicator|_{\limitcone} \ge 0$, and $\growthindicator|_{\interior\limitcone} > 0$.
	\end{theorem}
	
	\begin{definition}[Conformal measures]
		\label{def:ConformalMeasures}
		Given a closed subgroup $\Delta < G$, a Borel probability measure $\nu$ on $\Fboundary$ is called a \emph{$\Delta$-conformal measure} if, there exists $\psi \in \LieA^*$ such that for any $\gamma \in \Delta$ and $\xi \in \Fboundary$, 
		\[\frac{d\gamma_*\nu}{d\nu}(\xi) = e^{\psi(\beta_\xi(e,\gamma))},\]
		where $\gamma_*\nu(Q) := \nu(\gamma^{-1}Q)$ for any Borel subset $Q \subset \Fboundary$. In that case, we call $\nu$ a \emph{$(\Delta,\psi)$-conformal measure}.
	\end{definition}
	
	The measure $m_\Fboundary$ is a conformal measure. More precisely, let $\rho \in \LieA^*$ denote the half sum of the positive roots of $\LieA^+$, 
	\begin{equation}
		\label{eqn:Rho}
		\rho := \frac{1}{2}\sum_{\alpha \in \Phi^+}\alpha.
	\end{equation}
	We have the following lemma.
	
	\begin{lemma}[{\citealp[Proposition 3.3]{Qui06}}]
		\label{lem:KInvariantMeasureOnFboundary}
		The measure $m_\Fboundary$ is a $(G,2\rho)$-conformal measure.
	\end{lemma}
	
	The following theorem on the existence of $\Gamma$-conformal measures is due to Quint. A linear form $\psi \in \LieA^*$ is said to be a \emph{tangent form} if 
	\begin{equation}
		\label{eqn:TangentFormDefinition}
		\psi \geq \growthindicator \text{ and } \psi(\mathsf{v}) = \growthindicator(\mathsf{v}) \text{ for some } \mathsf{v} \in \interior\LieA^+ \cap \limitcone.
	\end{equation} 
	Note that if $\psi \in \LieA^*$ is a tangent form, then $\psi \circ \involution$ is also a tangent form.
	
	\begin{theorem}[{\citealp[Theorems 8.1 and 8.4]{Qui02b}}]
		\label{thm:ExistenceOfConformalMeasures}
		Let $\psi \in \LieA^*$. 
		\begin{enumerate}
			\item If there exists a $(\Gamma,\psi)$-conformal measure, then $\psi \ge \growthindicator$. 
			
			\item If $\psi$ is a tangent form, then there exists a $(\Gamma,\psi)$-conformal measure $\nu_\psi$ supported on $\limitset$.
		\end{enumerate}
	\end{theorem}
	
	\subsection{Geometric measures}
	\label{subsec:GeometricMeasures}
	
	In this subsection, we recall the definitions of the Bowen--Margulis--Sullivan and Burger--Roblin measures. 
	
	We first recall the Hopf parametrization of $G/M$ which will be used to define these measures. There is a unique open $G$-orbit in $\Fboundary \times \Fboundary$ given by
	\begin{equation}
		\label{eqn:FurstenbergBoundary2}
		\Fboundary^{(2)} := G\cdot(e^+,e^-) = (\Fboundary \times \Fboundary) \setminus \bigcup_{wM \notin W \setminus \{w_0M\}}G\cdot(e^+,we^+),
	\end{equation} 
	where the Weyl group $W$ and the element $w_0 \in K$ were defined in \cref{sec:EffectiveClosingLemmaForRegularDirections}. If $(x,y) \in \Fboundary^{(2)}$, then we say that $x$ and $y$ are in \emph{general position}. 
	
	Define a left $G$-action on $\Fboundary^{(2)} \times \LieA$ by
	\begin{equation*}
		g \cdot (x,y,v) := (gx,gy,v + \beta_x(g^{-1}, e)) = (gx,gy,v + \sigma(g,x))
	\end{equation*}
	for all $g \in G$ and $(x, y, v) \in \Fboundary^{(2)} \times \LieA$. Note that $\Stab_G(e^+,e^-,0) = M$.
	
	\begin{definition}[Hopf parametrization]
		\label{def:HopfParametrization}
		The \emph{Hopf parametrization} is a left $G$-equivariant diffeomorphism $G/M \to \Fboundary^{(2)} \times \LieA$ defined by
		\begin{equation*}
			gM \mapsto (g^+, g^-, \beta_{g^+}(e, g)) = (g^+, g^-, \sigma(g, e^+)).
		\end{equation*}
	\end{definition}
	
	For the remainder of \cref{sec:GeometricMeasures}, we fix a 
	\begin{equation}
		\label{eqn:FixTangentForm}
		\text{tangent form } \psi \in \LieA^*.
	\end{equation}
	Recall that $\psi \circ \involution$ is also a tangent form. Using \cref{thm:ExistenceOfConformalMeasures}, we fix
	\begin{equation}
		\label{eqn:FixConformalMeasures}
		(\Gamma,\psi)\text{-and } (\Gamma,\psi\circ\involution)\text{-conformal measures }\nu_\psi \text{ and } \nu_{\psi\circ\involution},
	\end{equation} 
	respectively. Let 
	\[\limitset^{(2)} := (\limitset \times \limitset) \cap \Fboundary^{(2)}\] 
	and let $dw$ denote the Lebesgue measure on $\LieA$.
	
	\begin{definition}[Bowen--Margulis--Sullivan measure]
		\label{def:BMSMeasures}
		Using the Hopf para\-metrization, the \emph{$(\nu_\psi,\nu_{\psi\circ\involution})$-Bowen--Margulis--Sullivan (BMS) measure} $\BMS:=m^{\text{BMS}}_{\nu_\psi,\nu_{\psi\circ\involution}}$ is defined on $G/M \cong \Fboundary^{(2)} \times \LieA$ by 
		\begin{equation}
			\label{eqn:BMSMeasureDefinition}
			dm^{\text{BMS}}_{\nu_\psi,\nu_{\psi\circ\involution}}(gM) := e^{\psi\left(\beta_{g^+}(e,g)\right) + (\psi \circ \involution)\left(\beta_{g^-}(e,g)\right)} \, d\nu_\psi(g^+) \, d\nu_{\psi \circ\involution}(g^-) \, dw.
		\end{equation}
		It is clear from \eqref{eqn:BMSMeasureDefinition} that $\BMS$ has support
		\begin{equation}
			\label{eqn:SupportOfBMSMeasure}
			\supp \BMS = \limitset^{(2)} \times \LieA.
		\end{equation}
		The BMS measure is left $\Gamma$-invariant and right $A$-invariant so it descends to a measure on $\Gamma \backslash G/M$ and by lifting using the Haar probability measure on $M$ we also obtain a measure on $\Gamma \backslash G$. Abusing notation, we call also this measure BMS measure and denote it by $\BMS$ as well, relying on context for the domain. The supports of the BMS measures on these domains are given by
		\begin{align}
			\label{eqn:Omega}
			& \tilde{\Omega} :=\{\Gamma g : g^\pm \in \limitset\} \subset \Gamma \backslash G;
			&& \Omega := \tilde{\Omega}/M \subset \Gamma \backslash G/M.
		\end{align}
	\end{definition}  
	
	\begin{definition}[Burger--Roblin measures]
		\label{def:BurgerRoblinMeasure}
		Following \cite[Section 3]{ELO20}, the $(m_\Fboundary,\nu_\psi)$- and $(\nu_{\psi\circ\involution},m_\Fboundary)$-\emph{Burger--Roblin (BR) measures} 
		\[\BR:=m^{\text{BR}}_{m_\Fboundary,\nu_{\psi\circ\involution}} \text{ and } \BRstar:=m^{\text{BR}_*}_{\nu_{\psi},m_\Fboundary},\] 
		respectively, are defined on $G/M$ in a similar fashion as $\BMS$ by 
		\begin{align}
			\label{eqn:BRMeasuresDefinition1}
			dm^{\text{BR}}_{m_\Fboundary,\nu_{\psi\circ\involution}}(gM) &:= e^{2\rho\left(\beta_{g^+}(e,g)\right) + (\psi \circ \involution)\left(\beta_{g^-}(e,g)\right)} \, dm_\Fboundary(g^+) \, d\nu_{\psi \circ\involution}(g^-) \, dw; 
			\\
			\label{eqn:BRMeasuresDefinition2}
			dm^{\text{BR}_*}_{\nu_{\psi},m_\Fboundary}(gM) &:= e^{\psi\left(\beta_{g^+}(e,g)\right) + 2\rho\left(\beta_{g^-}(e,g)\right)} \, d\nu_\psi(g^+) \, dm_\Fboundary(g^-) \, dw.
		\end{align}
		
		Then $\BR$ is a left $\Gamma$-invariant and right $N^+$-invariant measure and $\BRstar$ is a left $\Gamma$-invariant and right $N$-invariant measure and they induce measures on $\Gamma \backslash G / M$, $\Gamma \backslash G$ and $G$ that we use the same name and notation for. The supports of the BR measures on $\Gamma \backslash G$ are given by 
		\begin{align*}
			&\mathcal{E} := \supp\BR = \{ \Gamma g: g^- \in \limitset\}; &
			\mathcal{E}_* := \supp\BRstar = \{ \Gamma g: g^+ \in \limitset\}.
		\end{align*}
	\end{definition}  
	
	\begin{definition}
		\label{def:HaarMeasure}
		The \emph{right Haar measure} $dg$ on $G$ that we use can be written as 
		\begin{equation*}
			dg = e^{2\rho\left(\beta_{g^+}(e,g) + \beta_{g^-}(e,g)\right)} \, dm_\Fboundary(g^+) \, dm_\Fboundary(g^-) \, dw \, dm,
		\end{equation*}
		where $dm$ is the probability Haar measure on $M$ \cite[Section 3]{ELO20}. 
	\end{definition}
	
	\subsection{Product structure of the BR measures}
	\label{subsec:ProductStructureOfTheBRMeasures}
	
	In this subsection, we recall a product structure of the BR measures (see \cite[Section 4.3]{ELO20} for details) and define some related measures that will be useful in \cref{sec:Counting}. We retain the notations in \eqref{eqn:FixTangentForm} and \eqref{eqn:FixConformalMeasures}. 
	
	Let $dm$ denote the probability Haar measure on $M$. We define measures $\tilde{\nu}_\psi$ and $\nu_{\psi\circ\involution}$ on $K\cong\Fboundary\times M$ using the $M$-invariant lifts of the conformal measures by
	\begin{align}
		\label{eqn:LiftsOfConformalMeasuresToK}
		& d\nu^K_\psi(k) :=d\nu_\psi(k^+) \, dm;  
		& d \nu^K_{\psi\circ\involution}(k) := d\nu_{\psi\circ\involution}(k^-) \, dm.
	\end{align}
	
	The following lemma gives the decompositions of $\BR$ and $\BRstar$ in $KAN$-coordinates.
	
	\begin{lemma}[{\citealp[Lemma 4.9]{ELO20}}]
		\label{lem:BurgerRoblinMeasuresInKANCoordinates}
		For all $k_1,k_2 \in K$, $w_1,w_2 \in \LieA$, $h\in N^+$ and $n \in N$, we have 
		\begin{align*}
			d\BR(k_1\exp(w_1)h) & = e^{-\psi(w_1)} \, d\nu^K_{\psi\circ\involution}(k_1)  \, dw_1 \, dh;
			\\
			d\BRstar(k_2\exp(w_2)n) & = e^{\psi(w_2)} \, d\nu^K_{\psi}(k_2) \, dw_2 \, dn,
		\end{align*}
		where $dh$ and $dn$ are some Haar measures on $N^+$ and $N$, respectively. 
	\end{lemma}
	
	\begin{remark}
		\label{rem:HaarMeasureDecomposition}
		Recalling that the product maps $M\times A\times N\times N^+ \to MANN^+$, $A \times N^+ \times M \times N \to AN^+MN$ and $N^+\times M \times A \times N \to G$ are homeomorphisms onto Zariski open subsets of $G$, using the notation in \cref{lem:BurgerRoblinMeasuresInKANCoordinates}, we have the decomposition 
		\begin{equation}
			\label{eqn:HaarMeasureDecomposition}
			dg = dm \, dw \, dn \, dh = dw \, dh \, dm \, dn = e^{2\rho(w)} \, dh \, dw \, dm \, dn.
		\end{equation} 
	\end{remark}
	
	For $g \in G$, we define the measures $\nu_\psi^g$ and $\nu_{\psi\circ\involution}^g$ on $\Fboundary$ by  
	\begin{align}
		\label{eqn:TwistedConformalMeasuresDefinition}
		& d\nu_\psi^g(\xi) : = e^{\psi(\beta_{g\xi}(e,g))} \, d\nu_\psi(g\xi); 
		& d\nu_{\psi\circ\involution}^{g}(\xi) : = e^{(\psi\circ\involution)(\beta_{g\xi}(e,g))} \, d\nu_{\psi\circ\involution}(g\xi).
	\end{align}
	Note that $\nu_\psi^g$ and $\nu_{\psi\circ\involution}^g$ are $(g^{-1}\Gamma g, \psi)$- and $(g^{-1}\Gamma g, \psi\circ\involution)$-conformal measures, respectively. 
	
	In addition, identifying $N^+$ with $N^+e^+ \subset \mathcal{F}$, for all $g \in G$, we define the measures $\tilde{\nu}_{\psi}^g$ and $\tilde{\nu}_{\psi}^g|_B$ on $N^+$ by
	\begin{align}
		\label{eqn:LiftsOfConformalMeasuresToN+}
		& d\tilde{\nu}_{\psi}^g(h) :=e^{\psi(\beta_{h^+}(e,h))}  \, d\nu_\psi^g(h^+) ;
		& d\tilde{\nu}_{\psi}^g|_B(h) :=\mathbbm{1}_B(h)d\tilde{\nu}_{\psi}^g(h),
	\end{align} 
	where $B$ is a measurable subset of $G$.
	
	Identifying $NM$ with $Ne^-\times M \subset \mathcal{F} \times M$, for all $g \in G$, we also define the measures $\tilde{\nu}_{\psi\circ\involution}^g$ and $\tilde{\nu}_{\psi\circ\involution}^g|_{B}$ on $NM$ by
	\begin{align}
		\label{eqn:LiftsOfConformalMeasuresToNM1}
		& d\tilde{\nu}_{\psi\circ\involution}^g(nm) := e^{(\psi\circ\involution)(\beta_{n^-}(e,n))}\, d\nu_{\psi\circ\involution}^g(n^-) \, dm ;
		\\
		\label{eqn:LiftsOfConformalMeasuresToNM2}
		& d\tilde{\nu}_{\psi\circ\involution}^g|_B(nm) := \mathbbm{1}_B(nm)d\tilde{\nu}_{\psi\circ\involution}^g(nm).
	\end{align}
	
	\section{Anosov subgroups and maximal flat cylinders}
	\label{sec:AnosovSubgroupsAndMaximalFlatCylinders}
	The Anosov property (with respect to any parabolic subgroup of $G$) was first introduced by Labourie \cite{Lab06} for surface groups and later generalized by Guichard--Wienhard \cite{GW12} for Gromov hyperbolic groups (cf. \cite{KLP17,GGKW17,Wie18}). For Zariski dense Anosov subgroups with respect to $P$, \cite[Corollary 4.16]{GW12} gives an equivalent characterization of being Anosov which we take as the definition of Anosov we use throughout the paper.
	
	\begin{definition}[Anosov subgroup]
		\label{def:AnosovSubgroup}
		Let $\Gamma < G$ be a Zariski dense discrete subgroup. We say that $\Gamma$ is \emph{Anosov} if it is a finitely generated Gromov hyperbolic group and it admits a continuous $\Gamma$-equivariant continuous embedding from the Gromov boundary of $\Gamma$ to $\Fboundary$ such that for any two distinct points in the Gromov boundary, their images in $\Fboundary$ are in general position. 
	\end{definition}
	
	For the rest of the paper, let $\Gamma$ be a torsion-free Zariski dense Anosov subgroup. It follows from \cref{def:AnosovSubgroup} that 
	\begin{equation}
		\label{eqn:LimitSet2}
		\limitset^{(2)} := (\limitset \times \limitset) \cap \Fboundary^{(2)} = \{(x, y) \in \limitset \times \limitset: x \neq y\}.
	\end{equation}
	
	The following theorem was proved by \cite[Proposition 3.2 and Theorem 4.7]{Qui03} when $\Gamma$ is a Schottky subgroup. In general, \cref{thm:AnosovSubgroupsPSTheoryProperties} follows from \cite[Lemma 3.1]{GW12}, \cite[Corollaries 3.12, 3.13, and 4.9]{Sam14a} and \cite[Theorem 4.20]{Sam15} in light of \cite{BCLS15} using the Pl\"{u}cker representation (see also \cite[Propositions 4.6 and 4.11]{PS17}).
	
	\begin{theorem}
		\label{thm:AnosovSubgroupsPSTheoryProperties}
		The following holds.
		\begin{enumerate}
			\item
			\label{itm:NonTrivialElementsInGammaLoxodromic} 
			Every nontrivial element in $\Gamma$ is loxodromic. 
			
			\item
			\label{itm:LimitConeInInteriorLieA+} 
			The limit cone of $\Gamma$ is contained in $\interior\LieA^+ \cup \{0\}$.
			
			\item
			\label{itm:ConcaveAnalyticOnLimitCone} 
			On $\interior\limitcone$, $\growthindicator$ is analytic and strictly concave except along rays emanating from the origin.
			
			\item
			\label{itm:TagentAtAInteriorVector} 
			If $\psi \in \LieA^*$ is tangent to $\growthindicator$ at $v \in \interior \LieA^+ \cap \limitcone$, then $v \in \interior \limitcone$ and $\psi$ is positive on $\limitcone \setminus \{0\}$.
			
			\item
			\label{itm:PsiCriticalExponent} 
			If $\psi \in \LieA^*$ is positive on $\limitcone \setminus \{0\}$, then $\delta_\psi \psi$ is tangent to $\growthindicator$ at some $v \in \interior \limitcone$, where $\delta_\psi>0$ is the $\psi$-critical exponent, that is, the abscissa of convergence of the $\psi$-Poincar\'e series $t \mapsto \sum_{\gamma \in \Gamma} e^{-t\psi(\mu(\gamma))}$.
		\end{enumerate}
	\end{theorem}
	
	\begin{remark}
		Except for the requirement that $v\in\interior\limitcone$, \cref{thm:AnosovSubgroupsPSTheoryProperties}\ref{itm:PsiCriticalExponent} holds for any Zariski dense discrete subgroup as shown by Kim--Minsky--Oh \cite[Section 2]{KMO21}.
	\end{remark}
	
	We fix some notation for the rest of the paper. Fix a linear form $\psi$ tangent to $\growthindicator$ at a normalized direction $\mathsf{v} \in \interior\limitcone$, that is, 
	\[\psi \in \LieA^*, \, \psi \ge \growthindicator, \, \mathsf{v} \in \interior \limitcone \text{ and } \psi(\mathsf{v}) = \growthindicator(\mathsf{v}) = 1.\] 
	By \cref{thm:AnosovSubgroupsPSTheoryProperties}, there is a unique such $\mathsf{v}$ for each tangent form $\psi$. Using \cref{thm:ExistenceOfConformalMeasures}, we fix
	\begin{equation*}
		(\Gamma,\psi)\text{-and } (\Gamma,\psi\circ\involution)\text{-conformal measures }\nu_\psi \text{ and } \nu_{\psi\circ\involution},
	\end{equation*} 
	respectively. We remark that $\nu_\psi$ and $\nu_{\psi\circ\involution}$ are unique by \cite[Theorem 1.2]{LO22}, but we do not need to use this uniqueness in this paper. Let $\BMS$, $\BR$, $\BRstar$ denote the $(\nu_\psi,\nu_{\psi\circ\involution})$-BMS measure, $(m_\Fboundary,\nu_\psi)$- and $(\nu_{\psi\circ\involution},m_\Fboundary)$-BR measures, respectively, as defined in \cref{subsec:GeometricMeasures}. We will use the related notation from \cref{subsec:GeometricMeasures}.
	
	\subsection{\texorpdfstring{Ergodic decompositions of $\BMS$, $\BR$ and $\BRstar$}{Ergodic decompositions of the BMS measure and the BR measures}}
	\label{subsec:ErgodicDecompositions}
	
	In this subsection, we recall the $A$-, $N$-, and $N^+$-ergodic decompositions of $\BMS$, $\BR$ and $\BRstar$, respectively, due to Lee--Oh \cite[Theorem 1.1]{LO20a}. 
	
	Recall that $P^\circ$ denotes the connected component of the identity in $P$. Denote the set of $P^\circ$-minimal subsets of $\Gamma \backslash G$ by $\mathfrak{Y}_\Gamma$. Every $Y \in \mathfrak{Y}_\Gamma$ satisfies $Y = (Y\cap \tilde{\Omega})N$. Let $\mathfrak{Z}_\Gamma$ denote the set of all intersections of $P^\circ$-minimal subsets of $\Gamma \backslash G$ with $\tilde{\Omega}$, that is, 
	\begin{equation}
		\label{eqn:ErgodicComponents}
		\mathfrak{Z}_\Gamma := \{Y \cap \tilde{\Omega} : Y \in \mathfrak{Y}_\Gamma\}.
	\end{equation} 
	
	\begin{theorem}[{\citealp[Theorem 1.1]{LO20a}}]
		\label{thm:ErgodicDecompositions}
		Consider the measures $\BMS,\BR$ and $\BRstar$ on $\Gamma \backslash G$. We have 
		\begin{itemize}
			\item[(1)] the $A$-ergodic decomposition 
			\[\BMS = \sum_{Z \in \mathfrak{Z}_\Gamma} \BMS|_Z = \sum_{Z \in \mathfrak{Z}_\Gamma} \BMS|_Z;\]
			
			\item[(2)] the $N$-ergodic decomposition $\BRstar = \sum_{Z \in \mathfrak{Z}_\Gamma} \BRstar|_{ZN}$;
			
			\item[(3)] the $N^+$-ergodic decomposition $\BR = \sum_{Z \in \mathfrak{Z}_\Gamma} \BR|_{ZN^+}$.
		\end{itemize}
	\end{theorem}
	
	\subsection{\texorpdfstring{The support $\Omega$ of $\BMS$ as a vector bundle}{The support of the BMS measure as a vector bundle}}
	\label{subsec:TheSupportOfTheBMSMeasureAsAVectorBundle}
	For this subsection, we refer the reader to \cite[Section 4]{LO20b} and \cite[Appendix A]{Car21} for more details. The map $\pi_\psi: \limitset^{(2)} \times \LieA \to \limitset^{(2)} \times \R$ defined by 
	\begin{equation*}
		\pi_\psi(x, y, w) := (x, y, \psi(w)) \quad \text{for all $(x, y, w) \in \limitset^{(2)} \times \LieA$}
	\end{equation*}
	is a vector bundle with typical fiber $\ker\psi$.
	Note that $\Gamma$ acts on $\limitset^{(2)} \times \LieA$ on the left, via the Hopf parametrization. Note also that $\Gamma$ acts on $\limitset^{(2)} \times \R$ on the left by
	\begin{equation*}
		\gamma \cdot (x, y, t) := (\gamma x, \gamma y, t + \psi(\beta_x(\gamma^{-1}, e)))
	\end{equation*}
	for all $\gamma \in \Gamma$, $(x, y, t) \in \limitset^{(2)} \times \R$. 
	\begin{theorem}[{\cite[Proposition A.1]{Car21}}, see also {\cite[Theorem 4.15]{CS23}}]
		\label{thm:GammaActionIsNice}
		The left $\Gamma$-action on $\limitset^{(2)} \times \R$ is properly discontinuous and cocompact. 
	\end{theorem}
	
	Let 
	\[\mathcal{X}_\psi := \Gamma \backslash (\limitset^{(2)} \times \R).\]
	By \cref{thm:GammaActionIsNice}, $\mathcal{X}_\psi$ is a compact Hausdorff topological space. The map $\pi_\psi$ is $\Gamma$-equivariant and descends to a map $\pi_\psi:\Omega \to \mathcal{X}_\psi$ which is in fact, a trivial $\ker\psi$-vector bundle. We embed
	\begin{equation}
		\label{eqn:EmbeddingOfX}
		\mathcal{X}_\psi \cong \mathcal{X}_\psi \times \{0\} \subset \mathcal{X}_\psi \times \ker\psi \cong \Omega.
	\end{equation}
	
	Define the locally finite Borel measure $m_{\limitset^{(2)} \times \R}$ on $\limitset^{(2)} \times \R$ by
	\[dm_{\limitset^{(2)} \times \R}(\xi, \eta, t) := e^{\psi(\beta_{\xi}(e, g))+\psi(\involution(\beta_{\eta}(e, g)))} \, d\nu_\psi(\xi) \, d\nu_{\psi \circ \involution}(\eta) \, dt,\]
	where $g \in G$ is any element with $g^+ = \xi$ and $g^- = \eta$ and $dt$ denotes the Lebesgue measure on $\R$. Note that $m_{\limitset^{(2)} \times \R}$ is left $\Gamma$-invariant, so $m_{\limitset^{(2)} \times \R}$ descends to a finite measure $m_{\mathcal{X}_\psi}$ on $\mathcal{X}_\psi$. We have
	\begin{equation}
		\label{eqn:BMSMeasureAndXMeasure}
		d\BMS\bigr|_\Omega = dm_{\mathcal{X}_\psi} \, du,
	\end{equation}
	where $du$ denotes the Lebesgue measure on $\ker\psi$ which satisfies $dw = dt \, du$, where $w = t\mathsf{v} + u$, $u \in \ker\psi$, $dw$ and $dt$ are the Lebesgue measures on $\LieA$ and $\R$, respectively.
	
	The set $\limitset^{(2)} \times \R$ is also equipped with a natural flow $\Phi_t: \limitset^{(2)} \times \R \to \limitset^{(2)} \times \R$ defined by 
	\[\Phi_t(x,y,s) := (x,y,s+t)\] 
	for all $(x,y,s) \in \limitset^{(2)} \times \R$. The flow $\Phi$ descends to a flow on $\mathcal{X}_\psi$ which we call the \emph{translation flow} and we also denote by $\Phi$.
	
	\subsection{Local mixing}
	\label{subsec:LocalMixingTheorem}
	We recall the local mixing theorem for the Haar measure on $\Gamma \backslash G$ which will be used in \cref{sec:Counting}. Let $dx$ denote the right $G$-invariant measure on $\Gamma \backslash G$ induced by the Haar measure on $G$. Given an inner product $\langle\cdot,\cdot\rangle_*$ on $\LieA$, let $I: \ker\psi \to \R$ be defined by 
	\begin{equation}
		\label{eqn:IDefinition}
		I(u) := \langle u, u\rangle_* - \frac{\langle u, \mathsf{v} \rangle_*^2}{\langle \mathsf{v}, \mathsf{v}\rangle_*} \text{ for all }u \in \ker\psi.
	\end{equation} 
	
	\begin{theorem}[{\cite[Theorem 1.3]{CS23}}, {\cite[Theorem 3.4]{ELO22b}}]
		\label{thm:DecayofMatrixCoefficients}
		There exists $\kappa_{\mathsf{v}} >0$ and an inner product $\langle \cdot, \cdot \rangle_*$ on $\LieA$ such that for any $u \in \ker\psi$ and $\phi_1, \phi_2 \in C_{\mathrm{c}}(\Gamma \backslash G)$, we have
		\begin{multline*}
			\lim_{t \to +\infty} t^{\frac{\rankG - 1}{2}}e^{(2\rho - \psi)(t\mathsf{v} + \sqrt{t}u)} \int_{\Gamma \backslash G} 	\phi_1(x\exp(t\mathsf{v} + \sqrt{t}u)) \phi_2(x) \, dx \\
			=\frac{\kappa_{\mathsf{v}}e^{-I(u)}}{|m_{\mathcal{X}_\psi}|}  \sum_{Z \in \mathfrak{Z}_\Gamma} 	\BR\bigr|_{ZN^+}(\phi_1)\cdot 	\BRstar\bigr|_{ZN}(\phi_2),
		\end{multline*}
		where $\rho$ is given by \eqref{eqn:Rho} and $\mathfrak{Z}_\Gamma$ is given by \eqref{eqn:ErgodicComponents}.
		
		Moreover, there exist positive constants $\eta_{\mathsf{v}}$ and $T_{\mathsf{v}}$ such that for all $\phi_1, \phi_2 \in C_{\mathrm{c}}(\Gamma \backslash G)$, there exists a constant $D_\mathsf{v}(\phi_1,\phi_2)$ depending continuously on $\phi_1$ and $\phi_2$ such that for all $(t,u) \in (T_{\mathsf{v}},\infty) \times \ker\psi$ such that $t\mathsf{v}+\sqrt{t}u \in \limitcone$, we have
		\begin{multline*}
			\left|t^{\frac{\rankG - 1}{2}}e^{(2\rho - \psi)(t\mathsf{v} + \sqrt{t}u)} \int_{\Gamma \backslash G} \phi_1(x\exp(t\mathsf{v} + \sqrt{t}u)) \phi_2(x) \, dx\right| \\ \le D_\mathsf{v}(\phi_1,\phi_2) e^{-\eta_{\mathsf{v}} I(u)}.
		\end{multline*}
	\end{theorem}
	
	\subsection{Maximal flat cylinders}
	\label{subsec:MaximalFlatCylinders}
	
	\begin{definition}[Maximal flat cylinder]
		\label{def:MaximalFlatCylinder}
		Let $C=\Gamma g AM \subset \Gamma\backslash G/M$ be a closed $AM$-orbit. We say $C$ is \emph{nontrivial} if its stabilizer $g^{-1}\Gamma g \cap AM$ is nontrivial. We say $C$ is a \emph{maximal flat cylinder} if 
		\begin{equation}
			\label{eqn:PeriodsOfACylinder}
			g^{-1}\Gamma g \cap AM \cong \Z.
		\end{equation}
		If in addition to \eqref{eqn:PeriodsOfACylinder} we have 
		\[g^{-1}\Gamma g \cap \interior(A^+)M \ne \{e\},\]
		then we say $C$ is \emph{positively oriented}. 
	\end{definition}
	
	Denote by $\mathcal{C}_\Gamma$ the set of all positively oriented maximal flat cylinders:  
	\[\mathcal{C}_\Gamma := \{C \subset \Gamma \backslash G/M : C \text{ is a positively oriented maximal flat cylinder}\}.\] 
	Let $\primGamma$ denote the set of primitive elements of $\Gamma$ and $\left[\primGamma\right]$ denote the set of $\Gamma$-conjugacy classes in $\primGamma$. The following lemma justifies \cref{def:MaximalFlatCylinder} and shows why we will consider only positively oriented maximal flat cylinders and not all maximal flat cylinders in subsequent sections.
	
	\begin{lemma}
		\label{lem:ClosedAOrbitsAreMaximalFlatCylinders}
		Let $C=\Gamma g AM \subset \Gamma\backslash G/M$ be a nontrivial closed $AM$-orbit in $\Gamma \backslash G/M$. Then the following holds.
		\begin{enumerate}
			\item 
			\label{itm:ClosedAOrbitsAreMaximalFlatCylinders1}
			The stabilizer $g^{-1}\Gamma g \cap AM$ of $C$ is isomorphic to $\Z$. Hence, $C$ is a maximal flat cylinder and homeomorphic to $\mathbb{S}^1 \times \ker\psi$.
			
			\item 
			\label{itm:ClosedAOrbitsAreMaximalFlatCylinders2}
			If $C$ is not positively oriented, then $C$ is disjoint from $\Omega$.
			
			\item 
			\label{itm:ClosedAOrbitsAreMaximalFlatCylinders3}
			If $C$ is positively oriented, then $C$ is contained in $\Omega$. 
			
			\item 
			\label{itm:ClosedAOrbitsAreMaximalFlatCylinders4}
			If $C$ is positively oriented, then the semigroup $\Gamma\cap gA^+Mg^{-1}$ is generated by a single element $\gamma_{C,g}$. In particular, the semigroup
			\[(g^{-1}\Gamma g\cap \interior(A^+)M)/M \cong \{a \in \interior(A^+): \Gamma g M = \Gamma ga M\}\]
			is generated by a single element $\exp (v_C)$ which depends only on $C$ and satisfies $\lambda(\gamma_{C,g}) = v_C$.
			
			\item 
			\label{itm:ClosedAOrbitsAreMaximalFlatCylinders5}
			The map 
			\[C \mapsto \{\gamma_{C,g}: C = \Gamma gAM\}\]  
			is a bijection between $\mathcal{C}_\Gamma$ and $\left[\primGamma\right]$.
		\end{enumerate}  
	\end{lemma}
	
	\begin{proof}
		For \ref{itm:ClosedAOrbitsAreMaximalFlatCylinders1}, suppose for the sake of contradiction that $g^{-1}\Gamma g\cap AM$ contains two elements $g^{-1}\gamma_ig = a_{w_i}m_i$, $i=1,2$ that are not generated by a single element in $g^{-1}\Gamma g \cap AM$. Since $A$ and $M$ commute, we have 
		\[\gamma_1\gamma_2\gamma_1^{-1}\gamma_2^{-1} = gm_1m_2m_1^{-1}m_2^{-1}g^{-1} \in \Gamma.\] 
		Since $\Gamma$ is discrete and torsion free, $\gamma_1\gamma_2\gamma_1^{-1}\gamma_2^{-1}$ must be identity. It follows that $\Gamma$ contains a subgroup isomorphic to $\Z^2$, which is impossible since $\Gamma$ is a hyperbolic group. 
		
		(Alternatively, we observe that $w_1$ and $w_2$ cannot be rational multiples of each other, $\gamma_1^j\gamma_2^k \in  ga_{jw_1+kw_2}Mg^{-1}$ for $j,k\in \Z$, and hence $\{\lambda(\gamma_1^j\gamma_2^k):j,k\in \Z\}$ contains directions arbitrarily close to a wall of $\LieA^+$, contradicting $\limitcone \subset \interior \LieA^+$.) 
		
		Hence, $g^{-1}\Gamma g \cap AM \cong \Z$ and 
		\[C\cong ( g^{-1}\Gamma g \cap AM)\backslash AM/M \cong \Z \backslash \R^{\rankG} \cong \mathbb{S}^1 \times  \R^{\rankG-1}.\]
		We highlight a more insightful way to see $C$ as a cylinder $\mathbb{S}^1 \times \ker\psi$. Observe that under the vector bundle isomorphism $\Omega \cong \mathcal{X}_\psi \times \ker\psi$, $C=\Gamma g AM$ is identified with 
		\[\{\Gamma(g^+,g^-,t): t \in \R\} \times \ker\psi.\] 
		Since $C$ is closed, so is $\{\Gamma(g^+,g^-,t): t \in \R\}$. Since the flow $\Phi$ is given by an $\R$-action on a compact space (e.g. geodesic flow on the unit tangent bundle of a compact surface rather than the geodesic flow on the surface itself) it follows that the closed $\Phi$-orbit $\{\Gamma(g^+,g^-,t): t \in \R\}$ is homeomorphic to $\mathbb{S}^1$.
		
		For \ref{itm:ClosedAOrbitsAreMaximalFlatCylinders2} and \ref{itm:ClosedAOrbitsAreMaximalFlatCylinders3}, we have $\Omega \cong \Gamma \backslash (\limitset^{(2)} \times \LieA)$ \eqref{eqn:Omega} and by the Hopf parametrization (\cref{def:HopfParametrization}), $C=\Gamma g AM$ is identified with 
		\[\{\Gamma(g^+,g^-,v): v \in \LieA\}.\]
		Then $C$ is contained in $\Omega$ if and only if $g^+ \in \limitset$. If $C$ is positively oriented, then there exists $\gamma \in g\interior(A^+)Mg^{-1}$ and $g^+$ is the attracting fixed point of $\gamma$ so $g^+ \in \limitset$. On the other hand, if $C$ is not positively oriented, then there exists a representative $w$ of a Weyl element in $W \setminus \{M,w_0M\}$ such that there exists $\gamma \in gw\interior(A^+)Mw^{-1}g^{-1}$. Then the attracting fixed point of $\gamma$ is $(gw)^+ \in \limitset$. Then $g^+ \not\in \limitset$ otherwise by \eqref{eqn:LimitSet2}, $(g^+,(gw)^+) \in \limitset^{(2)} \subset \Fboundary^{(2)}$, which contradicts the definition of $\Fboundary^{(2)}$ \eqref{eqn:FurstenbergBoundary2}. 
		
		Now \ref{itm:ClosedAOrbitsAreMaximalFlatCylinders4} is immediate from the definitions. For \ref{itm:ClosedAOrbitsAreMaximalFlatCylinders5}, surjectivity of the map $C \mapsto \{\gamma_{C,g}: C = \Gamma gAM\}$ is clear and injectivity follows from the fact that $N_G(A^+ M)=AM$.
	\end{proof}
	
	In view of \cref{lem:ClosedAOrbitsAreMaximalFlatCylinders}\ref{itm:ClosedAOrbitsAreMaximalFlatCylinders4},\ref{itm:ClosedAOrbitsAreMaximalFlatCylinders5}, for each $C \in \mathcal{C}_\Gamma$, we fix some 
	\begin{equation}
		\label{eqn:GammaCDefinition}
		\gamma_C \in \{\gamma_{C,g}: C = \Gamma gAM\}
	\end{equation} 
	and define the \emph{$\psi$-circumference} of $C$ to be
	\begin{equation}
		\label{eqn:PsiCircumferenceDefinition}
		\ell_\psi(C):=\psi(\lambda(\gamma_C))
	\end{equation} 
	which is positive by \cref{thm:AnosovSubgroupsPSTheoryProperties}. Let $m_{C,g} \in M$ be the unique element satisfying $\gamma_{C,g} = g\exp(\lambda(\gamma_C))m_{C,g}g^{-1}$. The set 
	\[h_C:=\{m_{C,g}:g \in G\} \in [M]\]
	is a conjugacy class called the \emph{holonomy} of $C$. 
	
	\begin{remark}
		The notion of $\psi$-circumference cannot be well-defined for all maximal flat cylinders simultaneously. Indeed, consider a maximal flat cylinder $C=\Gamma g AM$ such that neither generator $am$ nor $(am)^{-1}$ of $g^{-1}\Gamma g \cap AM \cong \Z$ lies in $(\interior A^+)M$. Since \cref{thm:AnosovSubgroupsPSTheoryProperties} only guarantees that $\psi$ is positive on $\limitcone \setminus \{0\}$, to define the $\psi$-circumference, we would need to extract from $a$ and $a^{-1}$ an intrinsic element in $\limitcone$. The natural candidates are the Jordan projections of one of the generators $\gamma$ or $\gamma^{-1}$ of $\Gamma \cap gAMg^{-1}$. However, when $G$ is higher rank, the opposition involution $\involution$ may be nontrivial and 
		\[\psi(\lambda(\gamma^{-1})) = \psi(\involution(\lambda(\gamma))) \ne \psi(\lambda(\gamma))\] 
		in general.
	\end{remark}
	
	The next lemma gives a geometric description of $\psi$-circumferences. 
	
	\begin{lemma}
		\label{lem:PeriodicOrbits}
		The map 
		\[C = \Gamma gAM \in \mathcal{C}_\Gamma \to \pi_\psi(C)=\{\Gamma(g^+,g^-,t):t\in \R\} \subset \mathcal{X}_\psi\] 
		is a bijection between the set of positively oriented maximal flat cylinders and the set of periodic orbits of the translation flow $\Phi$ on $\mathcal{X}_\psi$. Using the embedding of $\mathcal{X}_\psi$ in $\Omega$ \eqref{eqn:EmbeddingOfX}, the intersection of $C = \Gamma gAM \in \mathcal{C}_\Gamma$ with $\mathcal{X}_\psi$ is the corresponding periodic orbit $\pi_\psi(C)$ of the translation flow and the $\psi$-circumference of $C$ is the period of $\pi_\psi(C)$.
	\end{lemma}
	
	\begin{proof}
		Observe that the $\Phi$-orbit $\{\Gamma(g^+,g^-,t):t\in \R\}$ is periodic if and only if there exists $\gamma \in \Gamma$ such that $x$ and $y$ are fixed points of some $\gamma \in \Gamma$. By the same reasoning as in the proof of \cref{lem:ClosedAOrbitsAreMaximalFlatCylinders}\ref{itm:ClosedAOrbitsAreMaximalFlatCylinders3}, $x$ and $y$ must be the attracting and repelling fixed points of $\gamma$ and this establishes a bijection between the set of periodic orbits of the translation flow and the set of conjugacy classes in $\Gamma$. In view of \cref{lem:ClosedAOrbitsAreMaximalFlatCylinders}\ref{itm:ClosedAOrbitsAreMaximalFlatCylinders5}, this establishes the desired bijection. 
		
		For the last assertion, let $C = \Gamma gAM \in \mathcal{C}_\Gamma$. Using the embedding of $\mathcal{X}_\psi$ in $\Omega$ \eqref{eqn:EmbeddingOfX}, we have      
		\[C \cap (\mathcal{X}_\psi \times \{0\}) = \{\Gamma(g^+,g^-,t):t\in \R\} \times \{0\}\]
		and the period of $\{\Gamma(g^+,g^-,t):t\in \R\}$ is precisely 
		\[\psi(\beta_{g^+}(\gamma_{C,g}^{-1},e)) = \psi(\lambda(\gamma_{C,g})) =\ell_\psi(C).\]
	\end{proof}
	
	For the rest of the paper, we will only consider positively oriented maximal flat cylinders and we will omit the phrase "positively oriented".
	
	\section{\texorpdfstring{Counting almost cylindrical maximal flats}{Counting almost cylindrical maximal flats}}
	\label{sec:Counting}
	
	The main result of this section is \cref{prop:CountingInVT}. \cref{prop:CountingInVT} can be thought of as an asymptotic for the number elements in $\Gamma$ corresponding to maximal flats in $\Gamma \backslash G/M$ which pass through and return to a given flow box in $\Omega$ and are almost positively oriented maximal flat cylinders with $\psi$-circumference at most $T$ (\cref{prop:CountingInVT}). In \cref{sec:ProofOfJointEquidistribution}, we then relate \cref{prop:CountingInVT} to the number of maximal flat cylinders of $\psi$-circumference at most $T$ which pass through $\tilde{\mathcal{B}}(g_0,\varepsilon)$ by applying the effective closing lemma for regular directions (\cref{lem:EffectiveClosingLemma}), and we prove the main joint equidistribution result \cref{thm:MuTJointEquidistribution}.  
	
	Recall that we have fixed a tangent form $\psi$ tangent to $\growthindicator$ at a normalized direction $\mathsf{v} \in \interior\limitcone$, that is, 
	\[\psi \in \LieA^*, \, \psi \ge \growthindicator, \, \mathsf{v} \in \interior \limitcone \text{ and } \psi(\mathsf{v}) = \growthindicator(\mathsf{v}) = 1.\]
	We fix a cone 
	\[\mathfrak{C} \subset \limitcone \text{ with } \mathsf{v} \in \interior \mathfrak{C}.\]  
	For $T > 0$, let 
	\begin{equation*}
		\mathfrak{C}_T := \{\exp(w): w \in \mathfrak{C}, \psi(w) \le T\}
	\end{equation*} 
	which is a bounded subset of $\mathcal{L}_\Gamma$ since $\mathfrak{C} \subset \limitcone$ and $\psi>0$ on $\limitcone \setminus \{0\}$. 
	
	Recall the definition of flow boxes (\cref{def:FlowBox}). Let $g_0\in G$, $T, \varepsilon >0$, and $\Theta$ be a Borel subset of $M$. We denote
	\begin{equation}
		\label{eqn:VTDefinition}
		\mathcal{V}_T(g_0,\varepsilon,\mathfrak{C},\Theta) := \mathcal{B}(g_0,\varepsilon)\mathfrak{C}_T\Theta\mathcal{B}(g_0,\varepsilon)^{-1}=g_0\mathcal{V}_T(e,\varepsilon,\mathfrak{C},\Theta)g_0^{-1}.
	\end{equation}
	Each $\gamma \in \Gamma \cap \mathcal{V}_T(g_0,\varepsilon,\mathfrak{C},\Theta)$ corresponds to a family of $AM$-orbits in $G$ which pass through the flow box $\mathcal{B}(g_0,\varepsilon)$ and then passes through $\gamma \mathcal{B}(g_0,\varepsilon)$ after translation by an element in $\mathfrak{C}_T\Theta$. In this sense, the elements of $\Gamma \cap \mathcal{V}_T(g_0,\varepsilon,\mathfrak{C},\Theta)$ correspond to almost cylindrical maximal flats. In this section, we prove an asymptotic (\cref{prop:CountingInVT}) for the number of elements in $\Gamma \cap \mathcal{V}_T(g_0,\varepsilon,\mathfrak{C},\Theta)$.
	
	\subsection{\texorpdfstring{Counting in $N^+AMN$-coordinates}{Counting in N+MAN-coordinates}}
	\label{subsec:CountinginNAMNCoordinates}
	
	In this subsection, we will prove an asymptotic for the number of elements in $\Gamma$ contained product subsets of $N^+AMN$ of a certain form. Throughout this subsection, we fix bounded Borel sets $\Xi_1 \subset N^+$, $\Xi_2 \subset N$ and $\Theta \subset M$. For $T > 0$ and $\varepsilon >0$, we denote  
	\begin{equation}
		\label{eqn:STDefinition}
		S_T := S_T(\Xi_1,\Xi_2,\mathfrak{C},\Theta) := \Xi_1\mathfrak{C}_T \Theta \Xi_2.
	\end{equation} 
	Let $g_0 \in G$. We prove using local mixing of the Haar measure (\cref{thm:DecayofMatrixCoefficients}) an asymptotic for $\#(\Gamma \cap g_0S_Tg_0^{-1})$ (\cref{prop:STAsymptotic}) which is the main input in the proof of \cref{prop:CountingInVT}. 
	
	Given a bounded Borel subset $B$ of $G$, define the counting function $F_B: G \times G \to \N$ by
	\[ F_B(g,h) := \sum_{\gamma \in \Gamma} 1_B(g^{-1}\gamma h) = \#(g^{-1}\Gamma h \cap B) = \#(\Gamma h \cap gB).\] 
	The function $F_B$ is $\Gamma$-invariant in both arguments so it descends to a function on $\Gamma \backslash G \times \Gamma\backslash G$ which we still denote by $F_B$. Note that
	\[F_{S_T}(e,e) = \#(\Gamma \cap S_T).\]
	For $F_1,F_2 :\Gamma \backslash G \times \Gamma\backslash G \to \R$, let 
	\[\langle F_1,F_2 \rangle := \int_{\Gamma \backslash G \times \Gamma\backslash G}F_1(x_1,x_2)F_2(x_1,x_2) \, dx_1 \, dx_2\]
	when the integral makes sense, where $dx_1,dx_2$ are both the Haar measure on $\Gamma \backslash G$. 
	
	For $\varepsilon>0$, we denote
	\begin{align*}
		& S_{T,\varepsilon}^- := \bigcap_{g_1,g_2 \in G_\varepsilon}g_1S_Tg_2; & S_{T,\varepsilon}^+ := \bigcup_{g_1,g_2 \in G_\varepsilon}g_1S_Tg_2.
	\end{align*}
	
	The sets $S_{T,\varepsilon}^\pm$ can be used to approximate $\#(\Gamma \cap S_T)$ as in the following lemma. For $\varepsilon >0$ less than the injectivity radius of $\Gamma$, we fix a nonnegative function $\psi_\varepsilon \in C^\infty(G)$ with $\supp \psi_\varepsilon \subset G_\varepsilon$ and $\int_G \psi_\varepsilon \, dg = 1$. Let $\Psi_\varepsilon \in C^\infty(\Gamma \backslash G)$ be defined by $\Psi_\varepsilon(\Gamma g) := \sum_{\gamma \in \Gamma} \psi_\varepsilon(\gamma g)$ for all $g \in G$. 
	
	\begin{lemma}
		\label{lem:InequalityForCountingInGamma}
		For any $T>0$ and $\varepsilon >0$, we have
		\[\langle F_{S_{T,\varepsilon}^-}, \Psi_\varepsilon \otimes \Psi_\varepsilon \rangle \le F_T(e,e) \le \langle F_{S_{T,\varepsilon}^+}, \Psi_\varepsilon \otimes \Psi_\varepsilon \rangle.\]
	\end{lemma}
	
	Our goal is to now estimate $\langle F_{S_{T,\varepsilon}^\pm}, \Psi_\varepsilon \otimes \Psi_\varepsilon \rangle$. We begin with \cref{lem:UnfoldingLemma,lem:LocalMixingSetup,lem:BRKANDecomposition,lem:QBEpsilonFormula} which are computations relating to $\langle F_B, \Psi_1 \otimes \Psi_2 \rangle$ for general bounded Borel subset $B$ of $G$. 
	
	A standard folding and unfolding argument gives the following lemma.
	
	\begin{lemma}
		\label{lem:UnfoldingLemma}
		For any bounded Borel subset $B$ of $G$ and for all $\Psi_1, \Psi_2 \in \mathrm{C}_{\mathrm{c}}(\Gamma \backslash G)$, we have 
		\[\langle F_B, \Psi_1 \otimes \Psi_2 \rangle = \int_B\langle\Psi_1, g\Psi_2 \rangle_{L^2(\Gamma \backslash G)} \, dg,\]
		where $g\Psi_2(x):=\Psi_2(xg)$.
	\end{lemma}
	
	In view of the matrix coefficient of $L^2(\Gamma \backslash G)$ in \cref{lem:UnfoldingLemma}, we now express the matrix coefficient in a way that lends itself to using local mixing (\cref{thm:DecayofMatrixCoefficients}). For convenience, we denote 
	\[\rank := \rankG \text{ and } a_w := \exp(w) \text{ for } w \in \LieA.\]
	
	\begin{lemma}
		\label{lem:LocalMixingSetup}
		For any bounded Borel subset $B$ of $G$ and for all $\Psi_1, \Psi_2 \in \mathrm{C}_{\mathrm{c}}(\Gamma \backslash G)$, we have 
		\begin{align*}
			& \langle F_B, \Psi_1 \otimes \Psi_2 \rangle
			\\ 
			& = \frac{\kappa_{\mathsf{v}}}{|m_{\mathcal{X}_\psi}|}\int_{h^{-1}a_{t\mathsf{v}+\sqrt{t}u}mn \in B}e^{t}e^{-I(u)}\sum_{Z \in \mathfrak{Z}_\Gamma} \BRstar\bigr|_{ZN}(h\Psi_1)\BR\bigr|_{ZN^+}(mn\Psi_2) 
			\\
			& \qquad \qquad \qquad \qquad \qquad \qquad \qquad  +e^{t}E(t,u,h,mn) \, dt \, du \, dh \, dm  \, dn,
		\end{align*}
		where $h \in N^+, m \in M, n \in N$, $t \in \R$, $u \in \ker\psi$ and
		\begin{multline}
			\label{eqn:MixingErrorTerm}
			E(t,u,h,mn) = t^{\frac{\rank-1}{2}}e^{2\rho(t\mathsf{v}+\sqrt{t}u)-t}\int_{\Gamma \backslash G} \Psi_1(xh)\Psi_2(xa_{t\mathsf{v}+\sqrt{t}u}mn) \, dx
			\\
			- \frac{\kappa_{\mathsf{v}}e^{-I(u)}}{|m_{\mathcal{X}_\psi}|}\sum_{Z \in \mathfrak{Z}_\Gamma} \BRstar\bigr|_{ZN}(h\Psi_1)\BR\bigr|_{ZN^+}(mn\Psi_2) 
		\end{multline}
		is the associated error term in \cref{thm:DecayofMatrixCoefficients} (note that $\psi(t\mathsf{v}+\sqrt{t}u) = t$).
	\end{lemma}
	
	\begin{proof}
		By \cref{lem:UnfoldingLemma}, we have
		\[\langle F_B, \Psi_1 \otimes \Psi_2 \rangle = \int_B\int_{\Gamma \backslash G} \Psi_1(x)\Psi_2(xg) \, dx \, dg.\]		
		Using the fact that the product map $N^+\times M \times A \times N \to G$ is a diffeomorphism onto a dense and open subset of $G$, we write $g = h^{-1}\exp(w)mn \in N^+MAN$, so $dg = e^{2\rho(w)} \, dh \, dw \, dm \, dn$. Then 
		\begin{align*}
			& \int_B\int_{\Gamma \backslash G} \Psi_1(x)\Psi_2(xg) \, dx \, dg
			\\
			& = \int_{h^{-1}\exp(w)mn \in B}\int_{\Gamma \backslash G} \Psi_1(x)\Psi_2(xh^{-1}a_wmn) e^{2\rho(w)} \, dx \, dh  \, dw \, dm \, dn
			\\
			& = \int_{h^{-1}a_wmn \in B}e^{2\rho(w)}\int_{\Gamma \backslash G} \Psi_1(xh)\Psi_2(xa_wmn)  \, dx \, dw \, dh \, dm  \, dn.
		\end{align*}
		We write $w = t\mathsf{v} + \sqrt{t}u$, where $t \in \R$ and $u \in \ker\psi$. Recall $\rank :=\rankG$. Then $dw=t^{\frac{\rank -1}{2}}dt\,du$, where $dt$ is the Lebesgue measure on $\R$ and $du$ is the Lebesgue measure on $\ker\psi$ from \cref{subsec:TheSupportOfTheBMSMeasureAsAVectorBundle}. Then
		\begin{align*}
			& \int_{h^{-1}a_wmn \in B}e^{2\rho(w)}\int_{\Gamma \backslash G} \Psi_1(xh)\Psi_2(xa_wmn)  \, dx \, dw \, dh \, dm  \, dn
			\\
			& = \int_{h^{-1}a_{t\mathsf{v}+\sqrt{t}u}mn \in B} t^{\frac{\rank-1}{2}}e^{2\rho(t\mathsf{v}+\sqrt{t}u)}\int_{\Gamma \backslash G} \Psi_1(xh)\Psi_2(xa_{t\mathsf{v}+\sqrt{t}u}mn)
			\\
			& \qquad \qquad \qquad \qquad \qquad \qquad \qquad \qquad \qquad \qquad  \, dx   \, dt \, du \, dh \, dm  \, dn
			\\
			& = \frac{\kappa_{\mathsf{v}}}{|m_{\mathcal{X}_\psi}|}\int_{h^{-1}a_{t\mathsf{v}+\sqrt{t}u}mn \in B}e^{t}e^{-I(u)}\sum_{Z \in \mathfrak{Z}_\Gamma} \BRstar\bigr|_{ZN}(h\Psi_1)\BR\bigr|_{ZN^+}(mn\Psi_2) 
			\\
			& \qquad \qquad \qquad \qquad \qquad \qquad +e^{t}E(t,u,h,mn) \, dt \, du \, dh \, dm  \, dn.
		\end{align*}
	\end{proof}
	
	We now specialize to the case when $\Psi_1 = \Psi_2 = \Psi_\varepsilon$ in \cref{lem:LocalMixingSetup}. Let $E_\varepsilon(t,u,h,mn)$ denote the associated error term in \eqref{eqn:MixingErrorTerm} and denote 
	\[Q_{B,\varepsilon}:= \langle F_B, \Psi_\varepsilon \otimes \Psi_\varepsilon \rangle - \int_{h^{-1}a_{t\mathsf{v}+\sqrt{t}u}mn \in B}e^{t}E_\varepsilon(t,u,h,mn) \, dt \, du \, dh \, dm  \, dn.\] 
	In the next lemma, we use the Iwasawa decompositions to decompose the BR-measures appearing in \cref{lem:LocalMixingSetup} into a form that will be useful in the proof of \cref{prop:STAsymptotic}. 
	
	Recall $\mathfrak{Z}_\Gamma$ from \eqref{eqn:ErgodicComponents}. For $Z \in \mathfrak{Z}_\Gamma$, let $\tilde{Z}$ denote the preimage of $Z$ in $G$, that is, 
	\[\tilde{Z}:=\{g \in G: \Gamma g \in Z\}.\]
	
	\begin{lemma}
		\label{lem:BRKANDecomposition}
		Let $\varepsilon >0$ be sufficiently small. Let $Z \in \mathfrak{Z}_\Gamma$, $h\in N^+, m \in M$ and $n \in N$. Then we have
		\begin{align*}
			& \BRstar\bigr|_{ZN}(h\Psi_\varepsilon)	 	\BR\bigr|_{ZN^+}(mn\Psi_\varepsilon)
			\\
			& = \int_{KAN\times KAN^+} \psi_\varepsilon(k_1a_{w_1}n_1h)\mathbbm{1}_{\tilde{Z}N}(k_1)\psi_\varepsilon(k_2a_{w_2}h_1mn) \mathbbm{1}_{\tilde{Z}N^+}(k_2)
			\\
			& \qquad \qquad \qquad \qquad \qquad \cdot  e^{\psi(w_1-w_2)}
			\, d\nu^K_{\psi}(k_1) \, dw_1 \, dn_1 \, d\nu^K_{\psi\circ\involution}(k_2) \, dw_2 \, dh_1,
		\end{align*}
		where $\nu^K_{\psi}$ and $\nu^K_{\psi\circ\involution}$ are defined by \eqref{eqn:LiftsOfConformalMeasuresToK}.
	\end{lemma}
	
	\begin{proof}
		For $g_1,g_2 \in G$, write $g_1 = k_1a_{w_1}n_1 \in KAN$ and $g_2 = k_2a_{w_2}h_1 \in KAN^+$.
		Using \cref{lem:BurgerRoblinMeasuresInKANCoordinates}, we have
		\begin{align*}
			& \BRstar\bigr|_{ZN}(h\Psi_\varepsilon)	 	\BR\bigr|_{ZN^+}(mn\Psi_\varepsilon)
			\\
			& = \int_{G\times G} \psi_\varepsilon(g_1h)\mathbbm{1}_{\tilde{Z}N}(g_1)\psi_\varepsilon(g_2mn) \mathbbm{1}_{\tilde{Z}N^+}(g_2) \, d\BRstar(g_1) \, d\BR(g_2)
			\\
			& = \int_{KAN\times KAN^+} \psi_\varepsilon(k_1a_{w_1}n_1h)\mathbbm{1}_{\tilde{Z}N}(k_1a_{w_1}n_1) \psi_\varepsilon(k_2a_{w_2}h_1mn)
			\\
			& \qquad \qquad \cdot \mathbbm{1}_{\tilde{Z}N^+}(k_2a_{w_2}h_1) e^{\psi(w_1-w_2)} 
			\, d\nu^K_{\psi}(k_1) \, dw_1 \, dn_1 \, d\nu^K_{\psi\circ\involution}(k_2) \, dw_2 \, dh_1
			\\
			& = \int_{KAN\times KAN^+} \psi_\varepsilon(k_1a_{w_1}n_1h)\mathbbm{1}_{\tilde{Z}N}(k_1)\psi_\varepsilon(k_2a_{w_2}h_1mn) \mathbbm{1}_{\tilde{Z}N^+}(k_2)
			\\
			& \qquad \qquad \qquad \qquad \qquad \cdot  e^{\psi(w_1-w_2)}
			\, d\nu^K_{\psi}(k_1) \, dw_1 \, dn_1 \, d\nu^K_{\psi\circ\involution}(k_2) \, dw_2 \, dh_1.
		\end{align*}
	\end{proof}
	
	To state the next lemma, it will be convenient for us to define some notation. We define $f_B:N^+\times MN \to \R$ by
	\begin{equation*}
		f_B(h,mn) := \frac{\kappa_{\mathsf{v}}}{|m_{\mathcal{X}_\psi}|}\int_{h^{-1}ma_{t\mathsf{v}+\sqrt{t}u}n \in B}e^{t}e^{-I(u)} \, dt \, du.
	\end{equation*}
	Let 
	\[\tilde{\mathfrak{Z}}_\Gamma := \{\tilde{Z}: Z \in \mathfrak{Z}_\Gamma\}.\]
	Define the natural projection maps 
	\begin{align*} 
		H_1: MANN^+ &\to M,
		\\ 
		I_1: MANN^+ &\to \LieA,
		\\ 
		J_1: MANN^+ &\to N^+,
		\\ 
		I_2: AN^+MN &\to \LieA,
		\\  
		J_2: AN^+MN &\to MN. 
	\end{align*}
	For each $kM \in K/M \cong \Fboundary$, fix a representative $k^* \in K$. In the next lemma, we use these projections and a change of variables to write $Q_{B,\varepsilon}$ using integrals over $G_\varepsilon$. 
	
	\begin{lemma}
		\label{lem:QBEpsilonFormula}
		For any bounded Borel subset $B$ of $G$, we have
		\begin{align*}
			Q_{B,\varepsilon} & = \sum_{\tilde{Z} \in \tilde{\mathfrak{Z}}_\Gamma} \int_{(K/M)\times K}\int_{G_\varepsilon \times G_\varepsilon} f_{B}(J_1((k_1^*) ^{-1}g'),J_2(k_2^{-1}g'')) 
			\\
			& \cdot  \psi_\varepsilon(g')\mathbbm{1}_{\tilde{Z}N}(k_1^* H_1((k_1^*)^{-1}g')) \psi_\varepsilon(g'') \mathbbm{1}_{\tilde{Z}N^+}(k_2)
			\\
			& \qquad \qquad  \cdot e^{\psi(I_1((k_1^*)^{-1}g')-I_2(k_2^{-1}g''))}   \, dg'  \, dg'' \, d\nu_\psi(k_1^+) \, d\nu^K_{\psi\circ\involution}(k_2).
		\end{align*}    
	\end{lemma}
	
	\begin{proof}
		By \cref{lem:LocalMixingSetup} and \cref{lem:BRKANDecomposition}, we have
		\begin{align*}
			Q_{B,\varepsilon}& = \frac{\kappa_{\mathsf{v}}}{|m_{\mathcal{X}_\psi}|}\sum_{Z \in \mathfrak{Z}_\Gamma} \int_{KAN\times KAN^+} \int_{h^{-1}a_{t\mathsf{v}+\sqrt{t}u}mn \in B}e^{t}e^{-I(u)}\psi_\varepsilon(k_1a_{w_1}n_1h)
			\\
			& \qquad \qquad \qquad \qquad \qquad \cdot \mathbbm{1}_{\tilde{Z}N}(k_1)\psi_\varepsilon(k_2a_{w_2}h_1mn) \mathbbm{1}_{\tilde{Z}N^+}(k_2) e^{\psi(w_1-w_2)}
			\\
			& \qquad \qquad \qquad \qquad 
			\, d\nu^K_{\psi}(k_1) \, dw_1 \, dn_1 \, d\nu^K_{\psi\circ\involution}(k_2) \, dw_2 \, dh_1 \, dt \, du \, dh \, dm  \, dn.
		\end{align*}
		Let $m_1 \in M$ such that $k_1 = k_1^* m_1$. Let $g_3 =  m_1a_{w_1}n_1h$ and $g_4=a_{w_2}h_1mn$. Then we have $dg_3 = dm_1 \, dw_1 \, dn_1 \, dh$ and $dg_4 = dw_2 \, dh_2 \, dm \, dn$ as in \cref{subsec:GeometricMeasures}. Using these change of variables, we have
		\begin{align*}
			Q_{B,\varepsilon} & = \sum_{\tilde{Z} \in \tilde{\mathfrak{Z}}_\Gamma} \int_{(K/M)\times K}\int_{G \times G} f_{B}(J_1(g_3),J_2(g_4)) \psi_\varepsilon(k_1^* g_3)\mathbbm{1}_{\tilde{Z}N}(k_1^* H_1(g_3))
			\\
			& \qquad \cdot   \psi_\varepsilon(k_2g_4) \mathbbm{1}_{\tilde{Z}N^+}(k_2)  e^{\psi(I_1(g_3)-I_2(g_4))} \, dg_3  \, dg_4 \, d\nu_\psi(k_1^+) \, d\nu^K_{\psi\circ\involution}(k_2)  
			\\
			& = \sum_{\tilde{Z} \in \tilde{\mathfrak{Z}}_\Gamma} \int_{(K/M)\times K}\int_{G_\varepsilon \times G_\varepsilon} f_{B}(J_1((k_1^*) ^{-1}g'),J_2(k_2^{-1}g'')) 
			\\
			& \cdot  \psi_\varepsilon(g')\mathbbm{1}_{\tilde{Z}N}(k_1^* H_1((k_1^*)^{-1}g')) \psi_\varepsilon(g'') \mathbbm{1}_{\tilde{Z}N^+}(k_2)
			\\
			& \qquad \qquad  \cdot e^{\psi(I_1((k_1^*)^{-1}g')-I_2(k_2^{-1}g''))}   \, dg'  \, dg'' \, d\nu_\psi(k_1^+) \, d\nu^K_{\psi\circ\involution}(k_2),
		\end{align*}   
		where $g' = k_1^* g_3$ and $g'' = k_2g_4$.
	\end{proof}
	
	We now specialize to $B = S_{T,\varepsilon}^\pm$ and estimate $Q_{S_{T,\varepsilon}^\pm,\varepsilon}$. In view of \cref{lem:QBEpsilonFormula}, to estimate $Q_{S_{T,\varepsilon}^\pm,\varepsilon}$, we need to show that for $k \in K$ that contribute to $Q_{S_{T,\varepsilon}^\pm,\varepsilon}$, the images of $kG_\varepsilon$ under the projections $H_1,I_1,J_1,I_2,J_2$  are close to the image of $k$. This can be done because $\Xi_1 \subset N^+$ and $\Xi_2 \subset N$ are bounded.
	
	\begin{lemma}
		\label{lem:UniformContinuityForNMAN}
		Fix bounded sets $U_1\subset N^+$ and $U_2\subset N$. For sufficiently small $\varepsilon >0$, the following holds. If $g_1 \in G$ with $g_1^{-1} = manh \in MANN^+$ and $g_1^+\in U_1^+$, then 
		\[g_1^{-1}G_\varepsilon \subset mM_{O(\varepsilon)} aA_{O(\varepsilon)}nN_{O(\varepsilon)}hN_{O(\varepsilon)}^+.\] 
		If $g_2\in K$ with $g_2^{-1} = ahmn \in AN^+MN$ and $g_2^-\in U_2^-$, then 
		\[g_2^{-1}G_\varepsilon \subset aA_{O(\varepsilon)}hN^+_{O(\varepsilon)} mM_{O(\varepsilon)}nN_{O(\varepsilon)}.\]
	\end{lemma}
	
	\begin{proof}
		We prove the first statement; the proof of the second is similar. The product map $M \times A \times N \times N^+ \to G$ is a diffeomorphism onto an open set containing the identity. In particular, 
		\[G_\varepsilon \subset M_{O(\varepsilon)}A_{O(\varepsilon)}N_{O(\varepsilon)}N^+_{O(\varepsilon)}.\]
		Let $g_\varepsilon \in G_\varepsilon$ with $g_\varepsilon = m_1a_1n_1h_1 \in M_{O(\varepsilon)}A_{O(\varepsilon)}N_{O(\varepsilon)}N^+_{O(\varepsilon)}$. Under the hypotheses, $h^{-1}$ is an element in the bounded set $U_1$ and
		\begin{align*}
			g_1^{-1} g_\varepsilon & = manhm_1a_1n_1h_1
			\\
			& = mm_1aa_1((m_1a_1)^{-1}n(m_1a_1))((m_1a_1)^{-1}h(m_1a_1))n_1h_1
			\\
			& = mm_1m_2aa_1a_2((m_2a_2)^{-1}(m_1a_1)^{-1}n(m_1a_1)(m_2a_2))n_2h_2h_1,
		\end{align*}
		where $((m_1a_1)^{-1}h(m_1a_1))n_1 = m_2a_2n_2h_2$ and we note that the assumption that $h$ is bounded and $n_1 \in N_{O(\varepsilon)}$ imply that $m_2 \in M_{O(\varepsilon)}$, $a_2 \in A_{O(\varepsilon)}$ and $h_2 \in hN_{O(\varepsilon)}^+$. 
	\end{proof}
	
	We will also need the following lemma which says that bounds the sets $S_{T,\varepsilon}^\pm$ by product subsets of $N^+AMN$ that approximate $S_T$. 
	
	\begin{lemma}
		\label{lem:STpmBounds}
		Assume that $\nu_\psi(\partial \Xi_1^+)$, $\nu_{\psi \circ \involution}(\partial (\Xi_2^{-1})^{-})$ and $\int_{\partial \Theta} \, dm$ are all $0$. For all sufficiently small $\varepsilon > 0$, there exists Borel sets $\mathfrak{C}_{T,\varepsilon}' \subset  \mathfrak{C}_T$, $\Xi_{1,\varepsilon}' \subset \Xi_1$, $\Xi_{2,\varepsilon}' \subset \Xi_2$ and $\Theta_{\varepsilon}' \subset \Theta$ such that 
		\[N_{O(\varepsilon)}^+\Xi_{1,\varepsilon}' \mathfrak{C}_{T,\varepsilon}' M_{O(\varepsilon)} \Theta_{\varepsilon}' \Xi_{2,\varepsilon}' N_{O(\varepsilon)}\subset S_{T,\varepsilon}^-,\]
		where an $O(\varepsilon)$-neighborhood of $\mathfrak{C}_{T,\varepsilon}'$ contains $\mathfrak{C}_T$ and $\nu_\psi((\Xi_{1,\varepsilon}')^+) \to \nu_\psi(\Xi_1^+)$, $\nu_{\psi\circ\involution}((\Xi_{2,\varepsilon}')^-) \to \nu_{\psi \circ \involution}(\Xi_2^-)$ and $\int_{\Theta_\varepsilon'} \, dm \to \int_{\Theta} \, dm$ as $\varepsilon \to 0$. 
		
		Similarly, there exists Borel sets $\mathfrak{C}_{T,\varepsilon}'' \supset  \mathfrak{C}_T$, $\Xi_{1,\varepsilon}'' \supset \Xi_1$, $\Xi_{2,\varepsilon}'' \supset \Xi_2$ and $\Theta_{\varepsilon}'' \supset \Theta$ such that 
		\[N_{O(\varepsilon)}^+\Xi_{1,\varepsilon}'' \mathfrak{C}_{T,\varepsilon}'' M_{O(\varepsilon)} \Theta_{\varepsilon}'' \Xi_{2,\varepsilon}'' N_{O(\varepsilon)}\supset S_{T,\varepsilon}^+,\]
		where an $O(\varepsilon)$-neighborhood of $\mathfrak{C}_{T}$ contains $\mathfrak{C}_{T,\varepsilon}''$ and $\nu_\psi((\Xi_{1,\varepsilon}'')^+) \to \nu_\psi(\Xi_1^+)$, $\nu_{\psi\circ\involution}((\Xi_{2,\varepsilon}'')^-) \to \nu_{\psi \circ \involution}(\Xi_2^-)$ and $\int_{\Theta_\varepsilon''} \, dm \to \int_{\Theta} \, dm$ as $\varepsilon \to 0$.   
	\end{lemma}
	
	\begin{proof}
		In view of the hypotheses on the boundaries, we may assume $\Xi_1$, $\Xi_2$ and $\Theta$ are open subsets. Using reasoning similar to that used in \cref{lem:UniformContinuityForNMAN}, we see that for all $g_1,g_2 \in G_\varepsilon$ and for all $g = hamn \in N^+MAN$ with $h$ and $n$ bounded, we have
		\[g_1gg_2 \in N_{O(\varepsilon)}^+ hA_{O(\varepsilon)}aM_{O(\varepsilon)}mnN_{O(\varepsilon)}.\]
		Then we can take $\mathfrak{C}_{T,\varepsilon}'$ to be the intersection of $\mathfrak{C}_T$ and the complement of the closed $O(\varepsilon)$-neighborhood of the exterior of $\mathfrak{C}_T$. Similarly for $\Xi_{1,\varepsilon}' \subset \Xi_1$, $\Xi_{2,\varepsilon}'' \subset \Xi_2$ and $\Theta_{\varepsilon}' \subset \Theta$. It is clear that these sets have the desired properties.
		
		The proof of the second assertion of the lemma is similar.     		
	\end{proof}
	
	We will use the following asymptotic notation. For real-valued functions $f_1,f_2$ of $T$, we write 
	\[f_1 \sim f_2 \iff \lim\limits_{T\to\infty} \frac{f_1(T)}{f_2(T)} = 1.\]
	For a real-valued function $f$ of $\varepsilon$, we write 
	\begin{equation}
		\label{eqn:LittleOEpsilonDefinition}
		f = o_\varepsilon(1) \iff \lim_{\varepsilon \to 0}f(\varepsilon)=0.
	\end{equation}
	We use local mixing (\cref{thm:DecayofMatrixCoefficients}) to prove the following asymptotic for $\#(\Gamma\cap g_0S_Tg_0^{-1})$ for $g_0 \in G$. 
	
	\begin{proposition}
		\label{prop:STAsymptotic}
		Let $g_0 \in G$. Assume that $\nu_\psi(\partial \Xi_1^+)$, $\nu_{\psi \circ \involution}(\partial (\Xi_2^{-1})^{-})$ and $\int_{\partial \Theta} \, dm$ are all $0$. Then 
		\begin{equation}
			\label{eqn:STAsymptotic}
			\#(\Gamma \cap g_0S_Tg_0^{-1}) \sim \frac{[M:M_\Gamma]}{|m_{\mathcal{X}_\psi}|} e^{T} \sum_{\tilde{Z} \in \tilde{\mathfrak{Z}}_\Gamma}\tilde{\nu}_{\psi}^{g_0}|_{g_0^{-1}\tilde{Z}N}(\Xi_1)\tilde{\nu}_{\psi\circ\involution}^{g_0}|_{g_0^{-1}\tilde{Z}N^+}(\Xi_2^{-1}\Theta^{-1}),
		\end{equation}
		where $\tilde{\nu}_{\psi}^{g_0}|_{g_0^{-1}\tilde{Z}N}$ and $\tilde{\nu}_{\psi\circ\involution}^{g_0}|_{g_0^{-1}\tilde{Z}N^+}$ are defined by \eqref{eqn:LiftsOfConformalMeasuresToN+} and \eqref{eqn:LiftsOfConformalMeasuresToNM2}.
	\end{proposition}
	
	\begin{remark}
		We note that the right hand side of \eqref{eqn:STAsymptotic} does not depend on the choice of $\mathfrak{C}$ as long as $\mathsf{v} \in \interior \mathfrak{C}$. This reflects the fact that the maximum of $\growthindicator$  on $\{w \in \limitcone: \psi(w) = T\}$ occurs in the $\mathsf{v}$ direction.
	\end{remark}
	
	\begin{proof}
		It suffices to prove the theorem for $g_0 = e$. For general $g_0 \in G$, we apply the same argument to $g_0\Gamma g_0^{-1}$, replacing $\nu_\psi$ and $\nu_{\psi\circ\involution}$ with $\nu_\psi^{g_0}$ and $\nu_{\psi\circ\involution}^{g_0}$ \eqref{eqn:TwistedConformalMeasuresDefinition} throughout. 
		
		By \cref{lem:UnfoldingLemma,lem:LocalMixingSetup,lem:BRKANDecomposition,lem:QBEpsilonFormula}, we have
		\begin{align*}
			& \langle F_{S_{T,\varepsilon}^\pm}, \Psi_\varepsilon \otimes \Psi_\varepsilon \rangle
			\\ 
			& = Q_{S_{T,\varepsilon}^\pm, \varepsilon} +\int_{h^{-1}a_{t\mathsf{v}+\sqrt{t}u}mn \in S_{T,\varepsilon}^\pm}e^{t}E_\varepsilon(t,u,h,mn) \, dt \, du \, dh \, dm  \, dn
		\end{align*}
		and 
		\begin{align*}
			Q_{S_{T,\varepsilon}^\pm,\varepsilon} & = \sum_{\tilde{Z} \in \tilde{\mathfrak{Z}}_\Gamma} \int_{(K/M)\times K}\int_{G_\varepsilon \times G_\varepsilon} f_{S_{T,\varepsilon}^\pm}(J_1((k_1^*) ^{-1}g'),J_2(k_2^{-1}g'')) 
			\\
			& \cdot  \psi_\varepsilon(g')\mathbbm{1}_{\tilde{Z}N}(k_1^* H_1((k_1^*)^{-1}g')) \psi_\varepsilon(g'') \mathbbm{1}_{\tilde{Z}N^+}(k_2)
			\\
			& \qquad \qquad  \cdot e^{\psi(I_1((k_1^*)^{-1}g')-I_2(k_2^{-1}g''))}   \, dg'  \, dg'' \, d\nu_\psi(k_1^+) \, d\nu^K_{\psi\circ\involution}(k_2).
		\end{align*}    
		We now do the proof in $2$ steps.
		
		{\bf Step 1:} We estimate $Q_{S_{T,\varepsilon}^\pm,\varepsilon}$ in terms of $S_T$.
		
		Recall from \cref{subsec:ErgodicDecompositions} that $ZN$ is a $P^\circ$-minimal subset of $\Gamma \backslash G$. Then using the correspondence between $P^\circ$-minimal subsets of $\Gamma \backslash G$ and $\Gamma$-minimal subsets of $G/P^\circ$, we have $\tilde{Z}N = \{x \in G: xP^\circ \in \Lambda_1\}$ for some $\Gamma$-minimal subset $\Lambda_1 \subset G/P^\circ$. Recall the holonomy group $M_\Gamma$ of $\Gamma$ (\cref{def:HolonomyGroup}). For fixed $k_1M \in \Lambda$, by \eqref{eqn:HolonomyGroupDefinition2}, there exists  $m' \in M$ such that $\{m \in M: k_1^* mP^\circ \in \Lambda_1\}=M_\Gamma m'$. Then 
		\[\mathbbm{1}_{\tilde{Z}N}(k_1^* H_1((k_1^*)^{-1}g')) = \mathbbm{1}_{M_\Gamma m'}(H_1((k_1^*)^{-1}g')).\]
		Since $M^\circ < M_\Gamma$, if $\varepsilon$ is sufficiently small, then $\mathbbm{1}_{M_\Gamma m'}(H_1((k_1^*)^{-1}g'))$ is constant on $G_\varepsilon$.
		
		It follows from \cref{lem:UniformContinuityForNMAN} and \cref{lem:STpmBounds} that
		\begin{align*}
			\psi(I_1((k_1^*)^{-1}g'))& =\psi(I_1((k_1^*)^{-1}))+O(\varepsilon);
			\\
			\psi(I_2(k_2^{-1}g''))& =\psi(I_2(k_2^{-1}))+O(\varepsilon)
		\end{align*} 
		and
		\begin{align*}
			J_1((k_1^*)^{-1}g')^{-1} & \in N_{O(\varepsilon)}^+J_1((k_1^*)^{-1})^{-1};
			\\
			J_2(k_2^{-1} g'') &\in M_{O(\varepsilon)}J_2(k_2^{-1})N_{O(\varepsilon)}.
		\end{align*}
		
		We first prove a lower bound for $Q_{S_{T,\varepsilon}^-, \varepsilon}$.
		Let $\mathfrak{C}_{T,\varepsilon}' \subset  \mathfrak{C}_T$, $\Xi_{1,\varepsilon}' \subset \Xi_1$, $\Xi_{2,\varepsilon}'' \subset \Xi_2$ and $\Theta_{\varepsilon}' \subset \Theta$ be as in \cref{lem:STpmBounds}. For convenience, let 
		\[S_{T,\varepsilon}':=\Xi_{1,\varepsilon}' \mathfrak{C}_{T,\varepsilon}' \Theta_{\varepsilon}' \Xi_{2,\varepsilon}'.\]
		Then for all $g',g''\in G_\varepsilon$, we have 
		\[f_{S_{T,\varepsilon}^-}(J_1((k_1^*) ^{-1}g'),J_2(k_2^{-1}g'')) \ge  f_{S_{T,\varepsilon}'}(J_1((k_1^*) ^{-1}),J_2(k_2^{-1})).\]
		It now follows that
		\begin{equation}
			\label{eqn:STAsymptotic1}
			\begin{aligned}[b]
				Q_{S_{T,\varepsilon}^-,\varepsilon}
				&\ge \sum_{\tilde{Z} \in \tilde{\mathfrak{Z}}_\Gamma} \int_{(K/M)\times K} f_{S_{T,\varepsilon}'}(J_1((k_1^*) ^{-1}),J_2(k_2^{-1}))  \mathbbm{1}_{\tilde{Z}N}(k_1^* H_1((k_1^*)^{-1})) 
				\\
				& \qquad\cdot \mathbbm{1}_{\tilde{Z}N^+}(k_2) (1+O(\varepsilon))e^{\psi(I_1((k_1^*)^{-1})-I_2(k_2^{-1}))} 
				\,  d\nu_\psi(k_1^+) \, d\nu^K_{\psi\circ\involution}(k_2).
			\end{aligned}
		\end{equation}
		
		Using \cref{lem:STpmBounds}, we observe that
		\begin{equation}
			\label{eqn:STAsymptotic2}
			\begin{aligned}[b] 
				&\int_{(K/M)\times K}  (f_{S_T} - f_{S_T^-})(J_1((k_1^*) ^{-1}),J_2(k_2^{-1})) \, d\nu_\psi(k_1^+)\, d\nu_{\psi\circ\involution}^K(k_2)
				\\ 
				&=O\bigg(e^T \int_{(K/M)\times K} \mathbbm{1}_{\Xi_1 \setminus \Xi_{1,\varepsilon}^-}(J_1((k_1^*) ^{-1})^{-1})
				\\
				& \qquad \qquad \qquad \qquad \qquad \qquad  +\mathbbm{1}_{\Theta\Xi_2\setminus\Theta_\varepsilon\Xi_{2,\varepsilon}^-}(J_2(k_2^{-1}))\, d\nu_\psi(k_1^+)\, d\nu_{\psi\circ\involution}^K(k_2)\bigg)
				\\
				& =e^T O\left(\nu_\psi\left(\Xi_1^+\setminus(\Xi_{1,\varepsilon}^-)^+\right) + \nu_{\psi\circ\involution}\left((\Xi_2^{-1})^-\setminus ((\Xi_{2,\varepsilon}^-)^{-1})^- \right) + m\left(\Theta\setminus\Theta_{\varepsilon}^-\right) \right)
				\\
				& = e^To_\varepsilon(1),
			\end{aligned}
		\end{equation}
		where $o_\varepsilon$ notation is defined by \eqref{eqn:LittleOEpsilonDefinition}. Combining \eqref{eqn:STAsymptotic1} and \eqref{eqn:STAsymptotic2} yields
		\begin{align*}
			& Q_{S_{T,\varepsilon}^-,\varepsilon}
			\ge e^To_\varepsilon(1)+\sum_{\tilde{Z} \in \tilde{\mathfrak{Z}}_\Gamma} \int_{(K/M)\times K} (1+O(\varepsilon))f_{S_T}(J_1((k_1^*)^{-1}),J_2(k_2^{-1}))  
			\\
			& \qquad \cdot  \mathbbm{1}_{\tilde{Z}N}(k_1^* H_1((k_1^*)^{-1}))\mathbbm{1}_{\tilde{Z}N^+}(k_2) e^{\psi(I_1((k_1^*)^{-1})-I_2(k_2^{-1}))} \, d\nu_\psi(k_1^+) \, d\nu^K_{\psi\circ\involution}(k_2).
		\end{align*}
		A similar argument shows that
		\begin{align*}
			& Q_{S_{T,\varepsilon}^+,\varepsilon}
			\le e^To_\varepsilon(1)+\sum_{\tilde{Z} \in \tilde{\mathfrak{Z}}_\Gamma} \int_{(K/M)\times K} (1+O(\varepsilon))f_{S_T}(J_1((k_1^*)^{-1}),J_2(k_2^{-1}))  
			\\
			& \qquad  \cdot  \mathbbm{1}_{\tilde{Z}N}(k_1^* H_1((k_1^*)^{-1}))\mathbbm{1}_{\tilde{Z}N^+}(k_2) e^{\psi(I_1((k_1^*)^{-1})-I_2(k_2^{-1}))} \, d\nu_\psi(k_1^+) \, d\nu^K_{\psi\circ\involution}(k_2)
		\end{align*}
		and we conclude that
		\begin{align*}
			& Q_{S_{T,\varepsilon}^\pm,\varepsilon}
			= e^To_\varepsilon(1)+\sum_{\tilde{Z} \in \tilde{\mathfrak{Z}}_\Gamma} \int_{(K/M)\times K} (1+O(\varepsilon))f_{S_T}(J_1((k_1^*)^{-1}),J_2(k_2^{-1}))  
			\\
			& \qquad \cdot  \mathbbm{1}_{\tilde{Z}N}(k_1^* H_1((k_1^*)^{-1}))\mathbbm{1}_{\tilde{Z}N^+}(k_2) e^{\psi(I_1((k_1^*)^{-1})-I_2(k_2^{-1}))} \, d\nu_\psi(k_1^+) \, d\nu^K_{\psi\circ\involution}(k_2).
		\end{align*}
		
		{\bf Step 2:} We conclude by appropriately decomposing $k_1^*$ and $k_2$ and applying the asymptotic in \cref{lem:MainTermAsymptotic}.
		
		Considering $k_1 \in K$ such that $(k_1^*)^{-1} = ma_wnh^{-1} \in MANN^+,$ we have 
		\begin{align*}
			J_1((k_1^*)^{-1}))^{-1} & = h;
			\\
			h^+ & = k_1M;
			\\
			k_1^* H_1((k_1^*)^{-1}) & = hn^{-1}a_{-w};
			\\
			I_1((k_1^*)^{-1}) & = w = \beta_{h^+}(e,h)
		\end{align*}
		and hence,
		\begin{align*}
			& \int_{(J_1((k_1^*)^{-1}))^{-1} \in \Xi_1} \mathbbm{1}_{\tilde{Z}N}(k_1^* 	H_1((k_1^*)^{-1}))e^{\psi(I_1((k_1^*)^{-1}))}  \, d\nu_\psi(k_1^+)
			\\
			& = \int_{h \in \Xi_1} \mathbbm{1}_{\tilde{Z}N}(h)e^{\psi(\beta_{h^+}(e,h))}  \, 	d\nu_\psi(h^+) = \tilde{\nu}_{\psi}|_{\tilde{Z}N}(\Xi_1).
		\end{align*}
		
		Similarly, considering $k_2^{-1} \in K$ such that $k_2^{-1} = a_whm^{-1}n^{-1} \in AN^+MN,$ we have
		\begin{align*}
			J_2(k_2^{-1}) & = m^{-1}n^{-1} \in MN;
			\\
			k_2^- & = n^-;
			\\
			I_2(k_2^{-1}) &=w=-i(\beta_{n^-}(e,n))
		\end{align*}
		and 
		\begin{align*}
			& \int_{\substack{k_2 \in K \\ J_2(k_2^{-1}) \in \Theta\Xi_2}} \mathbbm{1}_{\tilde{Z}N^+}(k_2) e^{-\psi(I_2(k_2^{-1}))}d\nu^K_{\psi\circ\involution}(k_2)
			\\
			& = \int_{nm \in \Xi_2^{-1}\Theta^{-1}}\mathbbm{1}_{\tilde{Z}N^+}(nm)e^{(\psi\circ\involution)(\beta_{n^-}(e,n))}\, d\nu_{\psi\circ\involution}(n^-) \, dm
			\\
			& = \tilde{\nu}_{\psi\circ\involution}|_{\tilde{Z}N^+}(\Xi_2^{-1}\Theta^{-1}).
		\end{align*}
		
		Hence, we obtain
		\begin{align*}
			& \langle F_{S_{T,\varepsilon}^\pm}, \Psi_\varepsilon \otimes \Psi_\varepsilon \rangle
			\\
			& = (1+O(\varepsilon))\frac{\kappa_{\mathsf{v}}}{|m_{\mathcal{X}_\psi}|}\int_{a_{t\mathsf{v}+\sqrt{t}u} \in \mathfrak{C}_T}e^{t}e^{-I(u)} \, dt \, du 
			\\  
			& \qquad \qquad \qquad \qquad \qquad \qquad \qquad\cdot\sum_{\tilde{Z} \in \tilde{\mathfrak{Z}}_\Gamma}\tilde{\nu}_{\psi}|_{\tilde{Z}N}(\Xi_1)\tilde{\nu}_{\psi\circ\involution}|_{\tilde{Z}N^+}(\Xi_2^{-1}\Theta^{-1}) 
			\\
			& \qquad +\int_{h^{-1}a_{t\mathsf{v}+\sqrt{t}u}mn \in S_{T,\varepsilon}^\pm}e^{t}E_\varepsilon(t,u,h,mn) \, dt \, du \, dh \, dm  \, dn +e^To_\varepsilon(1).
		\end{align*}
		
		Using \cref{lem:MainTermAsymptotic} and \cref{lem:InequalityForCountingInGamma}, taking $T\to\infty$ and then $\varepsilon \to 0$, we conclude that
		\[\#(\Gamma \cap S_T) \sim e^{T}\frac{\kappa_{\mathsf{v}}}{|m_{\mathcal{X}_\psi}|} \int_{\ker\psi}e^{- I(u)} \, du \sum_{\tilde{Z} \in \tilde{\mathfrak{Z}}_\Gamma}\tilde{\nu}_{\psi}|_{\tilde{Z}N}(\Xi_1)\tilde{\nu}_{\psi\circ\involution}|_{\tilde{Z}N^+}(\Xi_2^{-1}\Theta^{-1})\]
		and we note that $\kappa_{\mathsf{v}} \int_{\ker\psi}e^{- I(u)} \, du = [M:M_\Gamma]$ by \cref{prop:MixingConstant}.
	\end{proof}

	\begin{lemma}
		\label{lem:MainTermAsymptotic}
		We have 
		\begin{equation}
			\label{eqn:MainTermAsymptotic}
			\lim\limits_{T\to\infty}e^{-T}\int_{a_{t\mathsf{v}+\sqrt{t}u} \in \mathfrak{C}_T}e^{t} e^{- I(u)} \, dt \, du =\int_{\ker\psi}e^{- I(u)} \, du
		\end{equation}
		and for all sufficiently small $\varepsilon > 0$, 
		\begin{equation}
			\label{eqn:ErrorTermAsymptotic}
			\lim\limits_{T\to\infty}e^{-T}\int_{h^{-1}a_{t\mathsf{v}+\sqrt{t}u}mn \in S_{T,\varepsilon}^\pm}e^{t}E_\varepsilon(t,u,h,mn) \, dt \, du \, dh \, dm  \, dn = 0.
		\end{equation}
	\end{lemma}
	
	\begin{proof}
		First, we show \eqref{eqn:MainTermAsymptotic}. For $u \in \ker\psi$ and $T > 0$, let 
		\[R_T(u) = \{t>0: a_{t\mathsf{v}+\sqrt{t}u} \in \mathfrak{C}_T\} = \{0 < t \le T : t\mathsf{v}+\sqrt{t}u \in \mathfrak{C}\}.\] 
		Then we have
		\[e^{-T}\int_{a_{t\mathsf{v}+\sqrt{t}u} \in \mathfrak{C}_T}e^{t} e^{- I(u)} \, dt \, du = \int_{\ker\psi}e^{- I(u)} e^{-T}\int_{R_T(u)} e^{t}\, dt \, du.\]
		Observe that $e^{- I(u)} e^{-T}\int_{R_T(u)} e^{t} \, dt \le e^{- I(u)}$ and by definition of $I(u)$ \eqref{eqn:IDefinition}, $e^{-I(u)} \in L^1(\ker\psi)$. Then by the Lebesgue dominated convergence theorem, 
		\begin{align*}
			& \lim\limits_{T\to\infty}e^{-T}\int_{a_{t\mathsf{v}+\sqrt{t}u} \in \mathfrak{C}_T}e^{t} e^{- I(u)} \, dt \, du 
			\\
			& = \int_{\ker\psi}e^{- I(u)}\lim\limits_{T\to\infty} e^{-T}\int_{R_T(u)} e^{t}\, dt \, du = \int_{\ker\psi}e^{- I(u)} \, du,
		\end{align*}
		where the last equality uses the observation that for fixed $u$, $t \in R_T(u)$ for all sufficiently large $t \le T$. 
		
		Now we show \eqref{eqn:ErrorTermAsymptotic}. Since $\Xi_1 \subset N^+$ and $\Xi_2 \subset N$ are bounded, by \cref{thm:DecayofMatrixCoefficients}, there exists positive constants $\eta_{\mathsf{v}}, D_{\mathsf{v}}$ and $T_{\mathsf{v}}$ such that  
		\[|E_\varepsilon(t,u,h,mn)| \le D_{\mathsf{v}} e^{-\eta_{\mathsf{v}} I(u)}\] 
		for all $(t,u) \in (T_{\mathsf{v}},\infty) \times \ker\psi_\mathsf{v}$ and $h,mn$ such that $h^{-1}a_{t\mathsf{v}+\sqrt{t}u}mn \in S_{T,\varepsilon}^\pm$. Then \eqref{eqn:ErrorTermAsymptotic} follows by using similar reasoning as before and the fact that for fixed $u \in \ker\psi$, $\lim\limits_{t\to\infty}E_\varepsilon(t,u,h,mn) = 0$.
	\end{proof}
	
	\subsection{\texorpdfstring{Counting in $\mathcal{B}(g_0,\varepsilon)\mathfrak{C}_T\Theta\mathcal{B}(g_0,\varepsilon)^{-1}$}{Counting in VT}}
	\label{subsec:CountinginVT}
	
	In this subsection, we relate relate $\mathcal{V}_T(g_0,\varepsilon,\mathfrak{C},\Theta)$ \eqref{eqn:VTDefinition} and $g_0S_T(N^+_\varepsilon,N_\varepsilon^{-1},\mathfrak{C},\Theta)g_0^{-1}$ \eqref{eqn:STDefinition} and then apply \cref{prop:STAsymptotic} to get an asymptotic for $\#(\Gamma \cap \mathcal{V}_T(g_0,\varepsilon,\mathfrak{C},\Theta))$ (\cref{prop:CountingInVT}).  
	
	For Borel maps $f_1:\Gamma \backslash G/M \to \R$, $f_2:M \to\R$ and Borel sets $B_1 \subset \Gamma \backslash G/M$ and $B_2 \subset M$, we define
	\begin{align*}
		\mathsf{m}_\psi(f_1\otimes f_2) & := \left(\int_{X}f_1 \, d\mathsf{m}_\psi \right)\left( \int_M f_2 \, dm  \right);
		\\
		\mathsf{m}_\psi(B_1 \otimes B_2) & := \mathsf{m}_\psi(\mathbbm{1}_{B_1}\otimes \mathbbm{1}_{B_2}).
	\end{align*}
	Denote the projection of $\mathcal{B}(g_0,\varepsilon)$ to $\Gamma \backslash G/M$ by 
	\[\tilde{\mathcal{B}}(g_0,\varepsilon) := \Gamma \mathcal{B}(g_0,\varepsilon)M \subset \Gamma \backslash G/M.\] 
	Recall $\rank = \rankG$. Let $b_{\rank}(\varepsilon)$ denote the volume of the Euclidean $\rank$-ball of radius $\varepsilon$. We have the following formula for $\mathsf{m}_\psi(\tilde{\mathcal{B}}(g_0,\varepsilon) \otimes \Theta)$.
	
	\begin{lemma}
		\label{lem:FloxBoxBMSMeasureEstimate}
		For any $g_0 \in G$, $\varepsilon > 0$ and Borel subset $\Theta$ of $M$, we have
		\[\mathsf{m}_\psi(\tilde{\mathcal{B}}(g_0,\varepsilon) \otimes \Theta) = (1+O(\varepsilon)) b_{\rank}(\varepsilon)\tilde{\nu}_{\psi}^{g_0}(N_\varepsilon^+)\tilde{\nu}_{\psi\circ\involution}^{g_0}(N_\varepsilon\Theta),\]
		where $\tilde{\nu}_{\psi}^{g_0}$ and $\tilde{\nu}_{\psi\circ\involution}^{g_0}$ are defined by \eqref{eqn:LiftsOfConformalMeasuresToN+} and \eqref{eqn:LiftsOfConformalMeasuresToNM1}.
	\end{lemma}
	
	\begin{proof}
		We assume that $g_0 = e$. The proof for general $g_0 \in G$ is similar. By \cref{lem:FlowBoxProperties}\ref{itm:FlowBoxProperty2} and \ref{itm:FlowBoxProperty3}, we have $\mathcal{B}(e,\varepsilon)e^+ = N_\varepsilon^+e^+$, $\mathcal{B}(e,\varepsilon)e^- = N_\varepsilon e^-$ and the projection of $\mathcal{B}(e,\varepsilon)M \subset G/M \cong \Fboundary^{(2)} \times \LieA$ into $\Fboundary^{(2)}$ is $N_\varepsilon^+e^+ \times N_\varepsilon e^-$. Fix $g \in \mathcal{B}(e,\varepsilon)$ and write $w = \beta_{g^+}(e,g)$, $g = h_1 n_1 a_1 m_1 \in N^+_\varepsilon N A_\varepsilon M_\varepsilon$ and $g = n_2 h_2 a_2 m_2 \in N_\varepsilon N^+ A_{O(\varepsilon)}M_{O(\varepsilon)}$. Using \eqref{eqn:BMSMeasureDefinition} and properties of the Busemann function (\cref{def:BusemannFunction}), the lemma follows immediately from the following computation:
		\begin{align*}
			&\mathsf{m}_\psi(\mathcal{B}(e,\varepsilon)M) 
			\\
			& =\int_{g \in \mathcal{B}(e,\varepsilon)}e^{\psi\left(\beta_{g^+}(e,g)\right) + (\psi \circ \involution)\left(\beta_{g^-}(e,g)\right)} \, d\nu_\psi(g^+) \, d\nu_{\psi \circ\involution}(g^-) \, dw
			\\
			& =\int_{g \in \mathcal{B}(e,\varepsilon)}e^{\psi\left(\beta_{h_1^+}(e,h_1 a_1)\right) + (\psi \circ \involution)\left(\beta_{(n_2)^-}(e,n_2 a_2)\right)} \, d\nu_\psi(h_1^+) \, d\nu_{\psi \circ\involution}                                 (n_2^-) \, dw
			\\
			& =\int_{g \in \mathcal{B}(e,\varepsilon)}e^{\psi\left(\beta_{h_1^+}(e,h_1)+\log(a_1)\right) + (\psi \circ \involution)\left(\beta_{n_2^-}(e,n_2)-\involution(\log(a_2))\right)}
			\\ 
			& \qquad \qquad \qquad \qquad \qquad \qquad \qquad \qquad \qquad \qquad \qquad d\nu_\psi(h_1^+) \, d\nu_{\psi \circ\involution}(n_2^-) \, dw
			\\
			& =\int_{g\in \mathcal{B}(e,\varepsilon)}e^{\psi\left(\log(a_1) - \log(a_2)\right)} \, d\tilde{\nu}_{\psi}^{g_0}(h_1) \, d\tilde{\nu}_{\psi\circ\involution}^{g_0}(n_2M) \, dw
			\\
			& =(1+O(\varepsilon)) b_{\rank}(\varepsilon)\tilde{\nu}_{\psi}^{g_0}(N_\varepsilon^+)\tilde{\nu}_{\psi\circ\involution}^{g_0}(N_\varepsilon M).
		\end{align*}
	\end{proof}
	
	The next lemma relates $\mathcal{V}_T(e,\varepsilon,\mathfrak{C},\Theta)$ and $S_T(N^+_\varepsilon,N_\varepsilon^{-1},\mathfrak{C},\Theta)$. To state the lemma, we fix a constant $C >0$ such that 
	\begin{equation}
		\label{eqn:CDefinition}
		a_{-w}N_\varepsilon a_w \subset N_{\varepsilon e^{-C\psi(w)}} \text{ and } a_{w}N_\varepsilon^+ a_{-w} \subset N_{\varepsilon e^{-C\psi(w)}}^+ \text{ for all } w \in \limitcone.
	\end{equation}
	This is possible since $\limitcone \subset \interior\LieA^+$.
	
	\begin{lemma}
		\label{lem:VTandST}
		For all $T > T'>1$ and sufficiently small $\varepsilon >0$, 
		\[S_T(N_\varepsilon^+,N_{\varepsilon}^{-1},\mathfrak{C},\Theta) \subset \mathcal{V}_T(e,\varepsilon,\mathfrak{C},\Theta)\]
		and we have
		\begin{align*}
			& \mathcal{V}_T(e,\varepsilon,\mathfrak{C},\Theta) \setminus \mathcal{V}_{T'}(e,\varepsilon,\mathfrak{C},\Theta)
			\\
			& \subset S_{T+O(\varepsilon)}(N^+_{\varepsilon + O(\varepsilon e^{-C T'})},(N_{\varepsilon + O(\varepsilon e^{-C T'})})^{-1},\hat{\mathfrak{C}}_{O(\varepsilon)},\hat{\Theta}_{O(\varepsilon)})
			\\
			& \qquad \setminus S_{T'-O(\varepsilon)}(N^+_{\varepsilon + O(\varepsilon e^{-C T'})},(N_{\varepsilon + O(\varepsilon e^{-CT'})})^{-1},\hat{\mathfrak{C}}_{O(\varepsilon)},\hat{\Theta}_{O(\varepsilon)}),
		\end{align*}
		where $\hat{\Theta}_\varepsilon := \bigcup_{m_1,m_2 \in M_\varepsilon}m_1\Theta m_2$ and $\hat{\mathfrak{C}}_\varepsilon$ is the smallest closed cone in $\LieA$ containing the $\varepsilon$-neighborhood of $\mathfrak{C}_T \setminus \mathfrak{C}_{T'}$ for all $T > T' > 1$.
	\end{lemma}
	
	\begin{proof}
		Since $N_\varepsilon^\pm \subset \mathcal{B}(e,\varepsilon)$, the first inclusion is clear from the definitions.
		
		For the second inclusion, consider 
		\[g \in \mathcal{V}_T(e,\varepsilon,\mathfrak{C},\Theta) \setminus \mathcal{V}_{T'}(e,\varepsilon,\mathfrak{C},\Theta).\] 
		There exists $g_1,g_2 \in \mathcal{B}(e,\varepsilon)$, $a \in \mathfrak{C}_T \setminus \mathfrak{C}_{T'}$ and $m \in \Theta$ such that $g=g_1mag_2^{-1}$. By \cref{lem:FlowBoxProperties}\ref{itm:FlowBoxProperty3}, we have $g_1 = h_1n_1m_1a_1 \in N_\varepsilon^+N_{O(\varepsilon)}M_\varepsilon A_\varepsilon$ and $g_2 = n_2h_2m_2a_2 \in N_\varepsilon N^+_{O(\varepsilon)}M_{O(\varepsilon)}A_{O(\varepsilon)}$. Then
		\begin{equation*}
			g = h_1n_1m_1a_1ma (m_2a_2)^{-1}h_2^{-1}n_2^{-1}.
		\end{equation*}
		
		Let $a' = aa_1a_2^{-1}$, $m'=m_1mm_2^{-1}$ and $n_3 = (m'a')^{-1}n_1m'a'$. Then $a' \in (\hat{\mathfrak{C}}_{O(\varepsilon)})_{T+O(\varepsilon)} \setminus (\hat{\mathfrak{C}}_{O(\varepsilon)})_{T' - O(\varepsilon)}$, $n_3 \in N_{O(\varepsilon e^{-C T'})}$ and $g = h_1m'a'n_3h_2^{-1}n_2^{-1}$. Using our choice of constant  $C$ in \eqref{eqn:CDefinition}, by \cref{lem:TransversalityEstimate}, we have 
		\[n_3h_2^{-1} = m_3a_3h_4n_4 \in M_{O(\varepsilon)}A_{O(\varepsilon)}N^+_{O(\varepsilon)}N_{O(\varepsilon e^{-C T'})}.\] 
		Then 
		\begin{equation*}
			g = h_1m'a'm_3a_3h_4n_4n_2^{-1} = h_5m''a''n_5,
		\end{equation*}
		where $a'' = a'a_3$, $m''=m'm_3$, $h_5 = h_1(m''a'')h_4(m''a'')^{-1}$ and $n_5 =n_4n_2^{-1}$. Note that $a'' \in (\hat{\mathfrak{C}}_{O(\varepsilon)})_{T+O(\varepsilon)} \setminus (\hat{\mathfrak{C}}_{O(\varepsilon)})_{T' - O(\varepsilon)}$, $m'' \in\hat{\Theta}_{O(\varepsilon)}$, $h_5 \in N^+_{\varepsilon+O(\varepsilon e^{-C T'})}$ and $n_5 \in N_{\varepsilon+O(\varepsilon e^{-C T'})}$  which completes the proof.
	\end{proof}
	
	To state the next proposition, we will use the following asymptotic notation. 
	For a real-valued function $f$ of $\varepsilon$ and $T$, we write 
	\begin{equation}
		\label{eqn:LittleOTDefinition}
		f = o_T(1) \iff \lim_{T \to \infty}f(\varepsilon, T)=0
	\end{equation}
	
	Recall the holonomy group $M_\Gamma$ of $\Gamma$ (\cref{def:HolonomyGroup}). By \cref{lem:VTandST} and \cref{prop:STAsymptotic}, we obtain the following asymptotic for the number of elements in $\Gamma\cap\mathcal{V}_T(g_0,\varepsilon,\mathfrak{C},\Theta)$.
	
	\begin{proposition}
		\label{prop:CountingInVT} 
		Let $g_0 \in G$ and $\Theta \subset M_\Gamma$ be a Borel set with $\int_{\partial \Theta} \, dm = 0$. For all sufficiently small $\varepsilon>0$, we have  
		\begin{multline*}
			\#(\Gamma\cap\mathcal{V}_T(g_0,\varepsilon,\mathfrak{C},\Theta)) 
			\\
			= \frac{[M:M_\Gamma]}{|m_{\mathcal{X}_\psi}|}e^{ T}\left( \frac{\mathsf{m}_\psi(\tilde{\mathcal{B}}(g_0,\varepsilon) \otimes \Theta)}{b_{\rank}(\varepsilon)}  (1+O(\varepsilon)) +o_T(1)\right).
		\end{multline*}
	\end{proposition}
	
	\begin{proof}
		Again, we give a proof for $g_0 =e$ and the general case is similar. Note that $\partial N^\pm_\varepsilon$ are proper real algebraic subvarieties of $\Fboundary$ and hence $\nu_\psi(\partial N^+_\varepsilon) = \nu_{\psi \circ \involution}(\partial N_\varepsilon) = 0$ by \cite{KO23c} (this was proved in \cite{ELO22a} for the case $\psi=\psi\circ\involution$.)
		
		By \cref{lem:VTandST} and \cref{prop:STAsymptotic}, we have
		\begin{multline}
			\label{eqn:CountinginVT1}
			\#(\Gamma\cap \mathcal{V}_T(g_0,\varepsilon,\mathfrak{C},\Theta)) 
			\\
			\ge\frac{[M:M_\Gamma]}{|m_{\mathcal{X}_\psi}|} e^T\left(\sum_{\tilde{Z} \in \tilde{\mathfrak{Z}}_\Gamma}\tilde{\nu}_{\psi}|_{\tilde{Z}N}(N^+_\varepsilon)\tilde{\nu}_{\psi\circ\involution}|_{\tilde{Z}N^+}(N_\varepsilon\Theta^{-1}) + o_T(1)\right).
		\end{multline}
		
		We claim that for sufficiently small $\varepsilon>0$, we have 
		\begin{equation}
			\label{eqn:CountinginVT2}
			\sum_{\tilde{Z} \in \tilde{\mathfrak{Z}}_\Gamma}\tilde{\nu}_{\psi}|_{\tilde{Z}N}(N^+_\varepsilon)\tilde{\nu}_{\psi\circ\involution}|_{\tilde{Z}N^+}(N_\varepsilon\Theta^{-1}) = \tilde{\nu}_{\psi}(N_\varepsilon^+)\tilde{\nu}_{\psi\circ\involution}(N_\varepsilon\Theta).
		\end{equation}
		To see why the claim is true, observe that 
		\begin{align*}
			\tilde{\nu}_{\psi}(N_\varepsilon^+)=\sum_{\tilde{Z} \in \tilde{\mathfrak{Z}}_\Gamma}\tilde{\nu}_{\psi}|_{\tilde{Z}N}(N^+_\varepsilon) & \text{ and } \tilde{\nu}_{\psi\circ\involution}(N_\varepsilon\Theta) = \sum_{\tilde{Z} \in \tilde{\mathfrak{Z}}_\Gamma}\tilde{\nu}_{\psi\circ\involution}|_{\tilde{Z}N^+}(N_\varepsilon\Theta^{-1}),
		\end{align*}
		where in the second equality we use the fact that $\tilde{Z}N$ is right $M_\Gamma$-invariant for each $\tilde{Z } \in \tilde{\mathfrak{Z}}_\Gamma$. If $e^+ \notin \limitset$ or $e^-\notin \limitset$, then for sufficiently small $\varepsilon$, either $N_\varepsilon^+ \cap \tilde{Z}N = \emptyset$ for all $\tilde{Z} \in \tilde{\mathfrak{Z}}_\Gamma$ or $N_\varepsilon \cap \tilde{Z}N^+ = \emptyset$ for all $\tilde{Z} \in \tilde{\mathfrak{Z}}_\Gamma$ and the claim is clear since both sides of the equation would be 0. Otherwise, $e^+,e^- \in \limitset$, so  for sufficiently small $\varepsilon$, $N_\varepsilon^\pm \cap \limitset \subset \tilde{Z}$ for a single $\tilde{Z} \in \tilde{\mathfrak{Z}}_\Gamma$ and hence, the claim. 
		
		Combining \eqref{eqn:CountinginVT1} and \eqref{eqn:CountinginVT2} and using \cref{lem:FloxBoxBMSMeasureEstimate}, we obtain
		\begin{align*}
			& \#(\Gamma\cap(\mathcal{V}_T(e,\varepsilon,\mathfrak{C},\Theta))) 
			\\
			& \qquad \qquad \ge \frac{[M:M_\Gamma]}{|m_{\mathcal{X}_\psi}|} e^T\left( \frac{\mathsf{m}_\psi(\tilde{\mathcal{B}}(e,\varepsilon) \otimes \Theta)}{b_{\rank}(\varepsilon)}  (1+O(\varepsilon))+o_T(1)\right).
		\end{align*}  
		
		It remains to establish the reverse inequality. By \cref{lem:VTandST}, we know  $\mathcal{V}_T(e,\varepsilon,\mathfrak{C},\Theta) \setminus \mathcal{V}_{T/2}(e,\varepsilon,\mathfrak{C},\Theta)$ is contained in 
		\begin{equation*}
			S_{T+O(\varepsilon)}(N^+_{\varepsilon + O(\varepsilon e^{-C T/2})},(N_{\varepsilon + O(\varepsilon e^{-C T/2})}^-)^{-1},\hat{\mathfrak{C}}_{O(\varepsilon)},\hat{\Theta}_{O(\varepsilon)})
		\end{equation*}
		and 
		\begin{equation*}
			\mathcal{V}_{T/2}(e,\varepsilon,\mathfrak{C},\Theta) \subset S_{T/2+O(\varepsilon)}(N^+_{O(\varepsilon)},(N_{O(\varepsilon)})^{-1},\hat{\mathfrak{C}}_{O(\varepsilon)},\hat{\Theta}_{O(\varepsilon)}).
		\end{equation*}
		Assuming $\varepsilon$ is sufficiently small so that $M_{O(\varepsilon)} \subset M^\circ \subset M_\Gamma$, similar reasoning as above yields
		\begin{equation}
			\label{eqn:CountinginVT3}
			\begin{aligned}[b]
				&\#(\Gamma\cap\mathcal{V}_T(e,\varepsilon,\mathfrak{C},\Theta)) -\#(\Gamma\cap \mathcal{V}_{T/2}(e,\varepsilon,\mathfrak{C},\Theta))
				\\
				&  \le \frac{[M:M_\Gamma]}{|m_{\mathcal{X}_\psi}|}e^{(T+O(\varepsilon))}
				\\
				& \qquad \cdot\bigg(\frac{\mathsf{m}_\psi(\tilde{\mathcal{B}}(e,\varepsilon+O(\varepsilon e^{-C T/2})) \otimes \hat{\Theta}_{O(\varepsilon)})}{b_{\rank}(\varepsilon+O(\varepsilon e^{-C T/2}))} (1+O(\varepsilon+\varepsilon e^{-C T/2}))  + o_T(1)\bigg);
			\end{aligned}
		\end{equation}
		\begin{equation}
			\label{eqn:CountinginVT4}
			\begin{aligned}[b]
				& \#(\Gamma\cap \mathcal{V}_{T/2}(e,\varepsilon,\mathfrak{C},\Theta)) 
				\\
				& \le \frac{[M:M_\Gamma]}{|m_{\mathcal{X}_\psi}|}e^{(T/2+O(\varepsilon))}\left( \frac{\mathsf{m}_\psi(\tilde{\mathcal{B}}(e,O(\varepsilon)) \otimes \hat{\Theta}_{O(\varepsilon)})}{b_{\rank}(O(\varepsilon))}(1+O(\varepsilon))+o_T(1)\right).
			\end{aligned}
		\end{equation}
		Combining \eqref{eqn:CountinginVT3} and \eqref{eqn:CountinginVT4} and taking $T$ large, we obtain the desired inequality.
	\end{proof}

	\section{\texorpdfstring{Joint equidistribution with respect to $\psi$-circumferences}{Joint equidistribution with respect to psi-circumferences}}
	\label{sec:ProofOfJointEquidistribution}
	
	In this section, we prove the main theorem of this paper, \cref{thm:JointEquidistribution}, and \cref{cor:EquidistributionOfHolonomies}. In fact, the results we prove here are slightly more general as we allow ourselves to consider only maximal flat cylinders $C$ with corresponding Jordan projections $\lambda(\gamma_C)$ lying in a fixed cone $\mathfrak{C}\subset \limitcone$, where $\gamma_C$ is as in \eqref{eqn:GammaCDefinition}. 
	
	As in \cref{sec:Counting}, for this section, we fix
	\[\psi \in \LieA^* \text{ tangent to } \growthindicator \text{ at } \mathsf{v} \in \interior \limitcone \text{ with } \psi(\mathsf{v}) = \growthindicator(\mathsf{v}) = 1\]
	and we fix 
	\[\text{ cone } \mathfrak{C}\subset \limitcone \text{ with } \mathsf{v} \in \interior\mathfrak{C}.\]
	
	\subsection{Preparatory lemmas}
	\label{subsec:PreparatoryLemmas}
	
	Throughout this subsection, we fix 
	\begin{itemize}
		\item $g_0$ with $\Gamma g_0 M\in\supp\BMS$;
		
		\item $\varepsilon>0$ sufficiently small as in \cref{prop:CountingInVT};
		
		\item a conjugation-invariant Borel set $\Theta \subset M_\Gamma$ (\cref{def:HolonomyGroup}). 
	\end{itemize} 
	For $T>0$, we set
	\[\mathcal{W}_T(g_0,\varepsilon,\mathfrak{C},\Theta)=\{gamg^{-1}:g\in\mathcal{B}(g_0,\varepsilon),am\in \mathfrak{C}_T \Theta\}.\]
	Recall that $\mathcal{C}_\Gamma$ is the set of all (positively oriented) maximal flat cylinders in $\Gamma\backslash G/M$. For $C \in \mathcal{C}_\Gamma$, choose a lift $\tilde{C}\subset G/M$. Its stabilizer $\Gamma_{\tilde{C}}$ is generated by some element in $[\gamma_C]$, and $C$ can be identified with $\Gamma_{\tilde{C}}\backslash\tilde{C}$. Let
	\[I(C):=\{\Gamma_{\tilde{C}}\sigma\in \Gamma_{\tilde{C}}\backslash\Gamma:(\sigma\mathcal{B}(g_0,\varepsilon)M)\cap\tilde{C}\neq\emptyset\}.\]
	Note that $\#I(C)$ is independent of the choice of $\tilde{C}$.

	For $T >0$, let $\mathcal{C}_\psi(T)$ denote the set of all maximal flat cylinders with $\psi$-circumference \eqref{eqn:PsiCircumferenceDefinition} at most $T$:
	\[\mathcal{C}_\psi(T) := \{C \in \mathcal{C}_\Gamma : \ell_\psi(C) \le T\}.\]
	We will also consider the set of maximal flat cylinders $C$ with corresponding Jordan projections $\lambda(\gamma_C)$ lying in the cone $\mathfrak{C}$: 
	\[\mathcal{C}_\psi(T,\mathfrak{C})=\{C\in\mathcal{C}_\psi(T):\lambda(\gamma_C)\in\mathfrak{C}\}.\]
	
	\begin{lemma}
		\label{lem:FinitelyManyCylinders}
		\begin{enumerate}
			\item
			\label{itm:FinitelyManyCylinders2}
			The set $\mathcal{C}_\psi(T)$ is finite for any $T>0$.
			
			\item 
			\label{itm:FinitelyManyCylinders1}
			For any $C=\Gamma gAM\in\mathcal{C}_\Gamma$, $I(C)$ is finite.
		\end{enumerate}
	\end{lemma}
	
	\begin{proof}
		By discreteness of $\Gamma$, for any $g_0 \in G$ and $\varepsilon>0$, there are finitely many elements of $\Gamma$ in $\mathcal{W}_T(g_0,\varepsilon,\limitcone,M)$. On the other hand, 
		\begin{align*}
			\#\Gamma\cap\mathcal{W}_T(g_0,\varepsilon,\limitcone,M) & =\#\{\gamma = ga_w m g^{-1} : g\in\mathcal{B}(g_0,\varepsilon),\, \psi(w) \le T\}
			\\
			&=\#\{\tilde{C} : \tilde{C}\cap\mathcal{B}(g_0,\varepsilon)M\neq\emptyset,\, \ell_\psi(\Gamma\backslash\Gamma \tilde{C})<T\}\\
			&=\sum_{C\in\mathcal{C}_\psi(T)}\#I(C).
		\end{align*}
		From this, \ref{itm:FinitelyManyCylinders1} immediately follows, and we see that 
		\[\#\{C\in\mathcal{C}_\psi(T):C\cap\tilde{\mathcal{B}}(g_0,\varepsilon)\neq\emptyset\} < \infty.\]
		Then \ref{itm:FinitelyManyCylinders2} follows, using the fact that every maximal flat cylinder intersects the embedded compact set $\mathcal{X}_\psi \subset \Omega$ (\cref{subsec:TheSupportOfTheBMSMeasureAsAVectorBundle}).
	\end{proof}
	
	Denote by $V_C$ the measure on $C$ induced by the Haar measure on $A$. Let $\mathrm{Cl}(M)$ denote the set of continuous real-valued class functions on $M$ and $[M]$ denote the set of conjugacy classes in $M$. For each $T>0$, we define a Radon measure $\mu_{\mathfrak{C},T}$ on $\Gamma \backslash G/M \times [M]$ as follows. For $T>0$, $f \in \mathrm{C}_{\mathrm{c}}(\Gamma \backslash G/M)$ and $\varphi \in \mathrm{Cl}(M)$, let 
	\begin{align*}
		\mu_{\mathfrak{C},T}(f\otimes\varphi)&:=\sum_{C \in\mathcal{C}_{\psi}(T,\mathfrak{C})} V_C(f)\varphi(h_C),
		\\
		\mu_T(f\otimes\varphi) & :=\mu_{\limitcone, T}(f\otimes\varphi)=\sum_{C \in\mathcal{C}_{\psi}(T)} V_C(f)\varphi(h_C).
	\end{align*}
	
	We record several appropriately adapted lemmas from \cite[Section 5]{MMO14}. \cref{lem:BoxHits} and \cref{lem:MuTAndCountingPrimitiveHyperbolics} relates the $\mu_{\mathfrak{C},T}$-measure of  $\tilde{\mathcal{B}}(g_0,\varepsilon)\otimes[\Theta]$ to the number of elements in $\Gamma\cap\mathcal{W}_T(g_0,\varepsilon,\mathfrak{C},\Theta)$. The proofs are similar to those in \cite{MMO14}, but we provide them here for the sake of completeness.
	
	\begin{lemma}
		\label{lem:BoxHits}
		\mbox{}
		\begin{enumerate}
			\item
			\label{itm:BoxHits1} 
			For any $C \in \mathcal{C}_\Gamma$, we have
			\[V_C(\tilde{\mathcal{B}}(g_0,\varepsilon))=b_{\rank}(\varepsilon)\cdot\#I(C).\]
			
			\item
			\label{itm:BoxHits2}
			For any $T>0$ and , we have
			\[\mu_{\mathfrak{C},T}(\tilde{\mathcal{B}}(g_0,\varepsilon)\otimes[\Theta])=b_{\rank}(\varepsilon)\cdot\sum_{C\in\mathcal{C}_\psi(T,\mathfrak{C})}\#I(C)\cdot 1_{[\Theta]}(h_C).\]
		\end{enumerate}
	\end{lemma}

	\begin{proof}
		It is clear from the definition of $\mu_{\mathfrak{C},T}$ that \ref{itm:BoxHits2} follows from \ref{itm:BoxHits1}. Let $C=\Gamma gAM$, and choose $\tilde{C}=gAM$. Then we have
		\begin{align*} 
			V_C (\tilde{\mathcal{B}}(g_0,\varepsilon))&=\int_{\LieA \mod \lambda(\gamma_C)}1_{\tilde{\mathcal{B}}(g_0,\varepsilon)}(\Gamma g\exp(w)M)\ dw
			\\
			&=\sum_{[\sigma]\in\Gamma_{\tilde{C}}\backslash\Gamma}\int_{\LieA \mod \lambda(\gamma_C)} 1_{[\sigma]\mathcal{B}(g_0,\varepsilon)}(g\exp(w)M)\ dw 
			\\
			& = b_{\rank}(\varepsilon)\cdot\#I(C),
		\end{align*}
		where $dw$ denotes the Lebesgue measure on $\LieA$, the integral is over a fundamental domain in $\LieA$ for the $\Z$-action given by translation by $\lambda(\gamma_C)$ and the last equality uses \cref{lem:FlowBoxProperties}\ref{itm:FlowBoxProperty1}.
	\end{proof}
	
	Recall that $\primGamma$ denotes the set of primitive elements of $\Gamma$. 
	
	\begin{lemma}
		\label{lem:MuTAndCountingPrimitiveHyperbolics}
		For all sufficiently large $T$, we have
		\[\mu_{\mathfrak{C},T}(\tilde{\mathcal{B}}(g_0,\varepsilon)\otimes[\Theta])=b_{\rank}(\varepsilon)\cdot\#(\primGamma\cap\mathcal{W}_T(g_0,\varepsilon,\mathfrak{C},\Theta)).\]
	\end{lemma}

	\begin{proof}
		Let $C=\Gamma gAM\in\mathcal{C}_\psi (T,\mathfrak{C})$ and wihtout loss of generality, assume that $\tilde{C} = gAM$. By \cref{lem:BoxHits}\ref{itm:BoxHits1}, it suffices to show that
		\begin{equation}
			\label{eqn:MuTAndCountingPrimitiveHyperbolics1}
			\# I(C)\cdot 1_{[\Theta]} (h_C) = \#([\gamma_C]\cap\mathcal{W}_T(g_0,\varepsilon,\mathfrak{C},\Theta)).
		\end{equation}
		If $h_C \notin [\Theta]$, then \eqref{eqn:MuTAndCountingPrimitiveHyperbolics1} is clear. We assume that $h_C \in [\Theta]$. If $\Gamma_{\tilde{C}}\sigma\in I(C)$, then $\sigma^{-1}g\in \mathcal{B}(g_0,\varepsilon)AM$ and 
		\[\sigma^{-1}\gamma_{C,g} \sigma=\sigma^{-1}g \exp(\lambda(\gamma_C)) m_{C,g}g^{-1}\sigma\in[\gamma_C]\cap\mathcal{W}_T(g_0,\varepsilon,\mathfrak{C},\Theta).\]
		Conversely, if $\sigma \in \Gamma$ such that $\sigma^{-1}\gamma_{C,g}\sigma\in\mathcal{W}_T(g_0,\varepsilon,\mathfrak{C},\Theta)$, then we have $\gamma_{C,g}=\sigma g'am(g')^{-1}\sigma^{-1}$ where $g'\in\mathcal{B}(g_0,\varepsilon)$ and $am\in \mathfrak{C}_T \Theta$, but on the other hand, $\gamma_{C,g} = g\exp(\lambda(\gamma_C))m_{C,g}g^{-1}$, so $\sigma g' \in gAM$, and hence, $\sigma\mathcal{B}(g_0,\varepsilon) \cap gAM \ne \emptyset$, that is, $\Gamma_{\tilde{C}}\sigma\in I(C)$. 
		\[\sigma g' \in \sigma\mathcal{B}(g_0,\varepsilon) \cap gAM.\]
		
		Noting that the map $\Gamma_{\tilde{C}}\sigma \in \Gamma_{\tilde{C}} \backslash \Gamma \mapsto \sigma^{-1}\gamma_{C,g} \sigma \in [\gamma_C]$ is bijective since the centralizer of $\gamma_{C,g}$ in $\Gamma$ is $\langle \gamma_{C,g} \rangle = \Gamma_{\tilde{C}}$, we have thus given a bijection between $I(C)$ and $[\gamma_C]\cap\mathcal{W}_T(g_0,\varepsilon,\mathfrak{C},\Theta)$.
	\end{proof}
	
	In the next lemma, we use the effective closing lemma for regular directions (\cref{lem:EffectiveClosingLemma}) to relate $\Gamma \cap \mathcal{W}_T(g_0,\varepsilon,\mathfrak{C},\Theta)$ with $\Gamma \cap \mathcal{V}_T(g_0,\varepsilon,\mathfrak{C},\Theta)$.
	
	\begin{lemma}
		\label{lem:ComparisonLemmaForCountingInVTandWT}
		There exists $T_1 > 0$, depending only on $\Gamma$ and $\psi$, such that for all sufficiently large $T$, we have 
		\begin{multline*}
			\Gamma \cap \left(\mathcal{V}_{T-O(\varepsilon)}(g_0,\varepsilon(1-O(e^{-T})),\check{\mathfrak{C}}_{O(\varepsilon)},\check{\Theta}_{O(\varepsilon)})\setminus \mathcal{V}_{T_1}(g_0,\varepsilon,\mathfrak{C},\Theta)\right)
			\\
			\subset \Gamma \cap \mathcal{W}_T(g_0,\varepsilon,\mathfrak{C},\Theta),
		\end{multline*}
		where $\check{\Theta}_{\varepsilon} := \bigcap_{g_1,g_2 \in G_\varepsilon}g_1\Theta g_2$ and $\check{\mathfrak{C}}_{\varepsilon}$ is any cone containing $\R\mathsf{v}$ such that the $\varepsilon$-neighborhood of $\check{\mathfrak{C}}_{\varepsilon} \setminus (\check{\mathfrak{C}}_{\varepsilon})_{T_1}$ is contained in $\mathfrak{C}$.
	\end{lemma}
	
	\begin{proof}
		Fix $T_1 > 0$ such that if $w \in \limitcone$ and $\psi(w) \ge T_1$, then $\min_{\alpha \in \Phi^+} \alpha(w) \ge T_0$, where $T_0$ is as in \cref{lem:EffectiveClosingLemma}. Suppose 
		\[\gamma\in \Gamma \cap \left(\mathcal{V}_{T-O(\varepsilon)}(g_0,\varepsilon(1-O(e^{-T})),\check{\mathfrak{C}}_{O(\varepsilon)},\check{\Theta}_{O(\varepsilon)})\setminus \mathcal{V}_{T_1}(g_0,\varepsilon,\mathfrak{C},\Theta)\right).\]
		Then $\gamma = g_1\exp(w)mg_2$, where $g_1,g_2 \in \mathcal{B}(g_0,\varepsilon(1-O(e^{-T}))$, $w \in \check{\mathfrak{C}}_{O(\varepsilon)}$ with $T_1< \psi(w) \le T-O(\varepsilon)$, and $m \in \check{\Theta}_{O(\varepsilon)}$. By \cref{lem:EffectiveClosingLemma}, we have $\gamma = g\exp(w')m'g^{-1}$ for some $g \in \mathcal{B}(g_0,\varepsilon)$, $w' \sim_{O(\varepsilon)} w$ and $m \sim_{O(\varepsilon)} m'$. It follows that $w\in\mathfrak{C}$, $\psi(w') \le T$ and $m' \in \Theta$, so $\gamma \in \Gamma \cap \mathcal{W}_T(g_0,\varepsilon,\mathfrak{C},\Theta)$. 
	\end{proof}
	
	Let $m_\Gamma$ denote the Haar probability measure on $M_\Gamma$. The following lemma gives a lower bound for $\#\primGamma\cap\mathcal{W}_T(g_0,\varepsilon,\mathfrak{C},\Theta)$.
	
	\begin{lemma}
		\label{lem:BoundOnCountingPrimitiveHyperbolics}
		Suppose that $m_\Gamma(\Theta) > 0$ and $m_\Gamma(\partial \Theta) = 0$. Then for all sufficiently large $T$, we have
		\[\#\Gamma\cap(\mathcal{W}_T(g_0,\varepsilon,\Theta)\setminus\mathcal{W}_{2T/3}(g_0,\varepsilon,\Theta))\le\#\primGamma\cap\mathcal{W}_T(g_0,\varepsilon,\Theta).\]
	\end{lemma}
	
	\begin{proof}
		Let $\Gamma_{\mathrm{prim}^k}=\{\sigma^k:\sigma\in\primGamma\}$. We observe that
		\begin{align*}
			&\#\primGamma\cap\mathcal{W}_T(g_0,\varepsilon,\mathfrak{C},\Theta)&
			\\
			& \qquad =\#\left(\Gamma\cap\mathcal{W}_T(g_0,\varepsilon,\mathfrak{C},\Theta)\right)-\#\left(\bigcup_{k\ge 2}\Gamma_{\mathrm{prim}^k}\cap\mathcal{W}_T(g_0,\varepsilon,\mathfrak{C},\Theta)\right)
			\\
			&  \qquad \ge \#\left(\Gamma\cap\mathcal{W}_T(g_0,\varepsilon,\mathfrak{C},\Theta)\right)-\#\left(\bigcup_{k\ge 2}\Gamma\cap\mathcal{W}_{T/k}(g_0,\varepsilon,\mathfrak{C},\sqrt[k]{\Theta})\right),
		\end{align*}
		where $\sqrt[k]{\Theta}: = \{m \in M : m^k \in \Theta\}$. It suffices to show that for all sufficiently large $T$, we have
		\begin{equation}
			\label{eqn:BoundOnCountingPrimitiveHyperbolics1}
			\#\left(\bigcup_{k\ge 2}\Gamma\cap\mathcal{W}_{T/k}(g_0,\varepsilon,\mathfrak{C},\sqrt[k]{\Theta})\right) \le \#\Gamma\cap\mathcal{W}_{2T/3}(g_0,\varepsilon,\mathfrak{C},\Theta).
		\end{equation}
		Since $\mathcal{W}_{T/k}(g_0,\varepsilon,\mathfrak{C},\sqrt[k]{\Theta}) \subset \mathcal{V}_{T/k}(g_0,\varepsilon,\mathfrak{C},\sqrt[k]{\Theta})$ and $\Gamma\cap\mathcal{W}_{T/k}(g_0,\varepsilon,\mathfrak{C},\sqrt[k]{\Theta})$ is empty when $T/k$ is sufficiently small, using \cref{prop:CountingInVT}, it follows that
		\begin{equation}
			\label{eqn:BoundOnCountingPrimitiveHyperbolics2}	
			\#\left(\bigcup_{k\ge2}\Gamma\cap\mathcal{W}_{T/k}(g_0,\varepsilon,\mathfrak{C},\sqrt[k]{\Theta})\right)=O(Te^{ T/2}).
		\end{equation}
		On the other hand, using \cref{lem:ComparisonLemmaForCountingInVTandWT}, \cref{prop:CountingInVT} and the assumption that $m_\Gamma(\Theta) > 0$, we have 
		\begin{equation}
			\label{eqn:BoundOnCountingPrimitiveHyperbolics3}
			\#\Gamma\cap\mathcal{W}_{2T/3}(g_0,\varepsilon,\Theta) \ge O(e^{2T/3}).
		\end{equation}
		The inequality \eqref{eqn:BoundOnCountingPrimitiveHyperbolics1} now follows from \eqref{eqn:BoundOnCountingPrimitiveHyperbolics2} and \eqref{eqn:BoundOnCountingPrimitiveHyperbolics3}.
	\end{proof}
	
	\cref{lem:MuTAndCountingPrimitiveHyperbolics,lem:ComparisonLemmaForCountingInVTandWT,lem:BoundOnCountingPrimitiveHyperbolics} imply the following lemma.
	
	\begin{lemma}[Comparison Lemma]  
		\label{lem:ComparisonLemma}
		Suppose that $m_\Gamma(\Theta) > 0$ and $m_\Gamma(\partial\Theta) \allowbreak = 0$. For all sufficiently large $T$, we have
		\begin{multline*}
			b_{\rank}(\varepsilon)\cdot\#\Gamma\cap\left(\mathcal{V}_{T-O(\varepsilon)}(g_0,\varepsilon(1-O(e^{-T})),\check{\mathfrak{C}}_{O(\varepsilon)},\check{\Theta}_{O(\varepsilon)})\setminus\mathcal{V}_{2T/3}(g_0,\varepsilon,\mathfrak{C},\Theta)\right)
			\\
			\le \mu_{\mathfrak{C},T}(\tilde{\mathcal{B}}(g_0,\varepsilon)\otimes\Theta)\le b_{\rank}(\varepsilon)\cdot\#\Gamma\cap\mathcal{V}_T(g_0,\varepsilon,\mathfrak{C},\Theta).
		\end{multline*}
	\end{lemma}
	
	\begin{proof}
		The upper bound follows directly from \cref{lem:MuTAndCountingPrimitiveHyperbolics} and the trivial inclusion $\Gamma\cap\mathcal{W}_T(g_0,\varepsilon,\mathfrak{C},\Theta) \subset \Gamma\cap\mathcal{V}_T(g_0,\varepsilon,\mathfrak{C},\Theta)$. The lower bound follows by using \cref{lem:MuTAndCountingPrimitiveHyperbolics}, \cref{lem:BoundOnCountingPrimitiveHyperbolics}, and \cref{lem:ComparisonLemmaForCountingInVTandWT}:
		\begin{align*}
			& \mu_{\mathfrak{C},T}(\tilde{\mathcal{B}}(g_0,\varepsilon)\otimes\Theta)
			\\
			& = b_{\rank}(\varepsilon)\#(\primGamma\cap\mathcal{W}_T(g_0,\varepsilon,\mathfrak{C},\Theta))
			\\
			& \ge b_{\rank}(\varepsilon)\#\Gamma\cap(\mathcal{W}_T(g_0,\varepsilon,\mathfrak{C},\Theta)\setminus\mathcal{W}_{2T/3}(g_0,\varepsilon,\mathfrak{C},\Theta))
			\\
			& \ge
			b_{\rank}(\varepsilon)\#\Gamma\cap\left(\mathcal{V}_{T-O(\varepsilon)}(g_0,\varepsilon(1-O(e^{-T})),\check{\mathfrak{C}}_{O(\varepsilon)},\check{\Theta}_{O(\varepsilon)})\setminus\mathcal{V}_{2T/3}(g_0,\varepsilon,\mathfrak{C},\Theta)\right).
		\end{align*}
	\end{proof}
	
	Combining \cref{prop:CountingInVT} and \cref{lem:ComparisonLemma}, we obtain the following asymptotic for $\mu_{\mathfrak{C},T}(\tilde{\mathcal{B}}(g_0,\varepsilon)\otimes\Theta)$.
	
	\begin{proposition}        
		\label{prop:MuTOfFlowBox}
		For all $g_0 \in \supp\BMS$, for all conjugation-invariant Borel subsets $\Theta \subset M_\Gamma$ with $\int_{\Theta} \, dm_\Gamma > 0$ and $\int_{\partial \Theta} \, dm=0$ and for all sufficiently small $\varepsilon>0$, we have
		\[\mu_{\mathfrak{C},T}(\tilde{\mathcal{B}}(g_0,\varepsilon)\otimes\Theta)=\frac{[M:M_\Gamma]}{|m_{\mathcal{X}_\psi}|}e^{ T}\left(\BMS(\tilde{\mathcal{B}}(g_0,\varepsilon) \otimes \Theta)(1+O(\varepsilon))+o_T(1) \right).\]
	\end{proposition}
	
	\begin{proof}
		Inputting \cref{prop:CountingInVT} into \cref{lem:ComparisonLemma} yields
		\[\mu_{\mathfrak{C},T}(\tilde{\mathcal{B}}(g_0,\varepsilon)\otimes\Theta) \le \frac{[M:M_\Gamma]}{|m_{\mathcal{X}_\psi}|}e^{ T}\left( \BMS(\tilde{\mathcal{B}}(g_0,\varepsilon) \otimes \Theta)(1+O(\varepsilon))+o_T(1) \right)\]
		and
		\begin{align*}
			&\mu_{\mathfrak{C},T}(\tilde{\mathcal{B}}(g_0,\varepsilon)\otimes\Theta)
			\\
			&\ge \frac{[M:M_\Gamma]}{|m_{\mathcal{X}_\psi}|} e^{T-O(\varepsilon)}
			\\
			& \qquad \qquad \quad \cdot\left(\frac{b_{\rank}(\varepsilon)\BMS(\tilde{\mathcal{B}}(g_0,\varepsilon(1-O(e^{-T}))) \otimes \check{\Theta}_{O(\varepsilon)})}{b_{\rank}(\varepsilon(1-O(e^{-T})))}(1+O(\varepsilon))+o_T(1) \right)
			\\
			& \qquad \qquad - \frac{[M:M_\Gamma]}{|m_{\mathcal{X}_\psi}|}e^{2T/3}\left(\BMS(\tilde{\mathcal{B}}(g_0,\varepsilon) \otimes \Theta)(1+O(\varepsilon))+o_T(1) \right)
			\\
			& = \frac{[M:M_\Gamma]}{|m_{\mathcal{X}_\psi}|}e^{ T}\left( \BMS(\tilde{\mathcal{B}}(g_0,\varepsilon) \otimes \Theta)(1+O(\varepsilon))+o_T(1) \right).
		\end{align*}
	\end{proof}
	
	\subsection{Proofs of the main results}
	\label{subsec:ProofofMainTheorems}
	
	We are now ready to prove our main joint equidistribution theorem and its corollaries. We note that \cref{thm:JointEquidistribution} and \cref{cor:EquidistributionOfHolonomies} are special cases of the following statements when $\mathfrak{C} = \limitcone$. 
	
	\begin{theorem}[Joint equidistribution] 
		\label{thm:MuTJointEquidistribution}
		For any $f\in C_{\mathrm{c}}(\Gamma \backslash G/M)$, $\varphi \in \mathrm{Cl}(M)$ and cone $\mathfrak{C}\subset \limitcone$ with $\mathsf{v}\in\interior\mathfrak{C}$, we have
		\begin{equation*}\lim\limits_{T \to \infty}\frac{\mu_{\mathfrak{C},T}(f\otimes\varphi)}{e^{T}} = \frac{\BMS(f) }{|m_{\mathcal{X}_\psi}|}\int_{M_\Gamma} \varphi \, dm_\Gamma.
		\end{equation*}   
		
	\end{theorem}
	
	\begin{proof}
		The theorem is clear if $\supp\varphi \cap M_\Gamma = \emptyset$. We first prove the theorem when $\varphi = \mathbbm{1}_\Theta$ for some conjugation-invariant Borel subset $\Theta \subset M_\Gamma$ with $\int_{\Theta} \, dm_\Gamma > 0$ and $\int_{\partial \Theta} \, dm =0$. By using a partition of unity, we may assume without loss of generality that $f$ is supported on $\tilde{\mathcal{B}}(g_0,\varepsilon)$ for some $g_0 \in \supp\BMS$ and a $\varepsilon >0$ as in \cref{prop:CountingInVT}. Then we can approximate $f$ arbitrarily well by a linear combination of characteristic functions on boxes of the form $\tilde{\mathcal{B}}(h,\rho)$, with $h \in \tilde{\mathcal{B}}(g_0,\varepsilon)$ and arbitrarily small $\rho>0$. Then by applying \cref{prop:MuTOfFlowBox} to each $\mathbbm{1}_{\tilde{\mathcal{B}}(h,\rho)} \otimes \mathbbm{1}_{\Theta}$, we obtain
		\begin{align*}
			(1-O(\rho))\frac{[M:M_\Gamma]}{|m_{\mathcal{X}_\psi}|}\BMS(f \otimes \mathbbm{1}_{\Theta}) &\le \liminf_T e^{- T}\mu_{\mathfrak{C},T}(f\otimes\mathbbm{1}_{\Theta})
			\\
			& \le \limsup_T e^{- T}\mu_{\mathfrak{C},T}(f\otimes\mathbbm{1}_{\Theta})
			\\
			& \le (1+O(\rho))\frac{[M:M_\Gamma]}{|m_{\mathcal{X}_\psi}|}\BMS(f\otimes \mathbbm{1}_{\Theta})
		\end{align*}
		and hence 
		\begin{equation*}
			\begin{aligned}[b]
				\lim_Te^{- T}\mu_{\mathfrak{C},T}(f\otimes\mathbbm{1}_{\Theta}) & = \frac{[M:M_\Gamma]}{|m_{\mathcal{X}_\psi}|}\BMS(f\otimes\mathbbm{1}_{\Theta}) 
				\\
				& =\frac{\BMS(f)}{|m_{\mathcal{X}_\psi}|}\int_{M_\Gamma} \mathbbm{1}_{\Theta}\, dm_\Gamma.
			\end{aligned}
		\end{equation*}
		After identifying $M$ with the quotient of the Lie algebra of its maximal torus by the Weyl group relative to its maximal torus, a similar approximation argument can be used for general $\varphi \in \mathrm{Cl}(M)$. 
	\end{proof}
	
	For each $T>0$, we define the Radon measure $\eta_{\mathfrak{C},T}$ on $\Gamma \backslash G/M\times [M]$ in a similar fashion as $\mu_{\mathfrak{C},T}$, but normalizing the Haar measure on maximal flat cylinders by their $\psi$-circumferences. For $T>0$, $f\in\mathrm{C}_{\mathrm{c}}(\Gamma \backslash G/M)$ and $\varphi \in \mathrm{Cl}(M)$, let 
	\[\eta_{\mathfrak{C},T}(f\otimes\varphi):=\sum_{C \in\mathcal{C}_{\psi}(T,\mathfrak{C})} \frac{V_C(f)}{\ell_\psi(C)}\varphi(h_C).\]
	
	\begin{corollary} 
		\label{cor:EtaTJointEquidistribution}
		For any $f\in C_{\mathrm{c}}(\Gamma \backslash G/M)$, $\varphi \in \mathrm{Cl}(M)$ and cone $\mathfrak{C}\subset\limitcone$ with $\mathsf{v}\in\interior\mathfrak{C}$, we have
		\[\lim_{T\to\infty} \frac{\eta_{\mathfrak{C},T}(f\otimes\varphi)}{e^{ T}/T} = \frac{\BMS(f)}{|m_{\mathcal{X}_\psi}|}\int_{M_\Gamma} \varphi \, dm_\Gamma.\]
	\end{corollary}
	
	\begin{proof}
		We observe that $T\eta_{\mathfrak{C},T} \ge \mu_{\mathfrak{C},T}$, and for any $\varepsilon>0$, 
		\begin{align*}
			& Te^{- T}\eta_{\mathfrak{C},T}(f\otimes\varphi)
			\\ 
			&=Te^{- T}\left(\sum_{\mathcal{C}_{\psi}((1-\varepsilon)T,\mathfrak{C})}\frac{V_C(f)}{\ell_\psi(C)}\varphi(h_C)+\sum_{\mathcal{C}_{\psi}(T,\mathfrak{C})\setminus\mathcal{C}_{\psi}((1-\varepsilon)T,\mathfrak{C})}\frac{V_C(f)}{\ell_\psi(C)}\varphi(h_C)\right)
			\\
			&\le Te^{- T}\left(O\left(\sum_{\mathcal{C}_{\psi}((1-\varepsilon)T,\mathfrak{C})}V_C(f)\varphi(h_C)\right)\right.
			\\
			& \qquad \qquad \qquad \qquad \qquad \qquad \qquad \left.+\sum_{\mathcal{C}_{\psi}(T,\mathfrak{C})\setminus\mathcal{C}_{\psi}((1-\varepsilon)T,\mathfrak{C})}\frac{V_C(f)}{(1-\varepsilon)T}\varphi(h_C)\right)
			\\
			&=Te^{- T}\bigg(O(\mu_{\mathfrak{C}, (1-\varepsilon)T}(f\otimes\varphi))
			\\
			& \qquad \qquad \qquad \qquad \qquad +\frac{1}{(1-\varepsilon)T}\left(\mu_{\mathfrak{C},T}(f\otimes\varphi)-\mu_{\mathfrak{C},(1-\varepsilon)T}(f\otimes\varphi)\right)\bigg)
			\\
			&=O(Te^{-\varepsilon  T}) + \frac{1}{1-\varepsilon}e^{- T}\mu_{\mathfrak{C},T}(f\otimes\varphi).
		\end{align*}
		Using \cref{thm:MuTJointEquidistribution} and $\varepsilon > 0$ being arbitrary completes the proof.	
	\end{proof}
	
	\begin{corollary}[Equidistribution of holonomies]
		\label{cor:EquidistributionOfHolonomies6}
		For any $\varphi \in \mathrm{Cl}(M)$ and cone $\mathfrak{C}\subset \limitcone$ with $\mathsf{v}\in\interior\mathfrak{C}$, we have
		\[\sum_{C \in \mathcal{C}_{\psi}(T,\mathfrak{C})} \varphi(h_C) \sim \frac{e^{T}}{T}\int_{M_\Gamma} \varphi  \, dm_\Gamma \quad \text{as } T \to \infty.\]
	\end{corollary}
	
	\begin{proof}
		Recall from \cref{subsec:TheSupportOfTheBMSMeasureAsAVectorBundle} that $\Omega \cong \mathcal{X}_\psi \times\ker\psi$, $d\BMS\bigr|_\Omega =  dm_{\mathcal{X}_\psi} \, du$ and the Lebesgue measures we use satisfy $dw = dt \, du$, where $w = t\mathsf{v} + u$. Choose $f \in \mathrm{C}_{\mathrm{c}}(\Omega)$ such that $f = \mathbbm{1}_{\mathcal{X}_\psi} \otimes b$, where $b \in \mathrm{C}_{\mathrm{c}}(\ker\psi)$ with $\int_{\ker\psi}b(u) \, du= 1$. Then $\BMS(f) = |m_{\mathcal{X}_\psi}|$ and for every $C \in \mathcal{C}_\Gamma$, 
		\begin{align*}
			V_C(f) &= \int_{\LieA \mod  \lambda(\gamma_C)} f(\Gamma ga_w M) \, dw
			\\ 
			&= \int_{\ker\psi}\int_0^{ \ell_\psi(C)} b(u) \, dt\, du = \ell_\psi(C),
		\end{align*}
		since $\lambda(\gamma_C)= \ell_\psi(C)\mathsf{v}+u$ for some $u \in \ker\psi$.
		
		Applying \cref{cor:EtaTJointEquidistribution} to this choice of $f$, we obtain 
		\[\lim_{T\to\infty} \frac{1}{e^T/T}\sum_{C \in \mathcal{C}_{\psi}(T,\mathfrak{C})} \varphi(h_C) = \int_{M_\Gamma} \varphi  \, dm_\Gamma.\]
	\end{proof}
	
	Recall that \cref{lem:PeriodicOrbits} gives a bijection between the maximal flat cylinders and periodic orbits of the translation flow $\Phi$ on $\mathcal{X}_\psi$. Let 
	\[\mathcal{G}_{\mathcal{X}_\psi}(T):= \{\Phi\text{-periodic orbits of length at most } T\}.\]
	For $C \in \mathcal{G}_{\mathcal{X}_\psi}(T)$, denote by $\mathcal{L}_C$ the length measure on $C$ and associate to $C$ the $\psi$-circumference $\ell_\psi(C)$ and holonomy $h_C$ of the corresponding maximal flat cylinder. Holonomies also jointly equidistribute with the periodic orbits in $\mathcal{X}_\psi$.
	
	\begin{corollary}[Joint equidistribution in $\mathcal{X}_\psi$] 
		For any $f\in C_c(\mathcal{X}_\psi)$ and $\varphi\in\mathrm{Cl}(M)$, we have
		\begin{align*}
			& \lim_{T\to\infty}\frac{1}{e^T}\sum_{C\in\mathcal{G}_{\mathcal{X}_\psi}(T)} \mathcal{L}_C(f) \varphi(h_C)
			=\frac{m_{\mathcal{X}_\psi}(f)}{|m_{\mathcal{X}_\psi}|}\int_{M_\Gamma}\varphi\, dm;
			\\
			& \lim_{T\to\infty} \frac{1}{e^T/T}\sum_{C\in\mathcal{G}_{\mathcal{X}_\psi}(T)} \frac{\mathcal{L}_C(f)}{\ell_\psi(C)} \varphi(h_C)
			=\frac{m_{\mathcal{X}_\psi}(f)}{|m_{\mathcal{X}_\psi}|}\int_{M_\Gamma}\varphi\, dm;
		\end{align*}
	\end{corollary}
	
	\begin{proof}
		Apply \cref{thm:MuTJointEquidistribution} with $\mathfrak{C}=\limitcone$ to the function $f \otimes b$, where $b \in \mathrm{C}_{\mathrm{c}}(\ker\psi)$ with $\int_{\ker\psi}b(u) \, du= 1$.
	\end{proof}
	
	\section{Joint equidistribution with respect to norm-like orderings}
	\label{sec:NormAlternate}
	Let $\mathsf{N}:\LieA\to\R$ be twice continuously differentiable except possibly at the origin, convex, homogeneous of degree 1 and positive on $\limitcone \setminus \{0\}$. For example, any $L^p$-norm on $\LieA$ has these properties when $1 \le p < \infty$. In this section, we give a slight modification of the arguments in the previous sections to prove joint equidistribution with respect to the ordering on $\mathcal{C}_\Gamma$ determined by $\mathsf{N}$: 
	\[\mathcal{C}_\mathsf{N}(T) := \{C \in \mathcal{C}_\Gamma : \mathsf{N}(\lambda(\gamma_C)) \le T\}.\]
	We define the $\mathsf{N}$-critical exponent as 
	\begin{equation}
		\label{eqn:NCriticalExponent}
		\delta_\mathsf{N}:=\max_{\mathsf{N}(w)=1} \psi_\Gamma(w).
	\end{equation}
	By convexity of $\mathsf{N}$ and strict concavity of $\growthindicator$, there exists a unique 
	\[\mathsf{v}\in\interior\limitcone\] 
	such that the maximum in \eqref{eqn:NCriticalExponent} is achieved at $w=\delta_\mathsf{N} \mathsf{v}$. We note that when $\mathsf{N}$ is the Euclidean norm on $\LieA$, $\mathsf{v}$ is simply the maximal growth direction of $\growthindicator$. We also let $\psi$ denote the tangent form such that
	\[\psi \text{ is tangent to } \growthindicator \text{ at } \mathsf{v}.\] 
	We note that $\mathsf{N}(\delta_\mathsf{N}\mathsf{v}) = 1 = \frac{1}{\delta_\mathsf{N}}\growthindicator(\delta_\mathsf{N}\mathsf{v})$ so by convexity of $\mathsf{N}$, we have $\frac{1}{\delta_\mathsf{N}}\psi \le \mathsf{N}$ and equality holds along the $\mathsf{v}$ direction. In particular, the differential of $\mathsf{N}$ along the $\mathsf{v}$ direction is $\frac{1}{\delta_\mathsf{N}}\psi$. Fix a cone 
	\begin{equation}
		\label{eqn:FixedCone7}
		\mathfrak{C}\subset\limitcone \textrm{ with } \mathsf{v}\in\interior\mathfrak{C}.
	\end{equation} 
	Define 
	\[\mathcal{C}_\mathsf{N}(T,\mathfrak{C}) := \{C \in \mathcal{C}_\mathsf{N}(T): \lambda(\gamma_C) \in\mathfrak{C}\}.\]
	Then we define the measures $\mu^\mathsf{N}_{\mathfrak{C},T}$ and $\eta^\mathsf{N}_{\mathfrak{C},T}$ on $\Gamma \backslash G/M\times [M]$ as follows. For $T>0$, $f \in \mathrm{C}_{\mathrm{c}}(\Gamma \backslash G/M)$, $\varphi \in \mathrm{Cl}(M)$, and $\mathfrak{C}\subset\LieA^+$, let 
	\begin{align*}
		\mu^\mathsf{N}_{\mathfrak{C},T}(f\otimes \varphi) & :=\sum_{C \in\mathcal{C}_\mathsf{N}(T,\mathfrak{C})} V_C(f)\varphi(h_C);
		\\
		\eta^\mathsf{N}_{\mathfrak{C},T}(f\otimes \varphi) & :=\sum_{C\in\mathcal{C}_\mathsf{N}(T,\mathfrak{C})}\frac{V_C(f)}{\ell_\psi(C)} \varphi(h_C).
	\end{align*}
	
	Note that for $\eta^\mathsf{N}_{\mathfrak{C},T}$, we still normalize by using the $\psi$-circumferences, or in other words, the engths of the periodic orbits in $\mathcal{X}_\psi$ because this is useful to deduce equidistribution of holonomies from joint equidistribution. We now state the joint equidistribution theorem with respect to $\mathsf{N}$.
	
	\begin{theorem}[Joint equidistribution with respect to $\mathsf{N}$]                                       \label{thm:JointEquidsitributionNVersion} 
		There exists a constant $c_\mathsf{N} > 0$ such that for any $f\in C_c(\Gamma \backslash G/M)$, $\varphi\in\mathrm{Cl}(M)$, and any cone $\mathfrak{C}\subset\limitcone$ with $\mathsf{v}\in\interior\mathfrak{C}$, we have
		\begin{align}
			\label{eqn:MuNAsymptotic}
			\lim_{T\to\infty}\frac{\mu^\mathsf{N}_{\mathfrak{C},T}(f\otimes \varphi)}{e^{\delta_\mathsf{N} T}}
			=c_\mathsf{N}\, \frac{\BMS(f)}{|m_{\mathcal{X}_\psi}|}\int_{M_\Gamma}\varphi\, dm_\Gamma;
			\\
			\label{eqn:EtaNAsymptotic}
			\lim_{T\to\infty} \frac{\eta^\mathsf{N}_{\mathfrak{C},T}T(f\otimes\varphi)}{e^{\delta_\mathsf{N} T}/\delta_\mathsf{N} T} = c_\mathsf{N} \, \frac{\BMS(f)}{|m_{\mathcal{X}_\psi}|}\int_{M_\Gamma}\varphi\, dm_\Gamma .
		\end{align}
		The constant $c_\mathsf{N}$ is given by the formula 
		\[c_\mathsf{N}:=\frac{\kappa_\mathsf{v}}{[M:M_\Gamma]}\int_{\ker\psi} e^{-I(u)-\frac{\delta_\mathsf{N}^2}{2} u^\top \left(\mathrm{Hess}(\mathsf{N})(\delta_\mathsf{N} \mathsf{v})\right) u}\, du \le 1,\]
		where $\mathrm{Hess}(\mathsf{N})$ denotes the Hessian of $\mathsf{N}$, $I$ and $\kappa_\mathsf{v}$ are as in \cref{thm:DecayofMatrixCoefficients} and $du$ is as in \eqref{eqn:BMSMeasureAndXMeasure}.
	\end{theorem}
	
	\begin{remark} 
		If the norm-like function $\mathsf{N}$ is in fact a norm induced by an inner product, then $u^\top \left(\mathrm{Hess}(\mathsf{N})(\delta_\mathsf{N} \mathsf{v})\right) u = \mathsf{N}(u)^2.$
	\end{remark}

	\begin{corollary}[Equidistribution of holonomies with respect to $\mathsf{N}$]                      
		\label{cor:EquidistributionOfHolonomies8}
		For any $\varphi\in\mathrm{Cl}(M)$ and cone $\mathfrak{C}\subset\limitcone$ with $\mathsf{v}\in\interior\mathfrak{C}$, we have
		\[\sum_{C\in\mathcal{C}_\mathsf{N}(T,\mathfrak{C})}\varphi(h_C)\sim c_\mathsf{N}\frac{e^{\delta_\mathsf{N} T}}{\delta_\mathsf{N} T} \int_{M_\Gamma} \varphi\, dm_\Gamma  \quad \textrm{as } T \to \infty.\]
	\end{corollary} 
	
	\begin{corollary} 
		For any cone $\mathfrak{C}\subset\limitcone$ with $\mathsf{v}\in\interior\mathfrak{C}$, we have as $T\to\infty$,
		\[\#\{[\gamma] \in [\primGamma]: \mathsf{N}(\lambda(\gamma)) \le T, \, \lambda(\gamma) \in \mathfrak{C} \}\sim c_\mathsf{N} \frac{e^{\delta_\mathsf{N} T}}{\delta_\mathsf{N} T}.\]
	\end{corollary}
	
	The arguments in \cref{sec:Counting,sec:ProofOfJointEquidistribution} are easily adapted to the current setting with the only points of interest being the replacement of \cref{lem:MainTermAsymptotic} and deducing the asymptotic for $\eta^\mathsf{N}_{\mathfrak{C},T}$ from the asymptotic for $\mu^\mathsf{N}_{\mathfrak{C},T}$.
	
	Let 
	\[\mathfrak{C}^\mathsf{N}_T := \{\exp(w): w \in \mathfrak{C}, \mathsf{N}(w) \le T\}.\]
	The following lemma is the appropriate replacement of the first statement in \cref{lem:MainTermAsymptotic}.
	
	\begin{lemma} 
		\label{lem:NMainTermAsymptotic} We have
		\[\lim_{T\to\infty} e^{-\delta_\mathsf{N} T}\int_{a_{t\mathsf{v}+\sqrt{t}u}\in \mathfrak{C}^\mathsf{N}_T} e^t e^{-I(u)}\, dt\, du = \int_{\ker\psi} e^{-I(u)-\frac{\delta_\mathsf{N}^2}{2} u^\top \left(\mathrm{Hess}(\mathsf{N})(\delta_\mathsf{N} \mathsf{v})\right) u}\, du.\]
	\end{lemma}
	
	\begin{proof}
		For $u \in \ker\psi$ and $T > 0$, let 
		\[R_T(u) = \{t>0: a_{t\mathsf{v}+\sqrt{t}u} \in \mathfrak{C}^\mathsf{N}_T\},\] 
		so that we have
		\[e^{-\delta_\mathsf{N} T}\int_{a_{t\mathsf{v}+\sqrt{t}u} \in \mathfrak{C}^\mathsf{N}_T}e^{t} e^{- I(u)} \, dt \, du = \int_{\ker\psi}e^{- I(u)} e^{-\delta_\mathsf{N} T}\int_{R_T(u)} e^{t}\, dt \, du.\]
		Using the second order Tylor approximation of $\mathsf{N}$ at $\delta_\mathsf{N}\mathsf{v}$ we have
		\begin{equation*}
			\begin{aligned}[b]
				&\mathsf{N}(t\mathsf{v}+\sqrt{t}u) 
				\\
				& = \frac{t}{\delta_\mathsf{N}}\mathsf{N}\left(\delta_\mathsf{N}\mathsf{v}+\frac{\delta_\mathsf{N}}{\sqrt{t}}u\right)
				\\
				& = \frac{t}{\delta_\mathsf{N}}\left(1+\frac{1}{\delta_\mathsf{N}}\psi\left(\frac{\delta_\mathsf{N}}{\sqrt{t}}u\right) + \frac{\delta_\mathsf{N}^2}{2t} u^\top \left(\mathrm{Hess}(\mathsf{N})(\delta_\mathsf{N} \mathsf{v})\right)u + O\left(\frac{\|u\|^3}{t^{3/2}}\right) \right)
				\\
				& = \frac{t}{\delta_\mathsf{N}}\left(1 + \frac{\delta_\mathsf{N}^2}{2t} u^\top \left(\mathrm{Hess}(\mathsf{N})(\delta_\mathsf{N} \mathsf{v})\right)u + O\left(\frac{\|u\|^3}{t^{3/2}}\right) \right)
			\end{aligned}
		\end{equation*}
		and hence for fixed $u$, 
		\[R_T(u) = \left[O(1), \delta_\mathsf{N}T - \frac{\delta_\mathsf{N}^2}{2} u^\top \left(\mathrm{Hess}(\mathsf{N})(\delta_\mathsf{N} \mathsf{v})\right)u + O\left(\frac{\|u\|^3}{\sqrt{T}}\right)\right].\]
		Observe that $e^{- I(u)} e^{-\delta_\mathsf{N} T}\int_{R_T(u)} e^{t} \, dt \le e^{- I(u)}$ and by the formula for $I$ from \cref{thm:DecayofMatrixCoefficients}, $e^{-I(u)} \in L^1(\ker\psi)$. Then by the Lebesgue dominated convergence theorem, 
		\begin{align*}
			\lim\limits_{T\to\infty} &e^{-\delta_\mathsf{N} T}\int_{a_{t\mathsf{v}+\sqrt{t}u} \in \mathfrak{C}_T}e^{t} e^{- I(u)} \, dt \, du 
			\\
			& = \int_{\ker\psi}e^{- I(u)}\lim\limits_{T\to\infty} e^{-\delta_\mathsf{N} T}\int_{R_T(u)} e^{t}\, dt \, du
			\\
			& = \int_{\ker\psi}e^{- I(u)}\lim\limits_{T\to\infty} e^{-\delta_\mathsf{N} T} \left(e^{\delta_\mathsf{N} T - \frac{\delta_\mathsf{N}^2}{2} u^\top \left(\mathrm{Hess}(\mathsf{N})(\delta_\mathsf{N} \mathsf{v})\right)u + O\left(\frac{\|u\|^3}{\sqrt{T}}\right)}\right)\, du
			\\
			&= \int_{\ker\psi}e^{- I(u)-\frac{\delta_\mathsf{N}^2}{2}u^\top \left(\mathrm{Hess}(\mathsf{N})(\delta_\mathsf{N} \mathsf{v})\right)u} \, du.
		\end{align*}
		
	\end{proof}
	
	\begin{proof}[Proof of \cref{thm:JointEquidsitributionNVersion}]
		After replacing \cref{lem:MainTermAsymptotic} with \cref{lem:NMainTermAsymptotic}, the remainder of the proof for the asymptotic \eqref{eqn:MuNAsymptotic} of $\mu^\mathsf{N}_{\mathfrak{C},T}$ is almost identical to the proof of \cref{thm:MuTJointEquidistribution}.
		
		We present the proof of the asymptotic \eqref{eqn:EtaNAsymptotic} for $\eta^\mathsf{N}_{\mathfrak{C},T}$. Some care must be taken since in the definition of $\eta^\mathsf{N}_{\mathfrak{C},T}$, the volume measures on maximal flat cylinders $C$ are normalized using their $\psi$-circumference $\ell_\psi(C)$, rather than $\mathsf{N}(\lambda(\gamma_C))$.
		
		Note that for $C \in \mathcal{C}_\mathsf{N}(T,\mathfrak{C})$, we have $\ell_\psi(C) = \psi(\lambda(\gamma_C) \le \mathsf{N}(\lambda(\gamma_C)) \le T$ and hence we have 
		\begin{equation}
			\label{eqn:JointEquidsitributionNVersion1}
			T^{-1}\mu_{\mathfrak{C},T}^\mathsf{N} \le \eta_{\mathfrak{C},T}^\mathsf{N}.
		\end{equation}
		
		Next, we show an upper bound for $\eta^\mathsf{N}_{\mathfrak{C},T}$. Fix $\vartheta>0$ such that $\frac{\psi(w)}{\mathsf{N}(w)}>\vartheta$ for all $w\in\limitcone\setminus\{0\}$. Fix $\varepsilon>0$ and choose a cone $\mathfrak{C}'\subset\mathfrak{C}$ such that $\mathsf{v} \in \interior \mathfrak{C}'$ and $\frac{\psi(w)}{\delta_\mathsf{N} \mathsf{N}(w)}>1-\varepsilon$ for all $w\in\mathfrak{C'}$. Then we have 
		\begin{equation}
			\label{eqn:JointEquidsitributionNVersion2}
			\begin{aligned}[b]
				&\eta_{\mathfrak{C},T}^\mathsf{N}(f\otimes\varphi)
				\\
				&\le O\left(\sum_{\mathcal{C}_\mathsf{N}((1-\varepsilon)T,\mathfrak{C})}V_C(f) \varphi(h_C)\right)
				\\  
				& \qquad \qquad +\sum_{\mathcal{C}_\mathsf{N}(T,\mathfrak{C})\setminus\left(\mathcal{C}_\mathsf{N}(T,\mathfrak{C}')\cup\mathcal{C}_\mathsf{N}((1-\varepsilon)T,\mathfrak{C})\right)}\frac{V_C(f)}{(1-\varepsilon)\vartheta T} \varphi(h_C)
				\\
				& \qquad \qquad \qquad \qquad \qquad \qquad+\sum_{\mathcal{C}(T,\mathfrak{C}')-\mathcal{C}((1-\varepsilon)T)}\frac{V_C(f)}{(1-\varepsilon)^2\delta_\mathsf{N} T} \varphi(h_C)
				\\
				& = O\left(\mu_{\mathfrak{C},(1-\varepsilon)T}^\mathsf{N}\right) + \frac{1}{(1-\varepsilon)\vartheta T}\left(\left(\mu_{\mathfrak{C},T}^\mathsf{N} - \mu_{\mathfrak{C}',T}^\mathsf{N}\right) - \left(\mu_{\mathfrak{C},(1-\varepsilon)T}^\mathsf{N} - \mu_{\mathfrak{C}',(1-\varepsilon)T}^\mathsf{N}\right)\right) 
				\\
				& \qquad \qquad \qquad \qquad +\frac{1}{(1-\varepsilon)^2\delta_\mathsf{N} T} \left(\mu_{\mathfrak{C}',T}^\mathsf{N}-\mu_{\mathfrak{C}',(1-\varepsilon)T}^\mathsf{N}\right). 
			\end{aligned}
		\end{equation}
		
		Using the asymptotic \eqref{eqn:MuNAsymptotic} for $\mu_{\mathfrak{C},T}^\mathsf{N}$ in \eqref{eqn:JointEquidsitributionNVersion1} and \eqref{eqn:JointEquidsitributionNVersion2}, we obtain
		\begin{equation*}
			c_\mathsf{N} \frac{\BMS(f)}{|m_{\mathcal{X}_\psi}|}\int_{M_\Gamma}\varphi\, dm \le \lim_{T\to\infty} \frac{\eta_{\mathfrak{C},T}^\mathsf{N}(f\otimes\varphi)}{e^{\delta_\mathsf{N} T}/\delta_\mathsf{N} T}
			\le \frac{1}{(1-\varepsilon)^2}  c_\mathsf{N} \frac{\BMS(f)}{|m_{\mathcal{X}_\psi}|}\int_{M_\Gamma}\varphi\, dm.
		\end{equation*}
		
		Since $\varepsilon>0$ is arbitrary, this completes the proof.
	\end{proof}
	
	\begin{remark}
		We remark that in the above proof of the asymptotic for $\eta_{\mathfrak{C},T}^\mathsf{N}$ for a given cone $\mathfrak{C}$ as in \eqref{eqn:FixedCone7}, it was crucial that the asymptotic for $\mu_{\mathfrak{C}',T}^\mathsf{N}$ is the same for any smaller $\mathfrak{C}' \subset \mathfrak{C}$ satisfying \eqref{eqn:FixedCone7}.
	\end{remark}
	
	\appendix
	\section{\texorpdfstring{An identity between $\kappa_\mathsf{v}$ and $I$}{An identity between local mixing constants}}
	\label{sec:Appendix}
	
	Let $\Gamma < G$ be a Zariski dense Anosov subgroup (\cref{def:AnosovSubgroup}). Let $\psi \in \LieA^*$ be a tangent form \eqref{eqn:TangentFormDefinition} tangent to  $\growthindicator$ at some $\mathsf{v}$ in the interior of the limitcone $\limitcone$ of $\Gamma$ (\cref{def:LimitCone}). We take $\mathsf{v}$ normalized so that $\psi(\mathsf{v}) = 1$. The purpose of this appendix is to prove the following identity between the constant $\kappa_\mathsf{v}$ and function $I:\ker\psi \to \R$ appearing in the local mixing theorem (\cref{thm:DecayofMatrixCoefficients} or \cref{thm:LocalMixingAppendix}) and the index $[M:M_\Gamma]$ of the holonomy group $M_\Gamma$ of $\Gamma$ (\cref{def:HolonomyGroup}). This identity was required in \cref{sec:Counting,sec:ProofOfJointEquidistribution,sec:NormAlternate} to determine the constants in the asymptotics we proved.
	
	\begin{proposition}
		\label{prop:MixingConstant}
		Using the same notation as in \cref{thm:DecayofMatrixCoefficients} or \cref{thm:LocalMixingAppendix}, we have
		\[\kappa_{\mathsf{v}} \int_{\ker\psi}e^{- I(u)} \, du = [M:M_\Gamma],\]
		where $du$ denotes the Lebesgue measure on $\ker\psi$ as in \eqref{eqn:BMSMeasureAndXMeasure}.
	\end{proposition}
	
	We will in fact deduce \cref{prop:MixingConstant} from local mixing with respect to the BMS measure $\BMS$ and an accompanying uniformity statement.
	
	\begin{theorem}[Local mixing, {\cite[Theorem 1.3]{CS23}}]
		\label{thm:LocalMixingAppendix}	
		There exists $\kappa_{\mathsf{v}} >0$ and function $I: \ker\psi \to \R$ is defined by $I(\cdot) =
		\langle \cdot, \cdot\rangle_* - \frac{\langle \cdot, \mathsf{v} \rangle_*^2}{\langle \mathsf{v}, \mathsf{v}\rangle_*}$ for some inner product $\langle \cdot, \cdot \rangle_*$ on $\LieA$ such that for any $u \in \ker\psi$ and for any $\phi_1, \phi_2 \in C_{\mathrm{c}}(\Gamma \backslash G)$, we have
		\begin{multline*}
			\lim_{t \to +\infty} t^{\frac{\rank - 1}{2}}\int_{\Gamma \backslash G} 	\phi_1(x\exp(t\mathsf{v} + \sqrt{t}u)) \phi_2(x) \, d\BMS(x) 
			\\
			= \frac{\kappa_{\mathsf{v}}e^{-I(u)}}{|m_{\mathcal{X}_\psi}|}  \sum_{Z \in \mathfrak{Z}_\Gamma} \BMS\bigr|_{Z}(\phi_1)\cdot \BMS\bigr|_{Z}(\phi_2),
		\end{multline*}
		where $\rank = \rankG$, $\mathfrak{Z}_\Gamma$ denotes the finite set of $A$-ergodic components of $\BMS$ and $m_{\mathcal{X}_\psi}$ is as in \cref{subsec:TheSupportOfTheBMSMeasureAsAVectorBundle}. 
		
		Moreover, for all $\phi_1, \phi_2 \in C_{\mathrm{c}}(\Gamma \backslash G)$, there exists a constant $D_\mathsf{v}(\phi_1,\phi_2)$ depending on $\phi_1$ and $\phi_2$ such that for all $(t,u) \in [0,\infty) \times \ker\psi$, we have
		\begin{multline*}
			\left|t^{\frac{\rank - 1}{2}} \int_{\Gamma \backslash G} \phi_1(x\exp(t\mathsf{v} + \sqrt{t}u)) \phi_2(x) \, d\BMS(x)\right| 
			\\
			\le D_\mathsf{v}(\phi_1,\phi_2) e^{-2(\|t\mathsf{v}+\sqrt{t}u\|_*\|\mathsf{v}\|_*-\langle t\mathsf{v}+\sqrt{t}u,\mathsf{v} \rangle_*)},
		\end{multline*}
		where $\|\cdot\|_*$ denotes the norm induced by $\langle \cdot,\cdot \rangle_*$.
	\end{theorem}
	
	\begin{proof}[Proof of \cref{prop:MixingConstant}]
		We refer the reader to \cref{subsec:TheSupportOfTheBMSMeasureAsAVectorBundle} for the notation used in this proof. Let $S: \mathcal{X}_\psi \to \Omega$ be a continuous section of the trivial $\ker\psi$-vector bundle $\pi_\psi: \Omega \to \mathcal{X}_\psi$. Then we have a continuous function $s:\limitset^{(2)} \times \R \to \LieA$ such that for all $(x,y,r) \in \limitset^{(2)} \times \R$,
		\[S(\Gamma(x,y,r)) = \Gamma(x,y,s(x,y,r)) \text{ and } \psi(s(x,y,r)) = r.\]
		Then we have a homeomorphism $\mathcal{X}_\psi \times\ker\psi \cong \Omega$ given by
		\[(\Gamma(x,y,r),u') \mapsto \Gamma(x,y,s(x,y,r)+u').\]
		Observe that for $T > 0$ and $u \in \ker\psi$, the $\LieA$-coordinate of 
		\[\Gamma(x,y,s(x,y,r)+u')\exp(t\mathsf{v} + \sqrt{t}u) \in \Gamma \backslash(\limitset^{(2)}\times\LieA)\] 
		is given by
		\begin{equation}
			\label{eqn:MixingConstant1}
			\begin{aligned}[b]
				s(x,y,r)+u'+t\mathsf{v} + \sqrt{t}u = s(x,y,r+t) +(u'+ \sqrt{t}u) - \hat{s}(x,y,r,t),
			\end{aligned}
		\end{equation}
		where $\hat{s}(x,y,r,t) : = s(x,y,r+t)-s(x,y,r)-t\mathsf{v} \in \ker\psi$.
		
		Fix a compactly supported continuous function $b:\ker\psi \to [0,\infty)$ such that $\int_{\ker\psi}b(u') \, du' = 1$, where $du'$ is the Lebesgue measure on $\ker\psi$ satisfying $dm_{\mathcal{X}_\psi} \, du' = d\BMS\bigr|_\Omega$. Let $\phi \in C_\mathrm{c}(\tilde{\Omega})$ be the function given by the $M$-invariant lift of the function $\mathbbm{1}_{X_\psi} \otimes b$ on $\mathcal{X}_\psi \times \ker\psi \cong \Omega$. Note that
		\begin{equation}
			\label{eqn:MixingConstant2}
			\BMS\bigr|_{Z}(\phi) = \frac{1}{\#\mathfrak{Z}_\Gamma}\BMS(\phi) = \frac{|m_{\mathcal{X}_\psi}|}{[M:M_\Gamma]},
		\end{equation}
		where the last equality uses $\#\mathfrak{Z}_\Gamma = [M:M_\Gamma]$ \cite[Theorem 1.1]{LO20a}. Using \cref{thm:LocalMixingAppendix} applied to the functions $\phi_1 = \phi_2 = \phi$ and using \eqref{eqn:MixingConstant2}, we have
		\begin{multline}
			\label{eqn:MixingConstant3}
			\lim_{t \to +\infty} t^{\frac{\rank - 1}{2}}\int_{\mathcal{X}_\psi \times\ker\psi} b((u'+ \sqrt{t}u) - \hat{s}(x,y,r,t))b(u') \, dm_{\mathcal{X}_\psi}(\Gamma(x,y,r)) \, du'
			\\
			=\frac{\kappa_{\mathsf{v}}e^{-I(u)}}{|m_{\mathcal{X}_\psi}|}  \sum_{Z \in \mathfrak{Z}_\Gamma} \left(\BMS\bigr|_{Z}(\phi)\right)^2 = \frac{\kappa_{\mathsf{v}}e^{-I(u)}|m_{\mathcal{X}_\psi}|}{[M:M_\Gamma]}.
		\end{multline}
		
		Rearranging and integrating \eqref{eqn:MixingConstant3} with respect to $u \in \ker\psi$, we obtain
		\begin{equation}
			\label{eqn:MixingConstant4}
			\begin{aligned}[b]
				& \kappa_{\mathsf{v}} \int_{\ker\psi}e^{- I(u)} \, du 
				\\
				& = \frac{[M:M_\Gamma]}{|m_{\mathcal{X}_\psi}|}\int_{\ker\psi}\lim_{t \to +\infty} t^{\frac{\rank - 1}{2}}\int_{\ker\psi}\int_{\mathcal{X}_\psi} b((u'+ \sqrt{t}u) - \hat{s}(x,y,r,t))
				\\
				& \qquad \qquad \qquad \qquad \qquad \qquad \qquad \qquad \qquad b(u') \, dm_{\mathcal{X}_\psi}(\Gamma(x,y,r)) \, du' \, du
			\end{aligned}
		\end{equation}
		Next, we explain how to move the limit in \eqref{eqn:MixingConstant4} outside of the integral. By \cref{thm:LocalMixingAppendix}, there exists a constant $D := D_\mathsf{v}(\phi,\phi)>0$ such that for all $(t,u) \in [1,\infty)\times \ker\psi$, we have
		\begin{multline}
			\label{eqn:MixingConstant5}  
			t^{\frac{\rank - 1}{2}}\int_{\ker\psi}\int_{\mathcal{X}_\psi} b((u'+ \sqrt{t}u) - \hat{s}(x,y,r,t)) b(u') \, dm_{\mathcal{X}_\psi}(\Gamma(x,y,r)) \, du' \, du 
			\\
			\le De^{-2(\|t\mathsf{v}+\sqrt{t}u\|_*\|\mathsf{v}\|_*-\langle t\mathsf{v}+\sqrt{t}u,\mathsf{v} \rangle_*)}.
		\end{multline}
		Note that for $t \ge 1$,
		\begin{equation}
			\label{eqn:MixingConstant6}     
			\begin{aligned}[b]
				\|t\mathsf{v}+\sqrt{t}u\|_*\|\mathsf{v}\|_*-\langle t\mathsf{v}+\sqrt{t}u,\mathsf{v} \rangle_* & =  \frac{\|u\|_*^2\|\mathsf{v}\|_*^2-\langle u,\mathsf{v} \rangle_*^2}{\|\mathsf{v}+u/\sqrt{t}\|_*\|\mathsf{v}\|_*+\langle \mathsf{v}+u/\sqrt{t},\mathsf{v} \rangle_*}
				\\
				& \ge \frac{\|u\|_*^2\|\mathsf{v}\|_*^2-\langle u,\mathsf{v} \rangle_*^2}{2\|\mathsf{v}\|_*^2+(\|u\|_*\|\mathsf{v}\|_*+\langle u,\mathsf{v} \rangle_*)}.
			\end{aligned}
		\end{equation}
		Define the function $f:\ker\psi \to [0,\infty)$ by
		\[f(u) = \frac{\|u\|_*^2\|\mathsf{v}\|_*^2-\langle u,\mathsf{v} \rangle_*^2}{2\|\mathsf{v}\|_*^2+\|u\|_*\|\mathsf{v}\|_*+\langle u,\mathsf{v} \rangle_*}.\]
		Then combining \eqref{eqn:MixingConstant5} and \eqref{eqn:MixingConstant6}, for all $(t,u) \in [1,\infty)\times \ker\psi$, we have 
		\begin{multline}
			\label{eqn:MixingConstant7}  
			t^{\frac{\rank - 1}{2}}\int_{\ker\psi}\int_{\mathcal{X}_\psi} b((u'+ \sqrt{t}u) - \hat{s}(x,y,r,t)) b(u') \, dm_{\mathcal{X}_\psi}(\Gamma(x,y,r)) \, du' \, du 
			\\
			\le De^{-2f(u)}.
		\end{multline}
		Moreover, we observe that $e^{-2f(u)} \in L^1(\ker\psi)$ since for all $u \in \ker\psi \setminus \{0\}$, we have
		\[\lim\limits_{\lambda \to \infty}\frac{f(\lambda u)}{\lambda} =\frac{\|u\|_*^2\|\mathsf{v}\|_*^2-\langle u,\mathsf{v} \rangle_*^2}{\|u\|_*\|\mathsf{v}\|_*+\langle u,\mathsf{v} \rangle_*} > 0.\]
		Then by \eqref{eqn:MixingConstant7}, we can apply the Lebesgue dominated convergence theorem in \eqref{eqn:MixingConstant5}, followed by the change of variables $\tilde{u} = \sqrt{t}u$, Fubini's theorem and integrating using $\int_{\ker\psi}b(u') \, du' = 1$ twice to conclude that
		\begin{equation*}
			\begin{aligned}[b]
				& \kappa_{\mathsf{v}} \int_{\ker\psi}e^{- I(u)} \, du 
				\\
				& = \frac{[M:M_\Gamma]}{|m_{\mathcal{X}_\psi}|}\lim_{t \to +\infty} t^{\frac{\rank - 1}{2}}\int_{\ker\psi}\int_{\ker\psi}\int_{\mathcal{X}_\psi} b((u'+ \sqrt{t}u) - \hat{s}(x,y,r,t))
				\\
				& \qquad \qquad \qquad \qquad \qquad \qquad \qquad \qquad \qquad b(u') \, dm_{\mathcal{X}_\psi}(\Gamma(x,y,r)) \, du' \, du,
				\\
				& = \frac{[M:M_\Gamma]}{|m_{\mathcal{X}_\psi}|}\lim_{t \to +\infty} \int_{\ker\psi}\int_{\mathcal{X}_\psi}\int_{\ker\psi} b((u'+ \tilde{u}) - \hat{s}(x,y,r,t))
				\\
				& \qquad \qquad \qquad \qquad \qquad \qquad \qquad \qquad \qquad b(u') \, d\tilde{u} \, dm_{\mathcal{X}_\psi}(\Gamma(x,y,r)) \, du' 
				\\
				& = \frac{[M:M_\Gamma]}{|m_{\mathcal{X}_\psi}|}\lim_{t \to +\infty} \int_{\ker\psi}\int_{\mathcal{X}_\psi} b(u')  \, dm_{\mathcal{X}_\psi}(\Gamma(x,y,r)) \, du' 
				\\
				& = [M:M_\Gamma].
			\end{aligned}
		\end{equation*}
	\end{proof}
	
	\nocite{*}
	\renewcommand*{\bibfont}{\footnotesize}
	\bibliographystyle{alpha_name-year-title}
	\bibliography{References}
\end{document}